%% file: dualizability.tex
\documentclass[a4paper, reqno, 10pt]{amsart}

\input{prestuff}

\makeatletter
\def\paragraph{\@startsection{paragraph}{4}%
  \z@\z@{-\fontdimen2\font}%
  {\normalfont\bfseries}}
\makeatother

\newtheorem{theorem}{Theorem}[section]
\newtheorem*{theorem*}{Theorem}
\newtheorem{prop}[theorem]{Proposition}
\newtheorem*{prop*}{Proposition}
\newtheorem{cor}[theorem]{Corollary}
\newtheorem{lemma}[theorem]{Lemma}

\theoremstyle{definition}
\newtheorem{defn}[theorem]{Definition}

\newtheorem{definition}[theorem]{Definition}
\newtheorem*{defn*}{Definition}
\newtheorem*{warning*}{Warning}
\newtheorem{ex}[theorem]{Example}
\newtheorem{example}[theorem]{Example}

\newtheorem{remark}[theorem]{Remark}

\newcommand{\mpim}{\ensuremath{
  \mathchoice{\includegraphics[height=2.5ex]{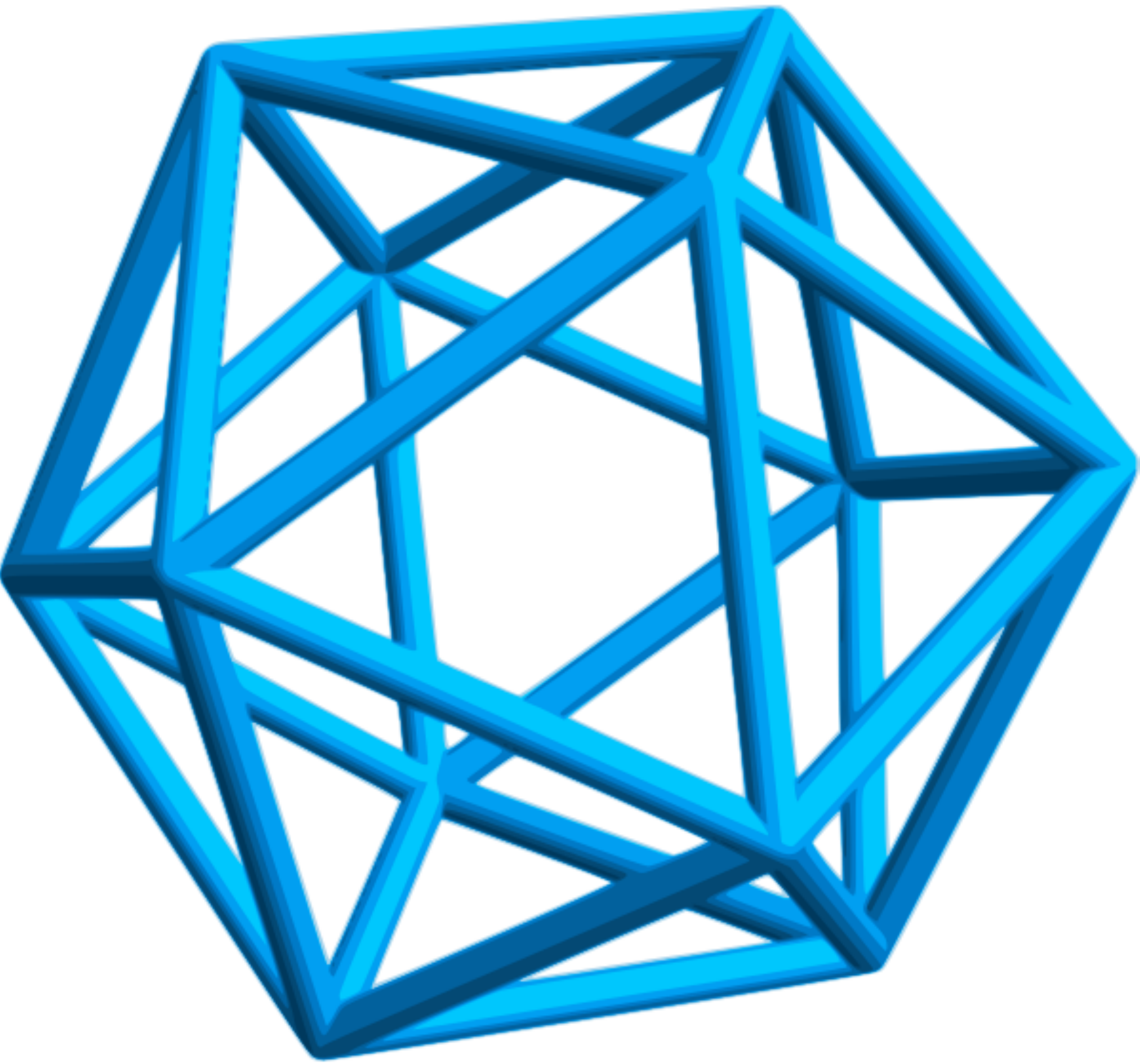}}
    {\includegraphics[height=2.5ex]{MPIM0}}
    {\includegraphics[height=2.5ex]{MPIM0}}
    {\includegraphics[height=2.5ex]{MPIM0}}
}}

\title{Duals and adjoints in higher Morita categories}

\author[O.~Gwilliam]{Owen Gwilliam}
\address{Max Planck Insitute for Mathematics, 53111 Bonn, Germany}
\email{gwilliam@mpim-bonn.mpg.de}
\author[C.~Scheimbauer]{Claudia Scheimbauer}
\address{Mathematical Insitute, University of Oxford, OX2 6GG Oxford, UK}
\email{scheimbauer@maths.ox.ac.uk}

\begin{document}

\begin{abstract}
We study duals for objects and adjoints for $k$-morphisms in $\Alg_n(\S)$, an $(\infty,n+N)$-category that models a higher Morita category for $E_n$ algebra objects in a symmetric monoidal $(\infty,N)$-category $\S$.
Our model of $\Alg_n(\S)$ uses the geometrically convenient framework of factorization algebras.
The main result is that $\Alg_n(\S)$ is fully $n$-dualizable, verifying a conjecture of Lurie. Moreover, we unpack the consequences for a natural class of fully extended topological field theories and explore $(n+1)$-dualizability. 
\end{abstract}

\maketitle

\tableofcontents

\section{Introduction}

Morita theory plays a key role in modern algebra,
emphasizing an algebra's category of modules and thus leading to a focus on categorical structures.\footnote{The literature since Morita's initiating paper \cite{Morita} is vast. We recommend \cite{Schwede, Toen} as starting places for contemporary activity.}
It is convenient to package this perspective into a 2-category
whose objects are associative algebras, 
whose 1-morphisms are bimodules,
and whose 2-morphisms are maps of bimodules \cite{Benabou}.
Many structural results are naturally articulated in this 2-category.\footnote{In this introduction and elsewhere, we will use the term $n$-category informally, meaning the appropriate weakened versions.}

There is a natural generalization of the notion of dual for a vector space to any monoidal category.\footnote{According to the survey \cite{BeckerGottlieb}, this notion goes back to Dold and Puppe~\cite{DoldPuppe}.}
As algebras can be tensored, the Morita 2-category is symmetric monoidal,
and so one can ask which algebras admit duals.
One can quickly check that every algebra $A$ has a dual given by the opposite algebra $A^{\opp}$.
The evaluation morphism $A \otimes A^{\opp} \to \kk$ is given by $A$ viewed as a bimodule,
and the coevalution morphism $\kk \to A^{\opp} \otimes A$ is given by the unit element.
(We will discuss in a moment the connection with topological field theory in the sense of Atiyah and Segal.)

Similarly, in any 2-category, one can ask which 1-morphisms admit left or right adjoints, generalizing the notion of adjoints for functors between categories.
For an $(A,B)$-bimodule $M$ viewed as a 1-morphism $A \to B$ in the Morita 2-category,
a straightforward algebraic argument implies that
$M$ admits a right adjoint $N: B \to A$ if and only if $M$ is finitely-generated and projective over $A$.
In that case, $N \cong \Hom_A(M,A)$.
(See, e.g., \cite{DoldPuppe}.)
The dual statement identifies a left adjoint as $\Hom_B(M,B)$,
under the hypothesis that $M$ is finitely-generated and projective over~$B$.

Following Lurie \cite{LurieTFT}, one can ask to identify the fully dualizable sub-2-category,
which consists of objects with duals such that the (co)evaluation morphisms admit right and left adjoints and all 1-morphisms between these objects that also admit right and left adjoints.
In light of our observations above, one finds that these objects are the separable $\kk$-algebras that are finitely-generated and projective over $\kk$.\footnote{To unpack the argument a little, note that if we apply the adjointability conditions on 1-morphisms to the (co)evaluation maps, we need $A$ to be finitely-generated and projective as a module over $\kk$ as well as finitely-generated and projective as an $A \otimes A^{\opp}$-module, which is the condition of separability.}
It is striking (though perhaps not surprising) that asking these purely categorical questions identifies important algebraic notions, such as projectivity or separability.

In this paper we will address the analogous questions in higher algebra.
We will replace associative $\kk$-algebras by algebras over the little $n$-disks operad $E_n$ in any sufficiently well-behaved symmetric monoidal $(\infty,N)$-category $\S$.
Such $E_n$ algebras can be understood intuitively as objects in $\S$ equipped with $n$ compatible associative multiplications;
that is, they are $n$-dimensional generalizations of associative algebras.
Such algebras lead to an $(\infty,n+N)$-category $\Alg_n(\S)$ in which, loosely speaking, the objects are $n$-dimensional algebras, the 1-morphisms are $(n-1)$-dimensional algebras in bimodules for the $n$-dimensional algebras, \dots, the $k$-morphisms are $(n-k)$-dimensional algebras in bimodules for the $k-1$-morphisms, and so on.

This broad, flexible framework encompasses situations of interest to contemporary mathematics.
As we explain in greater detail at the end of this introduction, 
interesting examples of such higher categories arise when choosing a well-behaved 2-category $\S$ of $\kk$-linear categories.
In that case, $\Alg_1(\S)$ is a 3-category consisting of monoidal linear categories, bimodule linear categories between them, functors, and natural transformations.
Similarly, $\Alg_2(\S)$ is a 4-category whose objects are braided monoidal linear categories.
Hence our results specialize to situations of importance for, e.g., representation theory.

Within the general framework, we examine when objects admit duals and when $k$-morphisms admit right or left adjoints, in the sense of \cite{LurieTFT}.
Our main result is the $n$-dimensional generalization of the fact that every $\kk$-algebra has a dual given by its opposite.
It verifies a conjecture of Lurie (see Claim 4.1.14 of~\cite{LurieTFT}).
We also give a partial answer to the question of the $(n+1)$-dualizable objects in $\Alg_n(\S)$,
which involves some subtleties which we address later in the introduction.

We mention that despite how natural it is to explore the existence of duals and adjoints for higher morphisms, from a higher categorical perspective,
there is paucity of work in this direction,
particularly for $n > 2$.
Indeed, the only other situation known to us where full $n$-dualizability has been established for large $n$ is by Haugseng, for higher span categories~\cite{HaugsengSpans}. 
This work can be seen as some first fruits in the exploration of these issues.

The rest of the introduction is devoted to fleshing out the the overview just given.
It ends by indicating applications and future directions of further research.

\subsection{Generalizing the Morita 2-category}

The Morita 2-category is a first example of a hierarchy of higher Morita categories that we expect to play an important role in higher algebra.

First, recall that types of algebras --- such as associative or commutative --- are typically governed by operads,
and that there is a sequence of operads 
\[
E_1 \hookrightarrow E_2 \hookrightarrow \cdots \hookrightarrow E_n \hookrightarrow \cdots \hookrightarrow E_\infty
\]
interpolating between the associative and commutative operads,
which correspond to $E_1$ and $E_\infty$, respectively.\footnote{To be more careful, these typically denote operads in a category of topological spaces. 
By viewing a set as a space, one can view the usual associative operad $Ass$ as an operad in spaces,
in which case one can quickly see that $E_1$ and $Ass$ are homotopy equivalent.
Moreover, the commutative operad $Comm$ and the $E_\infty$ operad are weakly equivalent as well. In non-positive characteristic, their algebras agree. Thus, in traditional, non-derived algebra, one does not see the difference.
However, in positive characteristic the situation is more subtle: the homology of free algebras differ in the two cases, leading to the Dyer-Lashof operations.
In the setting of homotopy theory, it is found that $E_\infty$ algebras show up more naturally,
and hence are viewed there as the correct notion of ``commutative algebra.''}
A concrete model for the $E_n$ operad is given by the little $n$-disks operad,
whose $k$-ary operations are parametrized by the space of $k$ disjoint Euclidean $n$-balls embedded into an $n$-ball.
In this sense, an $E_n$ algebra is literally an algebra whose multiplication is parametrized by configurations in an $n$-dimensional space.
Alternatively, the $E_n$ operad is the $n$-fold Boardman-Vogt tensor power of $E_1$,
and hence captures the notion of having $n$ directions of compatible multiplications;
this Dunn-Lurie additivity allows one to approach the theory in a highly algebraic manner.\footnote{There is an extensive literature on $E_n$ algebras, with a lot of activity in recent years. Boardman-Vogt \cite{BV} initiated the subject, but see \cite{May, Dunn} as well. We refer to Chapter 5 of \cite{LurieHA} for a recent extensive discussion of these notions.}
 
The intuitive geometry of the situation suggests that one views the Morita 2-category as follows.
An $E_1$ algebra lives along a real line, 
with the inclusion of disjoint intervals into bigger intervals parametrizing the multiplications.
A module should then live on a boundary point of the line,
as a module involves actions from one side.
Explicitly, we view interior intervals as labeled by the algebra but intervals containing the boundary as labeled by the module.
Geometrically, a bimodule would live on a point that divides a line into two regions,
labeled by the two algebras acting on it.
See Section \ref{sec: fact Alg_1} for a detailed visualization and discussion.

One generalizes to the $n$-dimensional setting by labeling $n$-dimensional disks by $E_n$ algebras
and by labeling a linear hypersurface with an $E_{n-1}$-algebra that has a compatible action of the $E_n$ algebras living in the regions on either side of the hypersurface.
One can imagine linear subspaces of every dimension, 
down to points.
See Section \ref{sec: fact Alg_n} for a detailed visualization of the two-dimensional case and discussion of the general case.

Hence, one can expect that for each positive integer $n$,
there is some kind of $(n+1)$-category whose objects consist of $E_n$ algebras,
whose 1-morphisms consist of $E_{n-1}$-algebras in bimodules between $E_n$ algebras, 
\dots, whose $k$-morphisms consist of $E_{n-k}$ algebras in bimodules between the $E_{n-k+1}$ algebras, and so on for $k \leq n$.
At $k = n$, which geometrically corresponds to points, 
we have objects that are bimodules for $E_1$ algebras.
We work here with an $(\infty,1)$-category of bimodules,
so that between two bimodules, one has a space of morphisms.
(There is a somewhat subtle issue about whether or not to work with $E_0$ algebras in bimodules,
which we discuss in Section \ref{sec pointed comment}.)
When $n = 1$, one recovers a version of the usual Morita 2-category.\footnote{This intuitive idea for how to generalize the Morita 2-category is explained clearly by Lurie in Section 4.1 of \cite{LurieTFT}, but it undoubtedly has a longer history with which we are unfortunately unfamiliar. The basic idea is certainly apparent in the Swiss cheese operad introduced by Voronov \cite{Voronov}, and we have heard Kevin Walker discuss equivalent ideas in early talks on blob homology.}

In this paper we use a rather direct realization of the geometric picture,
developed in \cite{CSThesis, CalaqueScheimbauer, JFS}.
In Section~\ref{sec: rec on morita}, we review that framework before using it to prove our main results.
In brief, one works with factorization algebras that are constructible with respect to the stratified spaces appearing in the intuitive sketch just given.
Lurie has shown that locally constant factorization algebras on $\RR^n$ form an $(\infty,1)$-category equivalent to that of $E_n$ algebras (see Theorem 5.4.5.9),
so that the $k$-morphisms ought to match the idea just sketched,
as they can be identified with $E_{n-k}$ algebras.
An alternative approach, not explored in this paper, would be to use the approach to higher Morita categories developed by Haugseng in \cite{HaugsengEn},
which is completely algebraic and combinatorial in nature and provides a different set of intuitions. For instance, there are useful dualizability and adjointability results in Section 4.6 of \cite{LurieHA} that provide partial answers in this latter algebraic setting.
In Setion \ref{sec one d dictionary} and Remark \ref{rem two d dictionary}, 
we sketch a dictionary between the two approaches.

\subsection{Dualizability and our main result}

Let $\Alg_n(\S)$ denote the factorization model of the Morita $(\infty,n+N)$-category of $n$-dimensional algebras with values in a symmetric monoidal $(\infty,N)$-category $\S$ whose tensor product $\otimes$ preserves sifted colimits separately in each variable.
We note that there is a canonical truncation to an $(\infty,n)$-category $\tau_{(\infty,n)} \Alg_n(\S)$,
which loosely speaking only keeps invertible bimodule maps;
it is the analogue of the 1-categorical version of the usual Morita 2-category.
(Strictly speaking, it is a $(2,1)$-category.)
We devote Section~\ref{sec: rec on morita} to recalling the definition.

Following Lurie \cite{LurieTFT}, it is interesting to ask about duals for objects and adjoints for $k$-morphisms in any a symmetric monoidal $(\infty,n)$-category $\C^\otimes$. 
Lurie uses the phrase $\C$ {\em has duals} to mean that 
\begin{itemize}
\item in the underlying homotopy category, every object {\em has a dual} in the usual sense, and
\item for all $0 < k < n$, every $k$-morphism {\em admits adjoints}, in the sense that viewing $f: X \to Y$ as a 1-morphism in the homotopy 2-category of $\Map_\C(X,Y)$, it admits both a left and a right adjoint in the usual 2-categorical sense.
\end{itemize}
This notion provides a systematic generalization of the 1-categorical notion of dualizability.
In Section~\ref{sec: dualizability}, in greater detail, we review this notion, which we prefer to call ``full $n$-dualizability'' to emphasize the dependence on $n$.
It is certainly natural, simply from a categorical perspective, to ask when $k$-morphisms admit adjoints in any higher category.

Our main result is Theorem \ref{thm main theorem},
which shows that $\Alg_n(\S)$ is ``fully $n$-dualizable''
({\it cf.} Claim 4.1.14 of \cite{LurieTFT}, where it is stated but unproven).

\begin{theorem*}
The symmetric monoidal $(\infty,n)$-category $\tau_{(\infty,n)} \Alg_n(\S)$ underlying the factorization higher Morita category $\Alg_n(\S)$ is fully $n$-dualizable, namely
\begin{itemize}
\item every object has a dual, and
\item for $1\leq k<n$, every $k$-morphism has a left and a right adjoint.
\end{itemize}
\end{theorem*}

Our arguments are quite geometric in nature, 
exploiting the fact that factorization algebras live on manifolds and hence admit manipulations deriving from that underlying geometry.
Section~\ref{sec: 2d case} is devoted to explaining the argument --- and visualizing it in great detail --- in two dimensions,
so that the intuition guiding the general argument in Section~\ref{sec: arbitrary dim} is as clear as possible.
It is much harder to see many of these manipulations directly in algebraic terms,
and so we feel it makes a good case for the convenience of approaching higher algebra using these factorization models.

\subsection{Relationship with topological field theory}\label{sec intro connection TFT}

The original definition of topological field theory (TFT) in the style of Atiyah and Segal calls for a category of vector spaces as its target \cite{Atiyah}.
If we extend the cobordisms by allowing manifolds with corners (and hence codimension 2 manifolds),
it is natural to take the Morita 2-category as the target,
since the endomorphisms of the unit object $\kk$ is given by the category of vector spaces.\footnote{This idea is commonplace now, but we do not know its original source.}
One can ask to continue to higher codimension, down to points, leading to {\em fully extended} TFTs,
as described by Baez and Dolan~\cite{BaezDolan}.

As we will now quickly sketch,
traditional notions in algebra naturally appear in this setting,
due to the deep relationship between duality (and its higher generalizations) and fully extended TFTs,
as articulated by the Cobordism Hypothesis~\cite{BaezDolan,LurieTFT}.

For instance, every algebra $A$ determines an oriented (equivalently, framed) one-dimensional field theory,
where the positively oriented point is assigned $A$ and the negatively oriented point is assigned its dual $A^{\opp}$.
This theory assigns $A \otimes_{A \otimes A^{\opp}} A = A/[A,A]$ to the circle,
and hence naturally rediscovers the zeroth Hochschild homology, the universal home of trace maps out of $A$.
(If one works with dg algebras, then one recovers the whole Hochschild chain complex.)
In brief, the dualizability condition to be a one-dimensional field theory highlights a striking structural feature of algebras --- the existence of opposites --- and hones in upon the universal trace.

As an even more interesting example,
consider the case of two-dimensional field theories with values in this Morita 2-category,
by which we mean functors of symmetric monoidal 2-categories out of a 2-category of framed bordisms.
Here one wishes to assign an algebra to each framed 0-dimensional manifold (i.e., a finite set of  points), a bimodule to each framed 1-manifold with boundary, and a bimodule map to each framed 2-manifold with corners.
Then such a theory must assign a separable $\kk$-algebra $A$ that is finitely-generated and projective over $\kk$,
and {\it vice versa} so that such an algebra $A$ determines such a TFT.\footnote{The nature of the bordisms also has interesting algebraic consequences.
Schommer-Pries \cite{SPThesis} showed that there is an equivalence of groupoids 
\[
{\rm Fun}^\otimes(\Bord^{\rm or}_2, \Alg_2) \simeq {\rm SepSymFrob}^\sim
\]
between fully extended oriented two-dimensional field theories with values in this Morita 2-category and the separable symmetric Frobenius algebras.
For algebras over a perfect field (e.g., characteristic zero),
these TFTs correspond to finite-dimensional, semisimple algebras with non-degenerate pairing. For the framed result, see \cite{Piotr}.}
Moreover, since the circle admits countably many distinct 2-framings,
one obtains a countable collection of invariants of $A$,
including $A/[A,A]$ and the center $Z(A)$.
(In the dg setting, $Z(A)$ generalizes to the whole Hochschild cochain complex~$Hoch^*(A,A)$.)

These conditions are identical to those that appeared when asking about duals for objects and adjoints for 1-morphisms.
This identification is, of course, the Cobordism Hypothesis applied to the situation of the Morita 2-category as the target, see \cite[Remark 4.1.27]{LurieTFT}

In light of these results, it is natural to ask about higher-dimensional analogues,
which would provide interesting examples of higher-dimensional fully extended TFTS.
In particular, one would need higher categories generalizing the Morita 2-category,
and then one could examine what algebraic properties appear in identifying dualizable objects in these settings.
This paper is a step in this direction,
using a model for the higher Morita categories due to the second author.
Our central result is a direct generalization of the fact that every algebra is dualizable (and hence determines a one-dimensional TFT):
the $n$-dimensional analogue of an algebra (for us, an $E_n$ algebra) is $n$-dualizable,
and more generally, the $(\infty,n)$-category of such $n$-dimensional algebras is fully $n$-dualizable.

We emphasize that we do {\em not} use the Cobordism Hypothesis in this paper,
but directly analyze dualizability, which is a well-posed notion for symmetric monoidal $(\infty,n)$-categories. 
Our result implies that every $E_n$ algebra $\Rr$ is $n$-dualizable. Hence, if one accepts the Cobordism Hypothesis, it determines a fully extended, framed $n$-dimensional field theory
$$\Bord_n^{\mathit{fr}}\xrightarrow{T_\Rr} \Alg_n(\S).$$
Indeed, the second author \cite{CSThesis, CalaqueScheimbauer} explicitly exhibited this class of theories using factorization homology,
so that in conjunction with our results here, 
one obtains another demonstration of the Cobordism Hypothesis in action.

Finally, our result implies more:  
any 1-morphism is ``$(n-1)$-times left and right adjunctible'' and hence, assuming the Cobordism Hypothesis with singularities, determines a defect theory
$$\Bord_n^{\mathit{fr}, \mathit{def}} \xrightarrow{T_\Rr} \Alg_n(\S).$$
If the 1-morphism is given by the geometric $(\Rr,\Ss)$-bimodule $\Aa$, 
this defect theory can be interpreted as a {\em relative} field theory
\[
\begin{tikzpicture}
  \path node (B) {} node[anchor=east] {$\Bord_n^{\mathit{fr}}$} +(2,0) node (C) {} node[anchor=west] {$\Alg_n(\S)$\, .};
  \draw[arrow] (B) .. controls +(1,-.5) and +(-1,-.5) .. coordinate (t)  (C) ;
  \draw[arrow] (B) .. controls +(1,.5) and +(-1,.5) .. coordinate (s)  (C);
  \draw[twoarrowlonger] (s) node[above] {$T_{\Rr}$} -- node[anchor=west] {$Z_{\Aa}$}  (t) node[below] {$T_{\Ss}$};
\end{tikzpicture}
\]
in the sense of Freed-Teleman \cite{FreedTeleman}.\footnote{In \cite{FreedTeleman}, they assume that $\Rr$ and $\Ss$ are even $(n+1)$-dualizable, which we drop in this interpretation.}

\subsection{Pointed comment}\label{sec pointed comment}

Ultimately, one would like to study higher dualizability of the $(\infty,n+N)$-category $\Alg_n(\S)$, and in particular, which objects allow for $(n+1)$-dualizability. This question returns us to a subtle issue in defining the higher Morita categories.

Let us indicate the issue in the classical setting:
the question is whether to work with ``pointed'' bimodules.
Let $\kk$ denote the base commutative ring.
It is common to work with {\em unital} associative algebras, 
which equips each algebra $A$ with a distinguished map $k \to A$.
Hence, there is a forgetful functor from the 1-category of unital $\kk$-algebras and $\kk$-algebra maps down to the category of ``pointed'' $\kk$-modules,
meaning the slice category $\Mod(k)_{k/}$ consisting of $\kk$-modules $M$ equipped with a map $k \to M$.
In constructing a Morita 2-category,
one could similarly choose to work with ``pointed'' bimodules,
so that a 1-morphism from a unital algebra $A$ to a unital algebra $B$ is an $(A,B)$-bimodule $M$ equipped with a pointing $m_0: k \to M$
and a 2-morphism is a map of $(A,B)$-bimodules that preserves the pointings.
In other words, we work with $E_0$ algebras in bimodules.

Working with pointed bimodules is not the classical approach,
but it is reasonable from the heuristic idea motivating the higher Morita category:
we want a $k$-morphism to be an $E_{n-k}$ algebra in bimodules, 
and so an $n$-morphism ought to be an $E_0$ algebra.

From a TFT prespective, the pointings are not unnatural. 
For example, the 2d TFT built from an algebra $A$ naturally equips the relevant bimodules with pointings, 
since the (co)evaluation morphisms are given by $A$ viewed as a bimodule in different ways. Its pointing is given by the unit of the algebra.
More generally, our result shows that the dualizability data for an $E_n$ algebra also enjoys this property.

Pointed bimodules are also natural when taking the factorization algebra approach to higher algebra.
In that setting, the pointing (and unitality of $E_k$ algebras generally) comes from the fact that the empty set is an open subset of any subset,
and hence determines a pointing.
Hence the factorization higher Morita $(\infty,n+1)$-category $\Alg_n(\S)$ works with pointed bimodules.\footnote{One can modify the construction of \cite{CSThesis,CalaqueScheimbauer} to work with unpointed bimodules as $n$-morphisms, but we do not pursue that variant here.}

We remark that for the main theorem about $n$-dualizability,
the choice of (un)point\-ed bimodules is irrelevant.
As will be seen in the arguments, we construct the right and left adjoints of $k$-morphisms by explicit constructions with factorization algebras;
indeed, the adjoints are represented by factorization algebras.
But a factorization algebra is naturally pointed in the sense that it assigns a pointed object to every open set,
due to the structure map determined by the inclusion of the empty set.

One can strengthen this statement as follows. We expect there should be a symmetric monoidal functor of $(\infty,n)$-categories from the higher Morita category with pointings to a ``depointed'' higher category, where bimodules of the $n$-morphisms are not pointed.\footnote{A model for the latter $n$-category is the algebraic higher Morita categories of Haugseng \cite{HaugsengEn} (or the variant of the previous footnote). However, at present there is no construction of such a symmetric monoidal functor in the literature.} 
This functor $U:\Alg_n(\S)\to \Alg_n^{unptd}(\S)$ simply forgets the pointing on the bimodules.
It should be an equivalence on the underlying $(\infty,n-1)$-categories, 
since we always work with unital $E_d$ (or factorization) algebras,
and the lower morphisms have such a structure.
We have shown that every $k$-morphism, for $0\leq k<n$, admits adjoints with natural pointings. 
Thus, it follows that $\Alg_n^{unptd}(\S)$ is fully $n$-dualizable as well.
In particular, every fully extended $n$-dimensional TFT valued in the unpointed version factors through $U$:
$$\begin{tikzcd}
\Bord_n^{\mathit{fr}} \arrow{rr} \arrow[dashed]{dr}&& \Alg_n^{unptd}(\S)\\
&\Alg_n(\S) \,.\arrow[swap]{ru}{U}
\end{tikzcd}$$
In short, the pointings do not affect the impact for $n$-dimensional TFTs of our main result.

On the other hand, the pointings have strong consequences for higher dimensional theories.
At the end of the paper, in Theorem \ref{thm pointing}, we prove the following, which indicates that the pointing forbids interesting extensions beyond dimension~$n$.\footnote{Theo Johnson-Freyd suggested this claim when collaborating with the second author, after he noted the simplest but most crucial case: in the 1-category of pointed $\kk$-vector spaces, the only dualizable object is the one-dimensional vector space $\kk$ itself. We benefited from discussions with him about the meaning and consequences of the general result. See \cite{JFHeis} for his perspective on this result.}

\begin{theorem*}
The unit $\unit$ is the only $(n+1)$-dualizable object in~$\Alg_n(\S)$.
\end{theorem*}

Note the divergence with case $n=1$ and the classical, unpointed Morita 2-category,
where finitely-generated and projective, separable algebras provide 2-dualizable objects.

\subsection{Connections with other results on duals and adjoints}

Full $2$-du\-al\-iza\-bility has been studied for several 2-categories besides the Morita bicategory.
There are, for instance, several 2-categories that ``deloop'' the category of vector spaces.
Thankfully,
it was shown in \cite[Appendix]{BDSPV} that the fully 2-dualizable subcategories are equivalent. 
This result gives a satisfying answer: for extended 2-dimensional TFTs, we do not have to worry about which of these targets we should choose.

Besides the usual Morita 2-category, variants of $\Alg_n(\S)$ have been investigated thoroughly in  the two settings, briefly mentioned already in the beginning of this introduction. In both situations, $\S$ is a well-behaved 2-category of $\kk$-linear categories, which we will denote by $\mathrm{Cat}_\kk$.\footnote{The technical conditions on $\mathrm{Cat}_\kk$ are crucial, but we choose not to emphasize them here. We refer the interested reader to the references.}

The first case examines $n=1$. 
In \cite{DSPS} and then generalized in \cite{BJS}, 
the authors study a 3-category of tensor categories, bimodule categories, functors, and natural transformations. 
It is a sub-3-category of $\Alg_1^{unptd}(\mathrm{Cat}_\kk)$ whose objects are certain tensor categories. 
In this setting, our main theorem establishes (full) 1-dualizability. 
However, Douglas--Schommer-Pries--Snyder and Brochier--Jordan--Snyder show even more: 
the extra finiteness conditions on objects and morphisms ensure full 2-dualizability and even partial higher dualizability. 
Finally, they establish that 3-dualizable objects are exactly the separable tensor categories (in characteristic zero, it amounts to fusion categories).

The second case examines $n=2$. 
In \cite{BJS}, the authors study a 4-category of braided tensor categories, tensor categories with central structures (bimodules), centered bimodule categories, functors, and natural transformations. 
More precisely, they study a sub-4-category of $\Alg_2^{unptd}(\mathrm{Cat}_\kk)$ whose objects are certain rigid braided tensor categories. 
Our main theorem establishes full 2-dualizability. 
The authors also obtain this result, in this specific categorical setting, by constructing the duals and adjoints by hand. 
But Brochier--Jordan--Snyder show even more: 
the extra finiteness conditions ensure full 3-dualizability. 
Finally, they establish that 4-dualizable objects are exactly the separable braided tensor categories (in characteristic zero it amounts to fusion categories). 

\begin{remark}
We mention that although we emphasize here questions about duals and adjoints,
there are concrete applications of these results.
One motivation for \cite{BJS} is to extend results of \cite{BZBJ},
who analyze the 2-dimensional TFTs determined by the braided tensor categories of representations of quantum groups.
They find rich connections with recent representation theory, such as Alekseev's moduli algebras.
The crucial tool is factorization homology, which yields the 2-dimensional TFTs as shown  in~\cite{CSThesis, CalaqueScheimbauer}.
\end{remark}

In light of our second main theorem, it is clear that both sets of authors work in an unpointed higher Morita category.
In future work, we plan to examine higher dualizability (i.e., existence of adjoints for higher morphisms) in unpointed versions of the higher Morita categories.

\subsection{Future directions}

Our motivation for seeking such results is rooted in our current work on TFT, growing out of the discussion in Section~\ref{sec intro connection TFT}.
In many situations, the object or $k$-morphism simply does not satisfy the necessary higher dualizability/adjunctability conditions: 
for example, most algebras are not finitely generated, projective, and separable.
It is possible to evade such strict conditions by switching to a relative situation {\em \`a la} Stolz-Teichner \cite{StolzTeichner}. 
Here one works with {\em twisted} field theories, which are natural transformations between functorial field theories. 
In the fully extended topological case, a precise definition and characterization in terms of dualizability is available in \cite{JFS}, predicated on the Cobordism Hypothesis.
Combining that result with our results here, one obtains that an $n$-dimensional\footnote{Even though this appears to be the same dimension as the setting we have previously described, the natural transformation $Z_{\Aa}$ should really be thought of as assigning values to certain very simple $(n+1)$-dimensional bordisms with defects.} twisted field theory, realized as a symmetric monoidal (op)lax natural transformation
\[
\begin{tikzpicture}
  \path node (B) {} node[anchor=east] {$\Bord_n^{\mathit{fr}}$} +(2,0) node (C) {} node[anchor=west] {$\Alg_n(\S)$ \, ,};
  \draw[arrow] (B) .. controls +(1,-.5) and +(-1,-.5) .. coordinate (t)  (C) ;
  \draw[arrow] (B) .. controls +(1,.5) and +(-1,.5) .. coordinate (s)  (C);
  \draw[twoarrowlonger] (s) node[above] {$T_{\Rr}$} -- node[anchor=west] {$Z_{\Aa}$}  (t) node[below] {$T_{\Ss}$};
\end{tikzpicture}
\]
is fully determined by a 1-morphism in $\Alg_n(\S)$ with the property that the $n$-morphisms appearing as units and counits for the adjunctions exhibiting $(n-1)$-adjunctibility have {\em certain} adjoints (more precisely, we require exactly half of these adjoints to exist).
This paper exhibits precisely which conditions one needs to check,
since we provide explicit constructions of the units and counits.

In a companion paper, we study low-dimensional examples of this twisted situation.
We show, for example, that when $n=2$, an algebra $\Aa$ provides a twisted field theory to its center $Z(\Aa)$ if and only if $\Aa$ is Azumaya. 
In the underived setting this amounts to being finitely generated and projective as a $Z(\Aa)$-module and separable as a $Z(\Aa)$-algebra.
This condition relaxes the 2-dualizability condition in a substantial but interesting way.

\subsection{Guide to reader}

In the body of the text, we assume the reader is familiar with $(\infty,N)$-categories, $E_n$ algebras, and factorization algebras, although we give informal descriptions and refer the reader to the precise definitions elsewhere.
Therefore, someone not familiar with the technical detail could follow many of arguments if they understand the geometric intuition and are comfortable with the usual Morita 2-category. 
In particular, our approach is essentially agnostic about many details of higher categories, although we will implicitly work with a complete higher Segal model of the factorization higher Morita category $\Alg_n(\S)$.

In Section \ref{sec: dualizability} we recall the definitions of ``having duals'' and ``adjoints'' in a higher category. In Section \ref{sec: rec on morita} we recall the factorization higher Morita category. We start with an informal discussion of the factorization Morita 2-category in Subsection \ref{sec: fact Alg_1}, which requires no prior knowledge of higher categories. We explain the general case in Subsection \ref{sec: fact Alg_n}, with many pictures to informally understand the case $n=2$. The technical conditions for the existence can be found in Theorem \ref{thm existence Morita}. In Section \ref{sec: 2d case} we explain the proof of our main theorem for the two-dimensional case with many pictures to visualize the argument. This can be understood without technical details. The full proof of the main theorem is the content of Section \ref{sec: arbitrary dim}. Finally, Section \ref{sec pointings} deals with pointings and how they prevent higher dualizability. 
Finally, in the appendix we prove that the higher Morita category is symmetric monoidal.

\subsection{Acknowledgements}

This project has benefited from a whole community of researchers who have dwelled on closely related problems and whose interest in our work has kept us motivated.
The impetus to prove the results in this paper was our exploration of twisted TFTs in the style of Stolz-Teichner;
we thank Stephan Stolz and Peter Teichner for inspiring conversations and persistent encouragement over many years.
Rune Haugseng has been a frequent interlocutor on the topics studied here, 
testing our approach with his,
and we thank him for his generosity with his ideas, critiques, technical aid, and non-mathematical company.
Theo Johnson-Freyd suggested the final result, and he has offered a lot of perspective on how to use factorization algebras to organize ideas in physics and higher algebra.
We have also had interesting conversations with Ben Antieau, David Ben-Zvi, Adrien Brochier, David Gepner, David Jordan, Chris Schommer-Pries, and Noah Snyder. We would also like to thank Rune Haugseng, Aaron Mazel-Gee, and Pavel Safronov for catching immediately a mistaken footnote in the first version.

Finally, this work was begun in Room B17 of the Max Planck Institute for Mathematics,
when OG and CS were both postdocs there, and we deeply appreciate the 
open and stimulating atmosphere of MPIM that made it so easy to begin and continue our collaboration. CS was partially supported by the Swiss National Science Foundation, grant P300P2\textunderscore 164652.

\section{Recollection: dualizability}
\label{sec: dualizability}

We very briefly recall the definitions of dualizability and adjointability in a (symmetric monoidal) $(\infty,N)$-category. We refer to \cite{LurieTFT} for more details and to \cite{HaugsengSpans} for a discussion internal to $N$-fold Segal spaces.

\begin{defn}
Let $(\C, \otimes, \unit)$ be a category $\C$ with symmetric monoidal structure $\otimes$ and monoidal unit $\unit$. For an object $\Rr$ in $\C$ a {\em dual} is an object $\Rr^\vee$ together with an evaluation map $\mathrm{ev}:\Rr^\vee \otimes \Rr \to \unit$ and a coevaluation map $\mathrm{coev}:\unit \to \Rr\otimes \Rr^\vee$ such that the snake identities
\begin{align*}
(\id_{\Rr}\otimes \operatorname{ev}) \circ (\operatorname{coev}\otimes \id_{\Rr}) & = \id_{\Rr}, \quad\text{and}\\
( \operatorname{ev} \otimes \id_{\Rr}) \circ (\id_{\Rr} \otimes \operatorname{coev})& = \id_{\Rr}.
\end{align*}
hold. In this case, we say that $\Rr$ is {\em dualizable} or {\em 1-dualizable}.
\end{defn}

\begin{defn}
Let $\C$ be a bicategory and $R: \Rr \to \Ss$ and $L:\Ss\to\Rr$ two 1-morphisms in $\C$. We say that $L$ is a {\em left adjoint of $R$} and $R$ is a {\em right adjoint of $L$} if there are two 2-morphisms
$$u: \id_\Ss \Rightarrow R\circ L \quad\text{and}\quad c: L\circ R\Rightarrow \id_\Rr$$
called the {\em unit} and {\em counit} of the adjunction, respectively, such that
\begin{align*}
L = L\circ \id_\Ss \xRightarrow{\id_L \circ u} &L\circ R \circ L \xRightarrow{c \circ \id_L } \id_\Rr\circ L \quad\text{and} \\
R = \id_\Ss \circ R \xRightarrow{u \circ \id_R} & R \circ L \circ R \xRightarrow{\id_R \circ c } R \circ \id_\Rr
\end{align*}
are identities.
(We will also refer to this as snake identities.)
\end{defn}
Now we lift these definitions to the higher categorical setting. To simplify notation, by ``0-morphism'' we mean an object.
\begin{defn}Let $\C$ be an $(\infty, N)$-category.
\begin{itemize}
\item  A 1-morphism $f: \Rr \to \Ss$ has a {\em left (or right) adjoint} if it has a left (or right) adjoint in the homotopy bicategory $h_2 \C$.
\item If $k\geq 2$, fix two $(k-2)$-morphisms $\Rr$ and $\Ss$ in $\C$ and let $\Aa$ and $\Bb$ be two $(k-1)$-morphisms from $\Rr$ to $\Ss$.
A $k$-morphism $f:\Aa\to \Bb$ has a {\em left (or right) adjoint} if it has left (or right) adjoint in the homotopy bicategory $h_2 \C(\Rr, \Ss)$.
\item If furthermore $\C$ is symmetric monoidal, an object $\Rr$ is {\em dualizable} if it is dualizable in the homotopy category of $\C$.
\end{itemize}
\end{defn}

\begin{remark}
When proving dualizability and adjunctibility we will work within the $(\infty,N)$-category and show the snake identities hold up to equivalence. This implies that in the homotopy (bi)categories, they are identities.
\end{remark}

\begin{defn}
A symmetric monoidal $(\infty,N)$-category $\C$ is said to be {\em fully $n$-dualizable} if every object is dualizable and for every $1\leq k<n$, every $k$-morphism has a left and right adjoint.
\end{defn}

\begin{remark}
In \cite{LurieTFT} the terminology is slightly different: a fully $n$-dualizable $(\infty,N)$-category is said to ``have duals''.
Conversely, $\C$ from above is fully $n$-dualizable in our sense if the underlying $(\infty,n)$-category of $\C$, given by discarding non-invertible higher morphisms, has duals. We choose to include the dependence on $n$ in the notation since the property changes dramatically when increasing $n$, as we will see later.
\end{remark}

We make a convenient observation about detecting $n$-dualizability.

\begin{lemma}\label{lem truncation detects dualizability}
Let $N \geq n$. 
If $\C$ is a symmetric monoidal $(\infty,N)$-category,
then it is fully $n$-dualizable if and only if its $(\infty, n)$-truncation is fully $n$-dualizable.
\end{lemma}

\section{Recollection: the factorization higher Morita category \texorpdfstring{$\Alg_n(\S)$}{Algn(S)}}
\label{sec: rec on morita}

In this section we recall the factorization higher Morita category from \cite{CalaqueScheimbauer}. We will illustrate it thoroughly in the case $n=1$ and $n=2$ to give intuition for the general case and be brief on the general construction. It depends on the dimension parameter $n$ and an input symmetric monoidal $(\infty,N)$-category, which is ``nice''. We give the precise conditions for what it means to be nice to ensure existence in Theorem \ref{technical condition}. 
Examples to keep in mind are spaces with the Cartesian product or chain complexes with tensor product (direct sum is also interesting), both for $N=1$. For $N=2$, a (suitable) bicategory of locally finitely presentable $\kk$-linear categories is an example.

\subsection{The factorization model for the Morita 2-category~\texorpdfstring{$\Alg_1(\S)$}{Alg1(S)}}
\label{sec: fact Alg_1}

In this subsection, we discuss the one-dimensional situation. We first explain the types of factorization algebras we will need, then introduce some simple piece-wise linear maps on $(0,1)$ which will be important for source, target, and composition. Finally, in Section~\ref{sec: one d Morita}, we explain the symmetric monoidal $(\infty,2)$-category~$\Alg_1(\S)$. To compare with the usual Morita bicategory, we add a dictionary at the end of this subsection.

\subsubsection{Constructible one-dimensional factorization algebras}

We will work with a restricted class of one-dimensional factorization algebras: constructible factorization algebras on the interval $(0,1)$ with respect to a stratification given by a finite set of points. 
To illustrate such a stratification, we use a different color for each connected component of the complement of the points, as in the following picture.

\begin{center}
\begin{tikzpicture}[scale=2]
\draw[bluebg, very thick] (-2,0) -- (-1.2,0);
\draw[redbg, very thick] (-1.2,0) -- (0.7,0);
\draw[greenbg, very thick] (0.7,0) -- (2,0);

\draw (2, 0) node [anchor=north] {1};
\draw (-2, 0) node [anchor=north] {0};
\fill (-1.2,0) circle (0.1em) node [anchor=south] {$s_1$};
\fill (0.7,0) circle (0.1em) node [anchor=south] {$s_2$};
\end{tikzpicture}
\end{center}

A constructible factorization algebra for such a stratification is determined (up to equivalence) by its values on the ``basic'' intervals and by the structure maps between them.
There are two kinds of basic intervals: those that contain none of the points (i.e., colored in just one color) and those that contain exactly one point (i.e., colored in two colors). A basic interval with no points can include to a basic interval with one point, but not {\it vice versa}.
Moreover, ``constructibility\footnote{We believe the phrase ``locally constant with respect to the stratification'' might be more descriptive.}'' means that inclusions of a small basic interval into a larger one of the same type must label a structure map that is an equivalence.
Finally, one can include several basic intervals with no points of the same color into a larger one with no point of the same color.

For the stratification picture above, a constructible factorization algebra is determined by the values of five different intervals (e.g.~the fat regions in the picture below) together with structure maps for them.
\begin{center}
\begin{tikzpicture}[scale=2]
\draw[bluebg, very thick] (-2,0) -- (-1.2,0);
\draw[redbg, very thick] (-1.2,0) -- (0.7,0);
\draw[greenbg, very thick] (0.7,0) -- (2,0);

\draw[bluecirc, line width=4pt] (-1.85,0) -- (-1.55,0)
(-1.35,0) -- (-1.2,0);
\draw[redcirc, line width=4pt] (-1.2,0) -- (-1.05,0)
(-0.5,0) -- (0,0)
(0.5,0) -- (0.7,0);
\draw[greencirc, line width=4pt] (0.7,0) -- (0.9,0)
(1.3, 0) -- (1.7,0);

\draw (2, 0) node [anchor=north] {1};
\draw (-2, 0) node [anchor=north] {0};
\fill (-1.2,0) circle (0.1em) node [anchor=south] {$s_1$};
\fill (0.7,0) circle (0.1em) node [anchor=south] {$s_2$};
\end{tikzpicture}
\end{center}
We discuss in detail how to interpret the pictures in \ref{sec one d dictionary}, after formulating our factorization version of the Morita 2-category. In brief, we view each one-colored interval as labeling an algebra and each a two-colored interval as labeling a bimodule.

\subsubsection{Collapse-and-rescale maps}\label{sec: collapse and rescale}

We recall certain piecewise linear maps from \cite{CalaqueScheimbauer} which we will need several times later on, for example for the source and target maps and the composition of 1-morphisms.

\begin{defn}
\label{def: collapse}
Let $0\leq b \leq a \leq 1$ such that $(b,a)\neq (0,1)$. Let $\varrho^b_a: [0,1]\to [0,1]$ be the piecewise linear map which {\em ``collapses''} the interval $[b, a]$ to a point and {\em rescales} the complement back to $[0,1]$. More precisely, if $b<a$, let
$$
\varrho^b_a(x)=
\begin{cases}
\frac{x}{1-(a-b)}, &x \leq b,\\
\frac{b}{1-(a-b)}, & b\leq x \leq a,\\
\frac{x-(a-b)}{1-(a-b)}, & a\leq x.
\end{cases}
$$
If $a = b$, let $\varrho^b_a=id_{[0,1]}$. 
\end{defn}

When $0< b < a < 1$, the map $\varrho^b_a$ sends everything in the interval $[b,a]$ to the point $\frac{b}{1-(a-b)}$ (``collapse''), which is depicted by the area between the dashed lines in the picture below. Everything outside that closed interval is rescaled to create a bijection onto $[0,1]$.
\begin{center}
\begin{tikzpicture}[scale=0.5]
\draw (0,0) -- (10,0);

\draw (9.9, 0.25) -- (10, 0.25) node [anchor=south] {\tiny $1$} -- (10,-0.25) -- (9.9, -0.25);
\draw (0.1, 0.25) -- (0, 0.25) node [anchor=south] {\tiny $0$} -- (0,-0.25) -- (0.1, -0.25);

\fill (2, 0) circle (0.2em) node [anchor=south] {\tiny $b$};
\fill (7, 0) circle (0.2em) node [anchor=south] {\tiny $a$};

\draw (0,-3) -- (10,-3);
\begin{scope}[shift={(0,-3)}]
\draw (9.9, 0.25) -- (10, 0.25) -- (10,-0.25) -- node [anchor=north] {\tiny $1$} (9.9, -0.25);
\draw (0.1, 0.25) -- (0, 0.25) -- (0,-0.25) node [anchor=north] {\tiny $0$} -- (0.1, -0.25);
\end{scope}

\fill (4, -3) circle (0.2em) node [anchor=north] {\tiny $\frac{b}{1-(a-b)}$};

\draw[dashed] (2,0) -- (4,-3) -- (7,0);
\draw[->] (8.5, -0.5) -- node[right] {\tiny $\varrho^b_a$} (8.5, -2.5);
\end{tikzpicture}
\end{center}

For $b=0$, depicted in the left picture below, the map $\varrho^0_a$ collapses everything to the left of $a$ to 0 and rescales $(a,1]$ to $(0,1]$. Dually, for $a=1$, depicted in the right picture below, the map $\varrho^b_1$ collapses everything to the right of $b$ to 1 and rescales $[0,b)$ to $[0,1)$.
\begin{center}
\begin{tikzpicture}[scale=0.4]
\draw (0,0) -- (10,0);
\draw (9.9, 0.25) -- (10, 0.25) node [anchor=south] {\tiny $1$} -- (10,-0.25) -- (9.9, -0.25);
\draw (0.1, 0.25) -- (0, 0.25) node [anchor=south] {\tiny $0$} -- (0,-0.25) -- (0.1, -0.25);

\fill (2, 0) circle (0.2em) node [anchor=south] {\tiny $a$};

\draw (0,-3) -- (10,-3);
\begin{scope}[shift={(0,-3)}]
\draw (9.9, 0.25) -- (10, 0.25) -- (10,-0.25) node [anchor=north] {\tiny $1$} -- (9.9, -0.25);
\draw (0.1, 0.25) -- (0, 0.25) -- (0,-0.25) node [anchor=north] {\tiny $0$} -- (0.1, -0.25);
\end{scope}

\draw[dashed] (0,0) -- (0,-3);
\draw[dashed] (2,0) -- (0,-3);
\draw[->] (8.5, -0.5) -- node[right] {\tiny $\varrho^0_a$} (8.5, -2.5);
\end{tikzpicture}
\hspace{1cm}
\begin{tikzpicture}[scale=0.4]
\draw (0,0) -- (10,0);
\draw (10, 0.25) node [anchor=south] {\tiny $1$};
\draw (9.9, 0.25) -- (10, 0.25) node [anchor=south] {\tiny $1$} -- (10,-0.25) -- (9.9, -0.25);
\draw (0.1, 0.25) -- (0, 0.25) node [anchor=south] {\tiny $0$} -- (0,-0.25) -- (0.1, -0.25);

\fill (7, 0) circle (0.2em) node [anchor=south] {\tiny $b$};

\draw (0,-3) -- (10,-3);
\begin{scope}[shift={(0,-3)}]
\draw (9.9, 0.25) -- (10, 0.25) -- (10,-0.25) -- node [anchor=north] {\tiny $1$} (9.9, -0.25);
\draw (0.1, 0.25) -- (0, 0.25) -- (0,-0.25) node [anchor=north] {\tiny $0$} -- (0.1, -0.25);
\end{scope}

\draw[dashed] (10,0) -- (10,-3);
\draw[dashed] (7,0) -- (10,-3);
\draw[->] (1.5, -0.5) -- node[left] {\tiny $\varrho^b_1$} (1.5, -2.5);
\end{tikzpicture}
\end{center}

\subsubsection{The factorization model for the Morita 2-category}
\label{sec: one d Morita}

The basic idea for the symmetric monoidal factorization Morita $(\infty,2)$-category $\Alg_1(\S)$ is to use the equivalence (Theorem 5.4.5.9 of \cite{LurieHA}) between $E_1$ algebras in $\S$ and locally constant factorization algebras on the interval $(0,1)$ valued in $\S$.
They will be the the objects of the higher category, which we will informally describe by describe the objects, 1, and 2-morphisms.

\paragraph{Objects}
An object (``0-morphism'') in $\Alg_1(\S)$ is just a locally constant factorization algebra $\Rr$ on $(0,1)$. We will depict an object by a little interval, which represents the space $(0,1)$ on which $\F$ lives:
\begin{center}
\begin{tikzpicture}[scale=2]
\draw[draw =bluebg, very thick] (0,0) -- node[anchor=north] {\scriptsize object} (1,0);
\end{tikzpicture}
\end{center}

\paragraph{1-morphisms}
An identity morphism is the same data as the object it lives on, namely a locally constant factorization algebra $\Rr$ on $(0,1)$.

A (generic) 1-morphism is a pair consisting of
\begin{enumerate}
\item a stratification of $(0,1)$ given by a single point $s$, and  
\item a factorization algebra $\F$ on $(0,1)$ that is constructible with respect to the stratification.
\end{enumerate}
We depict a morphism by the stratified space on which $\F$ lives:
\begin{center}
\begin{tikzpicture}[scale=2]
\begin{scope}[xshift=2cm]
\draw[draw =bluebg, very thick] (0,0) -- (0.3, 0) node[anchor=south] {$s$};
\draw[draw =redbg, very thick] (0.3, 0) -- (1,0) ;
\draw (0.3, 0) node[dot] {};
\path (0,0) -- node[anchor=north] {\scriptsize 1-morphism} (1,0);
\end{scope}
\end{tikzpicture}
\end{center}

Our convention is to read pictures from left to right: for a 1-morphism as pictured above, its source is the restriction of the factorization algebra $\F$ on the left of the stratification, i.e.~the blue part, rescaled to $(0,1)$ by pushing forward along $(\varrho^b_1)^{-1}$. Similarly, the target is the restriction of $\F$ to the red part rescaled by pushing forward along $(\varrho^0_b)^{-1}$.

\paragraph{Composition}
Morphisms are composed in two steps:

First we glue together factorization algebras over the common (equivalent) target and source. For example, for two morphisms, we have
$$
\begin{tikzpicture}[scale=2]
\begin{scope}[xshift=0cm]
\draw[draw =bluebg, ultra thick] (0,0) -- (0.3, 0);
\draw[draw =redbg, ultra thick] (0.3, 0) -- (1,0);
\draw (0.3, 0) node[dot] {};
\path (0,0) -- node[anchor=north] {\scriptsize $\F_0$} (1,0);
\draw (1.2, 0) node (A) {};
\end{scope}

\begin{scope}[xshift=2.5cm]
\draw[draw =redbg, ultra thick] (0, 0) -- (0.7,0);
\draw[draw =greenbg, ultra thick] (0.7, 0) -- (1,0);
\draw (0.7, 0) node[dot] {};
\path (0,0) -- node[anchor=north] {\scriptsize $\F_1$} (1,0);
\draw (-0.2, 0) node (B) {};
\end{scope}

\draw[<->] (A) --node[anchor=south] {\scriptsize $t(\F_0)\cong s(\F_1)$} (B);

\draw (4, 0) node {$\rightsquigarrow$};

\begin{scope}[xshift=4.5cm]
\draw[draw =bluebg, ultra thick] (0,0) -- (0.3, 0);
\draw[draw =redbg, ultra thick] (0.3, 0) -- (1,0);
\draw[draw =greenbg, ultra thick] (1.3, 0) -- (1,0);

\draw (0.3, 0) node[dot] {};
\draw (1, 0) node[dot] {};

\draw [thick, decoration={brace, mirror, raise=0.2cm}, decorate] (0, 0) --  node[anchor=north, yshift= -0.25cm] {\scriptsize $\F_0$} (1, 0);

\draw [thick, decoration={brace, raise=0.2cm}, decorate] (0.3, 0) --  node[anchor=south, yshift= 0.25cm] {\scriptsize $\F_1$} (1.3, 0);
\end{scope}

\end{tikzpicture}
$$
Second, we push forward along a ``collapse-and-rescale map'' $\varrho$, which collapses everything between the two points, and then rescales back to get $(0,1)$.

$$\begin{tikzpicture}
% top pic
\draw[draw =bluebg, ultra thick] (0,0) -- (0.3, 0);
\draw[draw =redbg, ultra thick] (0.3, 0) -- (1,0);
\draw[draw =greenbg, ultra thick] (1.3, 0) -- (1,0);

\draw (0.3, 0) node[dot] {};
\draw (1, 0) node[dot] {};

% middle pic
\draw[draw =bluebg, ultra thick] (0.35,-1) -- (0.65, -1);
\draw[draw =greenbg, ultra thick] (0.95,-1) -- (0.65, -1);
\path (0.95,-1) node (B) {} -- (0.35,-1) node (A) {};
\draw (0.65, -1) node[dot] {};

%lower pic
\draw[draw =bluebg, ultra thick] (0.15, -2) -- (0.65, -2);
\draw[draw =greenbg, ultra thick] (0.65, -2) -- (1.15, -2);
\draw (0.65, -2) node[dot] {};

\draw[densely dotted] (0, 0) -- (0.35, -1)
(0.3, 0) -- (0.65, -1) -- (1, 0)
(1.3, 0) -- (0.95,-1)

(0.35,-1) -- (0.15, -2)
(0.95, -1) -- (1.15, -2);

%arrow
\draw[->] (2, 0) -- node[anchor=west] {$\varrho \times \id$} (2, -2);
\end{tikzpicture}
$$

\paragraph{2-morphisms}

Note that objects and 1-morphisms were (essentially) given by certain factorization algebras. Factorization algebras valued in an $(\infty,1)$-category $\S$ form an $(\infty,1)$-category $\mathrm{Fact}_{(0,1)}(\S)$ \cite{CostelloGwilliam}, and we use this extra category layer to obtain the 2-morphisms in $\Alg_1(\S)$.

Unraveling this, a 2-morphism from 
\begin{center}
\begin{tikzpicture}[scale=2]
\begin{scope}[xshift=0cm]
\draw (-0.5, 0) node{$\F$ on };
\draw[draw =bluebg, very thick] (0,0) -- (0.3, 0) node[anchor=south] {$s$};
\draw[draw =redbg, very thick] (0.3, 0) -- (1,0) ;
\draw (0.3, 0) node[dot] {};
\end{scope}
\draw (1.5, 0) node{to $\G$ on};
\begin{scope}[xshift=2cm]
\draw[draw =bluebg, very thick] (0,0) -- (0.6, 0) node[anchor=south] {$t$};
\draw[draw =redbg, very thick] (0.6, 0) -- (1,0) ;
\draw (0.6, 0) node[dot] {};

\end{scope}
\end{tikzpicture}
\end{center}
consists of a morphism of factorization algebra $\ell_*\F \to \G$, where $\ell:(0,1) \to (0,1)$ is the unique piece-wise linear map preserving the endpoints and sending $s$ to $t$.

\paragraph{The symmetric monoidal structure}
If $\S$ has a symmetric monoidal structure, then it induces one on $\mathrm{Fact}_{(0,1)}(\S)$:
given $\Rr$ and $\Ss$, their tensor product  $\Rr \otimes \Ss$ is defined by the formula
\[
(\Rr \otimes \Ss )(U) = \Rr(U)\otimes \Ss(U)
\]
for each open set~$U$.
This symmetric monoidal structure then extends to a symmetric monoidal structure on $\Alg_1(\S)$. 
On objects, it is given by the symmetric monoidal product of the factorization algebras.
On 1-morphisms, if the stratifications are the same (so the marked point is at the same position), then one can also simply tensor the factorization algebras. 
In that case, the tensor product manifestly produces a constructible factorization algebra as desired.
If the stratifications are not identical, then pick a homeomorphism of the second interval such that the image of the stratification is identical to the first interval's stratification. 
Then push the second factorization algebra forward along this map and tensor with the factorization algebra on the first interval.

\subsubsection{A dictionary to the ``usual'' Morita category}\label{sec one d dictionary}

We now turn to explaining how to interpret, in terms of algebras and bimodules, this 2-category built from factorization algebras.
The dictionary is:
\begin{center}
\begin{tabular}{|c|c|c|}
\hline 
{\bf 2-category} & {\bf standard version} & {\bf factorization version} \\
\hline 
0-morphism & algebra & locally constant \\
\hline
1-morphism & bimodule & constructible for $\{s\} \subset (0,1)$ \\
\hline
2-morphism & bimodule map & map of factorization algebras \\
\hline
\end{tabular}
\end{center}

The justification runs as follows.

Consider first the objects, namely locally constant factorization algebras on an interval $(0,1)$. 
In Theorem 5.4.5.9 of \cite{LurieHA}, Lurie demonstrates that $E_1$ algebras correspond to locally constant factorization algebras in the homotopy-coherent sense.
But one can get intuition for this claim by considering a stricter situation:
suppose a factorization algebra $\F$ is strictly locally constant and that all structure maps are associative on the nose.
It is a direct check to see that $\F$ encodes a strictly associative algebra $A_\F$.
Namely, let $A_\F = \F((0,1))$, the value of $\F$ on the whole interval.
By hypothesis, for any interval, we have an isomorphism $\F((a,b)) \cong \F((0,1))$,
so we are free to view $A_\F$ as the value of any interval.
Hence, for any pair of disjoint intervals $(a,b) \sqcup (c,d) \subset (0,1)$, with $b < c$,
the structure map of the factorization algebra determines a map $A_\F \otimes A_\F \to A_\F$,
and it is independent of the choice of pairs.
This map determines an associative multiplication, as one can see by considering triples of disjoint intervals and how they include into disjoint pairs.

Conversely, given a unital associative algebra $A$, one can construct a locally constant factorization algebra $\F_A$ as follows.
To any interval $(a,b)$, $\F_A$ assigns $A$.
To an inclusion of intervals, $\F_A$ assigns the identity.
For any inclusion of $k$ disjoint intervals into a larger interval, the structure map of $\F_A$ is the $k$-fold multiplication, where one uses the natural ordering on the intervals inherited from the interval~$(0,1)$.

Similar reasoning applies to the 1-morphisms.
Let us start by explaining how a {\em pointed} $(A,B)$-bimodule $m_0: \unit \to M$ determines a constructible factorization algebra $\F_M$ on the stratification $\{s\} \subset (0,1)$.
On the subinterval $(0,s)$, $\F_M$ agrees with $\F_A$ we just described.
So to any interval $(a,b) \subset (0,s)$, $\F_M$ assigns $A$, and to inclusions of disjoint intervals into such an $(a,b)$, it uses multiplication in $A$.
Similarly, on the subinterval $(s,1)$, $\F_M$ assigns the factorization algebra $\F_B$ associated to the algebra $B$.
However, for any interval $(a,b)$ with $a < s < b$, $\F_M$ assigns $M$.
Here is where the pointing becomes important.
Consider the inclusion of intervals
\[
(a,b) \subset (c,d)
\]
with $c < a < b < s < d$.
The associated structure map for $\F_m$ must be given by a map $A \to M$
that is compatible with pointings $\unit \to A$ and $\unit \to M$,
which are the structure maps arising from the inclusion of the empty set.
This constraint fully determines the structure map:
the pointing $m_0: \unit \to M$ in the ambient category $\S$ determines a map of left $A$-modules $m_0^A: A \to M$ by the free-forget adjunction.
The same reasoning shows that the canonical map $m_0^B: B \to M$ of right  $B$-modules
determines the structure map for an inclusion of intervals
\[
(a,b) \subset (c,d)
\]
with $c < s < a < b < d$.
The final important case to consider is an inclusion of disjoint intervals
\[
(a,b) \sqcup (a',b') \sqcup (a'',b'') \subset (c, d)
\]
with 
\[
c < a < b < a' < s < b' < a'' < b'' < d.
\]
The corresponding structure map
\[
A \otimes M \otimes B \to M
\]
is determined by the bimodule structure of $M$;
we simply use the action of $A$ from the left and action of $B$ from the right.
More generally, given a collection of disjoint intervals including into an interval $(a,b)$ with $a < s < b$,
one uses the pointed bimodule structure to determine the structure maps.

Conversely, if a factorization algebra is strictly constructible with respect to the stratification $\{s\} \subset (0,1)$ and its structure maps are strictly associative,
then one can read off a pointed bimodule.
The factorization algebra on the subinterval $(0,s)$ corresponds to some algebra $A$,
and on the subinterval $(s,1)$ to some algebra $B$.
The value on any interval containing $s$ determines an $(A,B)$-bimodule $M$ by the structure maps,
reversing the construction we gave in the preceding paragraph.
With more care, one can deduce a homotopy-coherent version,
which is the bimodule analogue of Lurie's result for the locally constant case.

\subsection{Constructible factorization algebras for arbitrary dimension}

We will only need a restricted class of $n$-dimensional factorization algebras: constructible factorization algebras on the $n$-dimensional cube $(0,1)^n$ with stratifications of a very simple type. For many more details and more complicated stratified spaces, see  \cite{AFTlocal, AFTfact,Ginot}. The latter also serves as an expository introduction. 

For us the most important examples of stratifications will be ``affine flags''. The prototype thereof is the ``standard affine flag'' given by
\begin{align*}
\{(\tfrac12,\ldots, \tfrac12)\}&\times (0,1)^{n-k} \subset \{(\tfrac12,\ldots, \tfrac12)\}\times (0,1)^{n-k+1} \subset \cdots \\ 
&\cdots \subset \{ \tfrac12\}\times (0,1)^{n-1} \subset (0,1)^n
\end{align*}
for $0\leq k\leq n$.\footnote{This flag looks like a copy of $\RR^{n-k}$ (for the last $n-k$ coordinates) crossed with the standard flag on $k$-dimensional space, albeit shifted to the point $1/2$.} 
Since this flag will be the local picture of our stratifications, we will call this the {\em standard stratified disk of index $k$}.

\begin{example}
For $n=2$ there are three standard stratified disks and can be illustrated as follows:
\begin{center}
\begin{tikzpicture}[scale=0.3, baseline=(x1.south)]
\fill[color=redbg] (0,0) node (x1) {} circle (1);

\begin{scope}[xshift = 3cm]
\fill[color=redbg] (0,-1) arc (270:90:1);
\fill[color=bluebg] (0,-1) arc (-90:90:1);
\draw [thick] (0,1) -- (0,-1);
\end{scope}

\begin{scope}[xshift = 6cm]
\fill[color=redbg] (0,-1) arc (270:90:1);
\fill[color=bluebg] (0,-1) arc (-90:90:1);
\draw [thick] (0,1) --node[dot] {} (0,-1);
\end{scope}
\end{tikzpicture}
\end{center}
The colors just indicate the number of connected components of the 2-dimensional stratum.~\hfill$\mpim$
\end{example}

We are interested in stratifications
$$\emptyset=X_{-1}\subset X_0\subset X_1 \subset \cdots \subset X_n = (0,1)^n$$
such that for every point $x\in X$, there is an open neighborhood of $x$ that is diffeomorphic,  as stratified spaces, to a standard stratified disk of some index. This means that the diffeomorphism restricts to diffeomorphisms of the subspaces. In particular, every $X_i$ is an $i$-dimensional smooth closed submanifold of $(0,1)^n$. For simplicity, we only consider stratifications with strata with finitely many connected components.

\begin{example}
For $n=1$ this definition specifies to the ones in Section~\ref{sec: fact Alg_1}.

For $n=2$ this definition amounts to a 1-manifold $X_1$ embedded in the unit square and marked with certain distinguished points~$X_0$.
Consider the following example.
$$\begin{tikzpicture}
\fill[yellow] (0,0) -- (1,0) arc (0:90:1);

\fill[redbg] (1,0) arc (0:90:1) -- (0,1.5) .. controls (1,0.7) and (1.4, 3) .. (1.4, 1) ..controls (1.4, 0) and (2.6, 0.7) .. (3, 1) -- (3,0);

\fill[bluebg] (0,1.5) .. controls (1,0.7) and (1.4, 3) .. (1.4, 1) ..controls (1.4, 0) and (2.6, 0.7) .. (3, 1) -- (3,3) -- (0,3);

\filldraw[fill=greenbg, thick] (2.1,2.3) ellipse (0.7 and 0.5);

\draw[thick] (0,1.5) .. controls (1,0.7) and (1.4, 3) .. (1.4, 1) ..controls (1.4, 0) and (2.6, 0.7) .. (3, 1);
\draw[thick] (1,0) arc (0:90:1);

\draw (1.4, 1) node[dot] {};

\path (1,0) arc (0:25:1) node[dot] {} arc (25: 45:1) node[dot] {} arc (45: 70:1) node[dot] {};
\path (2.1,2.3) ellipse (0.7 and 0.5);
\path (2.8, 2.3) node[dot] {};
\path (1.4, 2.3) node[dot] {};
\path (2.1, 2.8) node[dot] {};
\path (2.1, 1.8) node[dot] {};
\end{tikzpicture}
$$
Note how the 1-manifolds have no boundary.~\hfill$\mpim$
\end{example}

In the rest of this article, by ``stratified space'' and ``stratification'' we always assume we are in the just described situation.
\begin{defn} Let $(0,1)^n$ be equipped with a stratification.
A {\em stratified disk $D$ of index $k$} is an open disk in $(0,1)^n$ that is diffeomorphic as stratified spaces to a standard stratified disk of index $k$.
\end{defn}

In other words, $D\cap(X_k\setminus X_{k-1})$ is connected and non-empty and~$D\subset X\setminus X_{k-1}$.

Note that locally, every such stratified space is diffeomorphic to an open disk $D$ such that for some $\alpha$, we have that $D\cap(X_\alpha\setminus X_{\alpha-1})$ is connected and non-empty and $D\subset X\setminus X_{\alpha-1}$. We call such a $D$ a {\em basic stratified disk of index $\alpha$}.

\begin{defn}
Let $\emptyset=X_{-1}\subset X_0\subset X_1\subset \cdots \subset X_n= (0,1)^n$ be a stratification of $(0,1)^n$. 
A factorization algebra $\F$ on the stratified space $X$ is called \emph{constructible} (or {\em locally constant with respect to the stratification}) 
if for any inclusion of two disks $D_1\hookrightarrow D_2$ that are each diffeomorphic to the same basic stratified disk, the structure map
$$\F(D_1) \xrightarrow{\simeq} \F(D_2)$$
is an equivalence. 
If the stratification is the empty stratification, then a constructible factorization algebra is called {\em locally constant}.
\end{defn}

\begin{example}
For $n=2$, 
a constructible factorization algebra is essentially (up to equivalence) determined by a finite amount of data.

For example, consider the stratification on the left.
We draw the two connected 1-dimensional components of $X_1\setminus X_0$ with a different pattern, to distinguish them  clearly.
There are five types of disks, as indicated in the middle (two different unicolored disks, two different bicolored disks, and one disk with a dot). Two examples of inclusions inducing the structure maps between them are indicated on the right:
\begin{center}
\begin{tikzpicture}[scale=2.5]
\fill[fill =bluebg] (0,0) -- (0.5, 0) -- (0.5, 1) -- (0,1) -- cycle;
\fill[fill =redbg] (0.5, 0) -- (0.5, 1) -- (1,1) -- (1,0) -- cycle;
\path (0,0) -- (1,0);
\draw[thick, dash pattern=on 2pt off 1pt on 4pt off 1pt] (0.5, 0) -- (0.5, 0.5);
\draw[thick] (0.5, 0.5) -- (0.5, 1);
\draw (0.5, 0.5) node[dot] {};

\begin{scope}[xshift = 1.5 cm]
\fill[fill =bluebg] (0,0) -- (0.5, 0) -- (0.5, 1) -- (0,1) -- cycle;
\fill[fill =redbg] (0.5, 0) -- (0.5, 1) -- (1,1) -- (1,0) -- cycle;
\path (0,0) -- (1,0);
\draw[thick, dash pattern=on 2pt off 1pt on 4pt off 1pt] (0.5, 0) -- (0.5, 0.5);
\draw[thick] (0.5, 0.5) -- (0.5, 1);
\draw (0.5, 0.5) node[dot] {};

\fill[fill=bluecirc] (0.25, 0.5) circle (0.1);
\draw (0.5, 0.5) circle (0.1);
\fill[fill=redcirc] (0.75, 0.5) circle (0.1);
\draw (0.5, 0.75) circle (0.1);
\draw (0.5, 0.25) circle (0.1);
\end{scope}

\begin{scope}[xshift = 3 cm]
\fill[fill =bluebg] (0,0) -- (0.5, 0) -- (0.5, 1) -- (0,1) -- cycle;
\fill[fill =redbg] (0.5, 0) -- (0.5, 1) -- (1,1) -- (1,0) -- cycle;
\path (0,0) -- (1,0);
\draw[thick, dash pattern=on 2pt off 1pt on 4pt off 1pt] (0.5, 0) -- (0.5, 0.5);
\draw[thick] (0.5, 0.5) -- (0.5, 1);
\draw (0.5, 0.5) node[dot] {};

\fill[fill=bluecirc] (0.25, 0.75) circle (0.1);
\draw (0.5, 0.75) circle (0.1);
\draw (0.5, 0.57) arc (-90:90:0.18) -- (0.25, 0.93) arc (90:270:0.18) -- cycle;

\draw (0.5, 0.25) circle (0.185);
\draw (0.5, 0.16) circle (0.08);
\draw (0.5, 0.34) circle (0.08);
\end{scope}
\end{tikzpicture}
\end{center}
Note that including the empty set $\emptyset$ (whose value is always $\unit$) into any of the basic disks also requires a structure map, which is part of the data.

In general, a constructible factorization algebra on a stratified square is essentially determined by the value of
$$\# \pi_0(X_0) + \# \pi_0(X_1\setminus X_0) + \# \pi_0 (X_2\setminus X_1) $$
basic disks and the structure maps for them.~\hfill$\mpim$
\end{example}

We will repeatedly use the following operation to produce examples of factorization algebras.

\begin{definition}
Let $p:X\to Y$ be a continuous map and $\F$ a factorization algebra on $X$. 
The {\em pushforward factorization algebra} $p_*\F$ on $Y$ is given by the formula
$$p_*\F(U) = \F(p^{-1}(U)),$$
which fully determines the structure maps in terms of those of~$\F$.
\end{definition}

\begin{lemma}[\cite{Ginot}, Corollary 6]\label{lemma constructible pushforward}
Let $p:X\to Y$ be an ``adequately stratified'' map of stratified spaces. If $\F$ is a constructible factorization algebra on $X$, then $p_*\F$ is constructible on $Y$.
\end{lemma}

In this paper we will use only the following examples of adequately stratified maps:
\begin{itemize}
\item (local) diffeomorphisms of stratified spaces,
\item collapse-and-rescale maps from Section~\ref{sec: collapse and rescale},
\item refinements of the stratification, and
\item locally trivial stratified fibrations.
\end{itemize}

\subsection{The factorization model for the higher Morita category~\texorpdfstring{$\Alg_n(\S)$}{Algn(S)}}
\label{sec: fact Alg_n}

In this section we briefly and informally review the construction of the Morita $(\infty,n)$-category of $E_n$ algebras using factorization algebras from \cite{CSThesis, CalaqueScheimbauer}. In fact, it is a complete $n$-fold Segal object in $\infty$-categories. The main idea is to use the equivalence of $\infty$-categories of $E_n$ algebras and of locally constant factorization algebras on $(0,1)^n$ of Theorem 5.4.5.9 of \cite{LurieHA}; we {\em define} the objects of our higher category as the latter.
The interested reader can find full details on the construction in the above mentioned reference.  

One compelling aspect of this approach is that it makes visible and explicit the {\em geometric} aspects of $n$-dimensional algebra, as the many pictures, which illustrate the case $n=2$, will show.

Let us again use the convention that ``0-morphism'' means ``object'' of the higher category.
For $0\leq k\leq n$, the $k$-morphisms are pairs consisting of
\begin{enumerate}
\item auxiliary data of some subintervals of $(0,1)$, which determine a special type of stratification of $(0,1)^n$, an ``affine flag'' consisting of a codimension $k$ subspace sitting inside a codimension $k-1$ subspace and so on:
\begin{align*}
\{(a^1,\ldots, a^k)\}\times (0,1)^{n-k} &\subset \{(a^1,\ldots, a^{k-1})\}\times (0,1)^{n-k+1} \subset \cdots \\ 
&\cdots \subset \{a^1\}\times (0,1)^{n-1} \subset (0,1)^n.
\end{align*}
\item a factorization algebra $\F$ on $(0,1)^n$ that is constructible with respect to the stratification.
\end{enumerate}

\begin{remark}
To give a flag as above, we must specify a point $a^i \in (0,1)$ for each $1 \leq i \leq k$ that says how to split each of the first $k$ coordinates into two pieces. In particular, an object ($k=0$) is just a locally constant factorization algebra on $(0,1)^n$.
The auxiliary data in the full definition specifies that point in a more convenient form for constructing a complete $n$-fold Segal space.
Moreover, we issue a small caveat to this gloss:
a $k$-morphism equivalent to an identity morphism need not have a nontrivial stratification.
It does, however, possess the auxiliary data.
\end{remark}

\begin{remark}
In this text we choose the top-dimensional stratum to always be $(0,1)^n$.
Instead, we could allow any product of intervals
$$(\alpha_1, \beta_1)\times \cdots \times (\alpha_n, \beta_n)\subset \R^n,$$
and identify (i.e.~add a path between) any two objects which are related by a rescaling given by a homeomophism.
Then we do not need as many linear rescalings in the definition of source, target, and composition below.
When drawing pictures, it will sometimes be convenient to draw rectangles instead of squares.
\end{remark}

\begin{example}\label{ex morphism pics}
We illustrate the allowed types of stratifications in the case~$n=2$:
$$
\begin{tikzpicture}[scale=2]
\fill[fill =bluebg] (0,0) -- node[anchor=north] {\scriptsize object} (1,0) -- (1,1) -- (0,1) -- cycle;

\begin{scope}[xshift=2cm]
\fill[fill =bluebg] (0,0) -- (0.3, 0) -- (0.3, 1) -- (0,1) -- cycle;
\fill[fill =redbg] (0.3, 0) -- (0.3, 1) -- (1,1) -- (1,0) -- cycle;
\path (0,0) -- node[anchor=north] {\scriptsize 1-morphism} (1,0);
\draw[thick] (0.3, 0) -- (0.3, 1);
\end{scope}

\begin{scope}[xshift=4cm]
\fill[fill =bluebg] (0,0) -- (0.6, 0) -- (0.6, 1) -- (0,1) -- cycle;
\fill[fill =redbg] (0.6, 0) -- (0.6, 1) -- (1,1) -- (1,0) -- cycle;
\path (0,0) -- node[anchor=north] {\scriptsize 2-morphism} (1,0);
\draw[thick] (0.6, 0) -- (0.6, 1);
\draw (0.6, 0.6) node[dot] {};
\end{scope}
\end{tikzpicture}
$$
An objects has the empty stratification, a (non-identity) 1-morphism has a stratification given by a vertical line, and a (non-identity) 2-morphism has a stratification consisting of a vertical line together with a point in its interior. In the illustration, we use a different color for each connected component of its complement.~\hfill$\mpim$
\end{example}

\paragraph{Source and target}
The source and target of a $k$-morphism is given by restricting the factorization algebra $\F$ to one of the two components of the complement of the hyperplane $\{x_k=a^k\}$ in $(0,1)^n$.  The source is given by restricting $\F$ to $\{x_k<a^k\}$ and the target is given by restricting $\F$ to $\{x_k>a^k\}$ in $(0,1)^n$.

Let us make this explicit and illustrate it for $n=2$, hopefully making the procedure clear in general.
Our convention is to read 2-dimensional pictures from left to right and from bottom to top. By this we mean the following.

Consider a 1-morphism as pictured in Example \ref{ex morphism pics}, where the line is given by $\{x_1=a^1\}$ for some $0<a^1<1$. Its source is the restriction of the factorization algebra $\F$ on the left of the stratification, i.e.~the blue part, rescaled to $(0,1)^2$ by pushing forward along $(\varrho^{a^1}_1\times \id)^{-1}$. Similarly, the target is the restriction to the red part rescaled by pushing forward along $(\varrho^0_{a^1}\times \id)^{-1}$.
$$\begin{tikzpicture}
\fill[fill =bluebg] (0,0) -- (0.3, 0) -- (0.3, 1) -- (0,1) -- cycle;
\fill[fill =redbg] (0.3, 0) -- (0.3, 1) -- (1,1) -- (1,0) -- cycle;
\draw[thick] (0.3, 0) -- (0.3, 1);
\draw[|->] (1.3, 0.5) -- (1.8, 0.5);
\draw[|->] (-0.3, 0.5) -- (-0.8, 0.5);

\begin{scope}[xshift=1.8cm]
\fill[fill =redbg] (0.3, 0) -- (0.3, 1) -- (1,1) -- (1,0) -- cycle;
\draw[densely dotted] (0.3, 0) -- (0.3, -1) (1,0) -- (1.3,-1);
\fill[fill =redbg] (0.3, -1) -- (0.3, -2) -- (1.3, -2) -- (1.3, -1) -- cycle;
\draw[|->] (0.65, -0.2) -- node{\tiny $(\varrho^0_{a^1}\times \id)^{-1}$} (0.8, -0.8);
\draw (0.8,-2) node[anchor=north] {\scriptsize target};
\end{scope}

\begin{scope}[xshift=-1.4cm]
\fill[fill =bluebg] (0,0) -- (0.3, 0) -- (0.3, 1) -- (0,1) -- cycle;
\draw[densely dotted] (0.3, 0) -- (0.3, -1) (0,0) -- (-0.7,-1);
\fill[fill =bluebg] (0.3,-2) -- (-0.7, -2) -- (-0.7, -1) -- (0.3, -1) -- cycle;
\draw[|->] (0.15, -0.2) -- node{\tiny $(\varrho^{a^1}_1\times \id)^{-1}$} (-0.2, -0.8);
\draw (-0.2,-2) node[anchor=north] {\scriptsize source};
\end{scope}
\end{tikzpicture}$$

\begin{remark}\label{rem visualize morphism}
We use the following notational conventations throughout the paper. We denote by $\Rr$ the locally constant factorization algebra of the source and by $\Ss$ the locally constant factorization algebra of the target and suggestively draw them into the picture. Moreover, let $\Aa$ denote the values of $\F$ at disks which intersect the line. We can also think of $\Aa$ as denoting the locally constant factorization algebra on $(0,1)$ obtained by pushing forward along the horizontal projection~$(x,y) \mapsto y$.
Hence we visualize a 1-morphism determined by $\F$ by
$$
\includestandalone{pics_tex/pic-morphism_small}\, .
$$
\end{remark}

For a 2-morphism, in addition to a vertical line $\{x_1=a^1\}$, there is a point lying at the intersection with the horizontal line $\{x_2=a^2\}$ for some $0<a^2<1$. The source (or target) is given by restricting $\F$ to the part below (or above, respectively) the horizontal line passing through the point, and rescaling.

$$\begin{tikzpicture}
\fill[fill =bluebg] (0,0) -- (0.6, 0) -- (0.6, 1) -- (0,1) -- cycle;
\fill[fill =redbg] (0.6, 0) -- (0.6, 1) -- (1,1) -- (1,0) -- cycle;
\draw[thick] (0.6, 0) -- (0.6, 1);
\draw (0.6, 0.6) node[dot] {};

\draw[|->] (1.2, 0.8) -- (1.8, 1.2);
\draw[|->] (1.2, 0.2) -- (1.8, -0.3);

%top right
\begin{scope}[shift={(2, 0.4)}]
\fill[fill =bluebg] (0,0.6) -- (0.6, 0.6) -- (0.6, 1) -- (0,1) -- cycle;
\fill[fill =redbg] (0.6, 0.6) -- (0.6, 1) -- (1,1) -- (1,0.6) -- cycle;
\draw[thick] (0.6, 0.6) -- (0.6, 1);

\draw[densely dotted] (1, 0.6) -- (2, 0.6) (1, 1) -- (2,1.6);
\end{scope}

\begin{scope}[shift={(4,1)}]
\fill[fill =bluebg] (0,0) -- (0.3, 0) -- (0.3, 1) -- (0,1) -- cycle;
\fill[fill =redbg] (0.3, 0) -- (0.3, 1) -- (1,1) -- node[anchor=west]{\scriptsize  source} (1,0) -- cycle;
\draw[thick] (0.3, 0) -- (0.3, 1);
\end{scope}

%bottom right
\begin{scope}[shift={(2, -.6)}]
\fill[fill =bluebg] (0,0) -- (0.6, 0) -- (0.6, 0.6) -- (0,0.6) -- cycle;
\fill[fill =redbg] (0.6, 0) -- (0.6, 0.6) -- (1,0.6) -- (1,0) -- cycle;
\draw[thick] (0.6, 0) -- (0.6, 0.6);

\draw[densely dotted] (1, 0.6) -- (2, 0.6) (1, 0) -- (2,-0.4);
\end{scope}

\begin{scope}[shift={(4,-1)}]
\fill[fill =bluebg] (0,0) -- (0.3, 0) -- (0.3, 1) -- (0,1) -- cycle;
\fill[fill =redbg] (0.3, 0) -- (0.3, 1) -- (1,1) -- node[anchor=west]{\scriptsize  target} (1,0) -- cycle;
\draw[thick] (0.3, 0) -- (0.3, 1);
\end{scope}

\end{tikzpicture}
$$

\begin{remark}\label{rem visualize 2-morphism}
For 2-morphisms we use the following notational conventations. Again, we denote by $\Rr$ and $\Ss$ the locally constant factorization algebras of the source and target objects, respectively. Moreover, let $\Aa$ denote the values of $\F$ at disks which intersect the line below the point and $\Bb$ the values of $\F$ at disks above the point. Finally, values of $\F$ at disks containing the point are denoted by $M$.
Hence we visualize a 2-morphism determined by $\F$ by
$$
\includestandalone{pics_tex/pic-2-morphism_small} \,.
$$
\end{remark}

\paragraph{Composition}
Morphisms are composed in two steps:

First we glue together factorization algebras over the common (equivalent\footnote{We ``compose'' morphisms when there exists an equivalence between target and source, so that composition is only well-defined up to homotopy. This is why this results in a 2-fold Segal space rather than an enriched 2-category on the nose.}) target and source. For example, when $n=2$, for two 1-morphisms, we have
$$
\begin{tikzpicture}[scale=1]
\begin{scope}[xshift=0cm]
\fill[fill =bluebg] (0,0) -- (0.3, 0) -- (0.3, 1) -- (0,1) -- cycle;
\fill[fill =redbg] (0.3, 0) -- (0.3, 1) -- (1,1) -- (1,0) -- cycle;
\path (0,0) -- node[anchor=north] {\scriptsize $\F_0$} (1,0);
\draw[thick] (0.3, 0) -- (0.3, 1);
\draw (0.7, 0.5) node (A) {};
\end{scope}

\begin{scope}[xshift=3cm]
\fill[fill =redbg] (0, 0) -- (0, 1) -- (0.7,1) -- (0.7,0) -- cycle;
\fill[fill =greenbg] (0.7, 0) -- (0.7, 1) -- (1,1) -- (1,0) -- cycle;
\path (0,0) -- node[anchor=north] {\scriptsize $\F_1$} (1,0);
\draw[thick] (0.7, 0) -- (0.7, 1);
\draw (0.3, 0.5) node (B) {};
\end{scope}

\draw[<->] (A) --node[anchor=south] {\scriptsize $t(\F_0)\simeq s(\F_1)$} (B);

\draw (5, 0.5) node {$\rightsquigarrow$};

\begin{scope}[xshift=6cm]
\fill[fill =bluebg] (0,0) -- (0.3, 0) -- (0.3, 1) -- (0,1) -- cycle;
\fill[fill =redbg] (0.3, 0) -- (0.3, 1) -- (1,1) -- (1,0) -- cycle;
\fill[fill =greenbg] (1.3, 0) -- (1.3, 1) -- (1,1) -- (1,0) -- cycle;

\draw[thick] (0.3, 0) -- (0.3, 1)
(1, 0) -- (1, 1);

\draw [thick, decoration={brace, mirror, raise=0.1cm}, decorate] (0, 0) --  node[anchor=north, yshift= -0.15cm] {\scriptsize $\F_0$} (1, 0);

\draw [thick, decoration={brace, raise=0.1cm}, decorate] (0.3, 1) --  node[anchor=south, yshift= 0.15cm] {\scriptsize $\F_1$} (1.3, 1);
\end{scope}

\end{tikzpicture}
$$
Second, we push forward along a piecewise-linear ``collapse-and-rescale map'', which collapses everything between the two hyperplanes, and then rescales back to get $(0,1)^n$. More explicitly, for $n=2$, it is a product of maps $\varrho\times id$ for the composition of 1-morphisms (resp.~$id \times \varrho$ for the vertical composition of 2-morphisms). Here $\varrho$ is similar to the ones defined in Section~\ref{sec: collapse and rescale}.
$$\begin{tikzpicture}
% top pic
\fill[fill =bluebg] (0,0) -- (0.3, 0) -- (0.3, 1) -- (0,1) -- cycle;
\fill[fill =redbg] (0.3, 0) -- (0.3, 1) -- (1,1) -- (1,0) -- cycle;
\fill[fill =greenbg] (1.3, 0) -- (1.3, 1) -- (1,1) -- (1,0) -- cycle;
\draw[thick] (0.3, 0) -- (0.3, 1)
(1, 0) -- (1, 1);

% middle pic
\fill[fill =bluebg] (0.35,-1) -- (0.35, -2) -- (0.65, -2) -- (0.65, -1) -- cycle;
\fill[fill =greenbg] (0.95,-1) -- (0.95,-2) -- (0.65, -2) -- (0.65, -1) -- cycle;
\path (0.35,-2) -- (0.95,-2) -- (0.95,-1) node (B) {} -- (0.35,-1) node (A) {} -- cycle;
\draw[thick] (0.65, -2) -- (0.65, -1) node (C) {};

\fill[fill =bluebg] (0.15, -4) -- (0.65, -4) -- (0.65, -3) -- (0.15,-3) -- cycle;
\fill[fill =greenbg] (0.65, -3) -- (0.65, -4) -- (1.15, -4) -- (1.15, -3) -- cycle;
\draw[thick] (0.65, -4) -- (0.65, -3);

\draw[densely dotted] (0, 0) -- (0.35, -1)
(0.3, 0) -- (0.65, -1) -- (1, 0)
(1.3, 0) -- (0.95,-1)

(0.35,-2) -- (0.15, -3)
(0.95, -2) -- (1.15, -3);

%arrow
\draw[->] (2, 0.5) -- node[anchor=west] {$\varrho \times \id$} (2, -3.5);
\end{tikzpicture}
$$
We will use the following shorthand notation later on: following our convention from Remark \ref{rem visualize morphism}, we denote the source and target of $\F_0$ as $\Rr$ and $\Ss$ and mark the line by~$\Aa$.
Similary, the source and target of $\F_1$ are denoted by $\Ss$ and $\Tt$ and the line marked by~$\Bb$.
Then we denote the composition suggestively by~$\Aa\circ_{\Ss} \Bb$.

Horizontal composition of two 2-morphisms for $n=2$ requires the same procedure. We add the picture for vertical composition to illustrate. The horizontal dashed lines are not part of the stratification; they just indicate the lines on which the points lie.
$$
\begin{tikzpicture}
\fill[fill =bluebg] (0,0) -- (0.3, 0) -- (0.3, 1.6) -- (0,1.6) -- cycle;
\fill[fill =redbg] (0.3, 0) -- (0.3, 1.6) -- (1,1.6) -- (1,0) -- cycle;
\draw[thick] (0.3, 0) -- (0.3, 1.6);
\draw[densely dashed, thick, red_hor] (0.3, 0.6) -- (1, 0.6)
(0.3, 1) -- (1, 1);
\draw[densely dashed, thick, blue_hor] (0, 0.6) -- (0.3, 0.6)
(0, 1) -- (0.3, 1);
\path (0.3, 0.6) node[dot] {} (0.3, 1) node[dot] {};

\draw[densely dotted] (1, 0.6) -- (2, 0.8) -- (1, 1)
(1,1.6) -- (2, 1.4)
(1,0) -- (2, 0.2);

\begin{scope}[shift={(2, 0.2)}]
\fill[fill =bluebg] (0,0) -- (0.3, 0) -- (0.3, 1.2) -- (0,1.2) -- cycle;
\fill[fill =redbg] (0.3, 0) -- (0.3, 1.2) -- (1,1.2) -- (1,0) -- cycle;
\draw[thick] (0.3, 0) -- (0.3, 1.2);
\path (0.3, 0.6) node[dot] {};
\end{scope}

\draw[densely dotted] (3, 0.2) -- (4, 0.3)
(3, 1.4) -- (4, 1.3);

\begin{scope}[shift={(4, 0.3)}]
\fill[fill =bluebg] (0,0) -- (0.3, 0) -- (0.3, 1) -- (0,1) -- cycle;
\fill[fill =redbg] (0.3, 0) -- (0.3, 1) -- (1,1) -- (1,0) -- cycle;
\draw[thick] (0.3, 0) -- (0.3, 1);
\path (0.3, 0.5) node[dot] {};
\end{scope}
\end{tikzpicture}
$$

\paragraph{Higher morphisms}
So far we have only discussed the objects and $k$-morphisms for $k\leq n$, all of which are described by factorization algebras on certain stratified spaces. For a fixed stratified space, there is an $(\infty,1)$-category of constructible factorization algebras, where morphisms are given by morphisms of factorization algebras. It can be presented as a relative category if $\S$ is presented as a relative category.

Incorporating morphisms of the underlying stratified spaces, one obtains an $n$-fold Segal object in $(\infty,1)$-categories.

Unravelling this, if $\F$ and $\G$ are factorization algebras presenting two $n$-morphisms, first assume (by rescaling) that the stratifications are the same. Then, an $(n+ \nobreak 1)$-morphism from $\F$ to $\G$ is a morphism of the underlying factorization algebras.

Finally, if $\S$ itself is an $(\infty,N)$-category, for $N>0$, then one can use the $k$-morphisms therein to define $(n+k)$-morphisms.

\paragraph{Symmetric monoidal structure}
If $\S$ is symmetric monoidal, then the $(\infty,1)$-categories of factorization algebras are symmetric monoidal, where $(\F\otimes \G)(U)=\F(U)\otimes\G(U)$. This can be used to endow $\Alg_n(\S)$ with a symmetric monoidal structure.

\paragraph{The technical condition for existence}

We briefly recall the full technical requirements from \cite{CSThesis, CalaqueScheimbauer, JFS}.
From an $(\infty,N)$-category, one can extract a diagram of $(\infty,1)$-categories as
$$\Delta^N\ni \vec k \mapsto \S^\Box_{\vec k}= \tau_{(\infty,1)}[\Theta^{\vec k},\S].$$
Here $\Theta^{\vec k}$ is a strict higher category appearing as an object in Joyal's category $\Theta_N$.
Intuitively, for ${\vec k} = (1,\ldots, 1,0,\ldots, 0)$ with $k$ ones, the $(\infty,1)$-category $\tau_{(\infty,1)}[\Theta^{\vec k},\S]$ extracts the $k$-morphisms.

\begin{remark}\label{rem variants (op)lax}
In \cite{JFS} the notation $\S^{
\tikz[baseline=(box.base)]\node[draw,rectangle,inner sep=1pt] (box) {\tiny\rm strong};
}$ is used for the diagram $\S^\Box$.
Moreover, there are two variants of $\S^\Box$, which correspond to allowing for lax and oplax versions of bimodules, but we will not need this subtlety in this article. 
Theorem~\ref{thm existence Morita} holds for all three variants.
\end{remark}

\begin{definition}\label{technical condition}
Let $\S$ be a symmetric monoidal $(\infty,N)$-category. 
The diagram $\S^\Box$ is \emph{\ensuremath\otimes-sifted-cocomplete} if it is an $N$-fold simplicial diagram of $\otimes$-sifted-cocomplete categories and $\otimes$-sifted-cocontinuous functors.
\end{definition}

\begin{theorem}\label{thm existence Morita}
Let $\S$ be a symmetric monoidal $(\infty,N)$-category
such that $\S^\Box$ is \ensuremath\otimes-sifted-cocomplete.
Then:
\begin{enumerate}
\item The $(n+N)$-uple simplicial space $\Alg_n(\S^\Box_{\vec\bullet})_{\vec\bullet}$ satisfies the Segal condition and is complete in each variable separately. Hence it determines an $(\infty,n+N)$-category~$\Alg_n(\S)$.
\item The symmetric monoidal structure on $\S$ determines a symmetric monoidal structure on~$\Alg_n(\S)$.
\end{enumerate}
\end{theorem}

\begin{definition}
The symmetric monoidal $(\infty,n+N)$-category $\Alg_n(\S)$ is called the {\em factorization higher Morita category}.
\end{definition}

The first half of the statement is proven in \cite[Theorem 8.5 (1)]{JFS}. The proof of the second half is straightforward, given the results in \cite{JFS}. We add the short, but technical argument in Proposition \ref{prop symm monoidal structure} in Appendix~\ref{appx: sym mon}.

\begin{remark}
For $N=1$, the condition on $\S^\Box$ boils down to requiring $\S$ to be \ensuremath\otimes-sifted-cocomplete. The condition of having sifted colimits ensures that the necessary pushforwards (for face maps) exist, and the compatibility with $\otimes$ is to ensure that $\Alg_n(\S)$ has a symmetric monoidal structure.
\end{remark}

\begin{remark}\label{rem two d dictionary}
We can extend our dictionary from Section \ref{sec one d dictionary} to the higher dimensional setting:
\begin{center}\footnotesize
\begin{tabular}{|c|c|c|}
\hline 
{\bf $n$-category} & {\bf standard version} & {\bf factorization version} \\
\hline 
0-morphism & $E_n$ algebra & locally constant \\
\hline
1-morphism & bimodule of $E_{n-1}$ algebras & constructible for $\{a^1\}\times (0,1)^{n-1}$ \\
\hline
2-morphism & 
\begin{tabular}{@{}c@{}}
bimodule of bimodules\\ of $E_{n-2}$ algebras
\end{tabular}
 &
\begin{tabular}{@{}c@{}}
constructible for \\
$\{(a^1,a^2)\}\times (0,1)^{n-2} \subset \{a^1\}\times (0,1)^{n-1}$
\end{tabular}
\\
\hline \vdots & \vdots & \vdots \\
\hline
$n$-morphism & bimodule of \dots bimodules & constructible for full flag\\
\hline
$(n+1)$-morphism & bimodule map & map of factorization algebras \\
\hline
\end{tabular}
\end{center}
For objects, this dictionary is discussed in Section 4.1 of \cite{LurieTFT}; see Definition~4.1.11. 

Let us extend our sketch of the comparison from Section \ref{sec one d dictionary} to the two-dimensional setting. 
Recall from Remark \ref{rem visualize 2-morphism} the notational conventions for a 2-morphism given by a factorization algebra $\F$:
$$\includestandalone{pics_tex/pic-2-morphism_small}\, .$$
The source and target objects are given by locally constant factorization algebras $\Rr$ and $\Ss$ and therefore are $E_2$ algebras.
The factorization algebras $\Aa$ and $\Bb$ on the lines (i.e., on the two connected components of $X_1\setminus X_0$) are locally constant and therefore $E_1$ algebras.
Since they come from the factorization algebra $\F$ and so we have the structure maps for the basic disks,
we can view $\Rr$ and $\Ss$ as acting on $\Aa$ and $\Bb$ and hence can be viewed as bimodules.
Finally, the same argument as in the one-dimensional case implies that $M$ can be thought of as a bimodule for $\Aa$ and~$\Bb$.
Moreover, this structure is compatible with the actions of $\Rr$ and~$\Ss$.

Note that we will never use this putative identification in theorems or their proofs.
\end{remark}

\section{Full 2-dualizability of~\texorpdfstring{$\Alg_2(\S)$}{Alg2(S)}}
\label{sec: 2d case}

In this section we prove our main theorem just in the 2-dimensional case, with many pictures.
This case allows us to introduce all the crucial ideas and techniques --- even for the general case --- in a setting where the meaning and motivations can literally be seen;
the proof of the general case then becomes conveniently compact and, we hope, easier to understand.

In dimension two, proving our main theorem reduces to proving that
\begin{enumerate}
\item every object has a dual, and
\item every 1-morphism has a left and right adjoint.
\end{enumerate}
We will show (1) in Proposition \ref{prop duals exist} and (2) in Proposition \ref{prop adjoints exist}.

\subsection{Constructing new elements in \texorpdfstring{$\Alg_2(\S)$}{Alg2(S)}}\label{sec constructions}

To prove the theorem, we first need some techniques for constructing new objects, 1- and 2-morphisms in $\Alg_2(\S)$, which we will use later to construct the (co)evaluation and (co)unit maps.
With these techniques in hand, we prove a pair of results, Propositions \ref{prop adjoints exist} and \ref{prop duals exist}, that together imply the main theorem in the 2-dimensional case.

The objects and morphisms in $\Alg_2(\S)$ are encoded by constructible factorization algebras of a special class of stratifications, but we can produce nontrivial examples of such objects and morphisms by using more general stratifications and pushforwards.
In other words, we can take advantage of the greater flexibility of factorization methods to produce interesting constructions inside the straitjacket of this fixed class of stratifications.

\subsubsection{Diffeomorphisms}

Diffeomorphisms that preserve the stratifications are the simplest source of new objects, 1-, and 2-morphisms.

The simplest case is when there is no stratification at all.
Much as locally constant sheaves are invariants of diffeomorphism type, 
a crucial feature of locally constant factorization algebras is that they are preserved under diffeomorphisms of framed manifolds,
a feature which is explored and verified in \cite{LurieHA, AyalaFrancis}.
We record this property as a lemma for convenient referencing.

\begin{lemma}\label{lemma orientation preserving}
Let $\F$ be a locally constant factorization algebra on $(0,1)^2$. Every orientation-preserving diffeomorphism $\phi: (0,1)^2 \to (0,1)^2$ determines an equivalence of locally constant factorization algebras $\widetilde{\phi}: \F \to \phi_* \F$.
\end{lemma}

\begin{ex}\label{ex: inversion}
Let $r:(0,1)\to (0,1)$ be the reflection of the interval $x\mapsto 1-x$. Then let $inv = r\times r: (0,1)^2 \to (0,1)^2$ denote the inversion $(x,y) \mapsto (1-x,1-y)$.

Given an object of $\Alg_2(\S)$, consisting of a locally constant factorization algebra $\Rr$ on $(0,1)^2$, Lemma \ref{lemma orientation preserving} implies that $inv_* \Rr \simeq \Rr$, as $inv$ is orientation-preserving. Hence the objects represented by $\Rr$ and $inv_* \Rr$ are equivalent.

We now consider the consequences of applying inversion to a 1-morphism. For example, consider the stratification on the box $(0,1)^2$ given by the vertical bisector $\{\frac 12\} \times (0,1)$ and let $\F$ be a constructible factorization algebra on this stratified space.
Recall from Remark \ref{rem visualize morphism} that we visualize the 1-morphism determined by $\F$ by
$$
\includestandalone{pics_tex/pic-morphism} \,.
$$
If we apply inversion, then the pushforward $inv_* \F$ corresponds to
$$

\includestandalone{pics_tex/pic-morphism-op} \,.
$$
In other words, we simply swap the sides on which $\Rr$ and $\Ss$ appear,
since $inv_* \Rr \simeq \Rr$ and $inv_* \Ss \simeq \Ss$.
Disks intersecting the line are assigned the value $\Aa$. However, as the vertical axis is reflected, however, the factorization structure changes: we obtain the ``opposite'' factorization algebra $\Aa^{\op}$ to~$\Aa$.
~\hfill$\mpim$\end{ex}

Using an orientation-reversing diffeomorphism may give a non-equivalent factorization algebra.
\begin{definition}
For any locally constant factorization $\F$ on $(0,1)^2$, 
let $\F^{\rev}$ denote the pushforward $rev_* \F$, where $rev = r\times \id: (0,1)^2 \to (0,1)^2$ sends $(x,y)$ to $(1-x,y)$.
\end{definition}

This definition leads immediately to our first example, which is the simplest example arising from orientation reversal.
\begin{ex}\label{ex reverse object}
Given an object in $\Alg_2(\S)$ determined by a locally constant factorization algebra $\Rr$ on $(0,1)^2$, its {\em reverse}\footnote{See Construction 5.2.5.18 in \cite{LurieHA} for the analogous construction for $E_k$-algebras.} is the object determined by $\Rr^{\rev}$.
~\hfill$\mpim$\end{ex}

\subsubsection{Using the collapse maps}

These two examples show how we can use certain diffeomorphisms of $(0,1)^2$ to produce new objects and 1-morphisms in $\Alg_2(\S)$. To produce more interesting examples, a key tool are the following two Lemmata, which generalize these examples. They follow from unpacking the definition of constructible and the collapse-and-rescale-maps.

\begin{lemma}\label{lemma stripe}
Let $0< b<a<1$.
Assume we are given a stratification of $(0,1)^2$ which is homotopic to one  of the form $\{s_1,\ldots,s_k\}\times (0,1)$, i.e.~a cylinder on a stratification of the first $(0,1)$, and is contained in $\left(b, a\right)\times (0,1)$.
If $\F$ is a constructible factorization algebra on this stratified space, then $(\varrho^b_a\times \id)_*\F$ together with its induced stratification is a 1-morphism in $\Alg_2(\S)$.
\end{lemma}

\begin{lemma}\label{lemma cross}
Let $0< b<a <1$ and $0<d<c<1$.
Assume we are given a stratification of $(0,1)^2$ whose intersection with $(0,1) \times \left((0,d]\cup [c,1)\right)$ is homotopic to a cylinder on a stratification of the first component, and is contained in $\left(b, a\right)\times (0,1)$.
If $\F$ is a constructible factorization algebra on this stratified space, then $(\varrho^b_a\times \varrho^d_c)_*\F$ together with its induced stratification is a 2-morphism in $\Alg_2(\S)$.
\end{lemma}

The following figures depict examples of such stratifications, the left one for Lemma \ref{lemma stripe} and the right one for Lemma~\ref{lemma cross}.
\begin{center}
\includestandalone{pics_tex/pic-lemma-stripe}\hspace{1.5cm}\includestandalone{pics_tex/pic-lemma-cross}
\end{center}
We will use these lemmas to produce new 1- and 2-morphisms. Let us look at some first examples.

\begin{ex}\label{ex:twistframing}
Consider a diffeomorphism $\varphi$ on $(0,1)^2$
$$
\includestandalone{pics_tex/pic-twistobject}
$$
which amounts to 
\begin{itemize}
\item ``expanding and sliding'' the small interior of the dashed arc until it becomes the upper half of the box and 
\item ``shrinking and sliding'' the larger exterior until it becomes the lower half of the box.
\end{itemize}
Note that this diffeomorphism ``unbends'' the semicircle to the horizontal bisector of the box by taking the top strand and bending it to the right.

Since this diffeomorphism is orientation-preserving, the pushforward of a locally constant factorization algebra $\varphi_* \F$ is equivalent to $\F$ itself. In the next paragraph we will extract an interesting 1-morphism in~$\Alg_2$.

Consider how the diffeomorphism acts on the local framing --- we do not depict the images of the framings on the left, but rather the pattern of the images of standard framings.
We think of the framings as a record of how pushing forward affects the original horizontal and vertical product.
$$
\includestandalone{pics_tex/pic-twistframingobject}
$$
In a neighborhood of the left edge, the framing looks like the original framing (the ``page framing'').
In a neighborhood of the right edge, the framing is the inverted framing. Let these neighborhoods be given by $(0,b)\times(0,1)$ and $(a,1)\times(0,1)$, respectively.
Note that the framing on the middle line indicates that if we added that line as a stratification, the locally constant factorization algebra on $(0,1)$ obtained by pushforward along the horizontal projection $(x,y)\mapsto y$ is the pushforward of $\F$ along the reverse of the vertical projection, namely the map $(x,y)\mapsto 1-x$.

Applying Lemma \ref{lemma stripe} with the chosen neighborhoods, 
we obtain a non-trivial 1-morphism from $\F$ to $inv_* \F\simeq\F$. 
Starting with an object $\Rr$ in $\Alg_2(\S)$, this procedure gives a 1-morphism from $\Rr$ to itself:
$$\includestandalone{pics_tex/pic-twistframing-collapse_equal}.$$
We will denote this 1-morphism by~$\Rr^{\,\circlearrowright}$

Finally, we could have chosen to apply this procedure using a diffeomorphism $\psi$ that rotates the framing counterclockwise instead, such as $\psi=\varphi^{-1}$:
$$\includestandalone{pics_tex/pic-twistframing-collapse_equal_left}.$$
Again we obtain a non-trivial 1-morphism from $\Rr$ to itself, 
which we denote by~$\Rr^{\,\circlearrowleft}$.

Note that now the factorization algebra on $(0,1)$ obtained by pushforward along the horizontal projection to the line is the pushforward of $\Rr$ along the vertical projection itself (rather than its reverse).

We claim that these 1-morphisms are inverses (up to homotopy). Observe that the target of $\Rr^{\,\circlearrowright}$  is $inv_*\Rr$. Applying the equivalence $inv_*$ to $\Rr^{\,\circlearrowleft}$, we can compose the two 1-morphisms. The glued strip is diffeomorphic, with fixed endpoints, to the strip with the constant framing, which shows the claim.
~\hfill$\mpim$\end{ex}

\begin{ex}\label{ex:counit}
Here we will produce a nontrivial 2-morphism by pushing forward along a composite of a diffeomorphism and a collapse map.

Start again with the 1-morphism $\F$ from Example \ref{ex: inversion}, visualized as
$$
\includestandalone{pics_tex/pic-morphism} \,.
$$ 
We produce a diffeomorphism $\phi$ such that the pushforward $\phi_* \F$ is visualized~as
$$
\includestandalone{pics_tex/pic-bendright} \,.
$$
Think of $\phi$ as bending the top of the line to the right and down. One way to do this is to rotate by $\pi/2$ to put $\Rr$ in the bottom half-box, apply the inverse $\varphi^{-1}$ of the diffeomorphism $\varphi$ from Example \ref{ex:twistframing}, and then rotate by $\pi/2$ to place the small interior inside the arc at the bottom of the box.

We now apply Lemma \ref{lemma cross} to obtain a 2-morphism: we use a collapse-and-rescale map for a ``cross'' that contains the small interior to the arc. Choose $0<a,b,c,d<1$ such that the division of the box into nine pieces can be visualized as follows:
$$
\raisebox{-0.45cm}{\includestandalone{pics_tex/pic-bendright_collars}} \,.
$$

Now consider the collapse-and-rescale map $\varrho^b_a\times \varrho^d_c$. Note that it
\begin{itemize}
\item collapses the innermost square to a point, 
\item compresses the middle box on the top row to a vertical blue line (i.e., it projects away the horizontal direction),
\item compresses the middle box on the bottom row to a vertical blue line (i.e., it projects away the horizontal direction),
\item likewise compresses the middle boxes on the left and right sides to horizontal dashed blue lines (by projecting away the vertical directions),
\item acts as the identity on the corner squares, and
\item rescales back to $(0,1)^2$.
\end{itemize}
If we pushforward along this map, by Lemma \ref{lemma constructible pushforward} our factorization algebra remains constructible and can be visualized as
$$
\includestandalone{pics_tex/pic-bendright_collapse} \,.
$$
By Lemma \ref{lemma cross}, this constructible factorization algebra determines a 2-morphism in $\Alg_2(\S)$ from a 1-morphism $\Bb$ (built from $\Rr$, $\Ss$, and $\Aa$ via Lemma \ref{lemma stripe}) to the identity 1-morphism on $\Rr$ (given by $\Rr$ viewed as living on the upper half of the vertical bisector) .
Note that we do not view the horizontal dashed blue line as contributing to the stratification for this constructible factorization algebra.

Let us compute the source $\Bb$. It is produced via pushforward by compressing a strip glued from several more elementary pieces that we have already encountered above:
$$
\includestandalone{pics_tex/pic-bendright-strip-composition}\, .
$$
The first piece is just the 1-morphism itself. 
We considered the second and last pieces in Example \ref{ex:twistframing} and the third piece in Example~\ref{ex: inversion}. 
This composition is
$$\Aa \circ_{\Ss} \Ss^{\,\circlearrowright} \circ_{\Ss} \Aa^{op} \circ_{\Rr} \Rr^{\,\circlearrowleft}$$
in our shorthand notation.~\hfill$\mpim$\end{ex}

\subsubsection{Fold maps}\label{sec folding}
Another way of producing new 1-morphisms is by ``folding''. (We could also produce new 2-morphisms this way, but we will not need it.)

Start with a 1-morphism given by a factorization algebra $\F$ which we depict as in Example \ref{ex: inversion}.
Now consider ``folding up and over'' the right edge of the square in 3-dimensional space as shown, such that the stratification is exactly at the right:
$$
\includestandalone{pics_tex/pic-fold_morphism} \, .
$$
More precisely, consider the embedding $\tilde f: (0,1) \hookrightarrow (0,1) \times \RR$ sending $x$ to $\left(\frac12\sin(\pi x),\cos(\pi x)\right)$,
which bends the right edge of the plane up and over to the left edge. We will use its product with the identity on $(0,1)$, as pictured above.

Now we post-compose the embedding $\tilde f$ with the projection $(0,1) \times \RR \to (0,1)$, which forgets the second direction.
We will call the composite map the ``fold map''\footnote{When applied to an identity 1-morphism this terminology matches with Lurie's definition in Construction 4.6.3.7 in \cite{LurieHA}.} and denote it suggestively by~$f$.

The pushforward $(f\times \id_{(0,1)})_*\F$  determines a 1-morphism in $\Alg_2(\S)$, visualized~by
$$
\includestandalone{pics_tex/pic-folded_morphism}
$$
This 1-morphism $\Aa^{>}$ goes from $\Ss^{\rev}\otimes \Rr$ to~$\unit$.

Similarly, for a 1-morphism given by a factorization algebra $\G$ with source $\Ss$ and target $\Tt$, 
we can fold ``up and over'' the left edge of the square in 3-dimensional space and project back to the plane using a fold map
\[
\begin{tikzcd}[row sep=tiny]
g:&[-25pt] (0,1) \arrow[hookrightarrow] {r} & (0,1)\times  \R \arrow[->>]{r} & (0,1), \\ 
& x \arrow[|->]{rr} & & 1-\frac12\sin(\pi x)
\end{tikzcd}
\]
to obtain a 1-morphism $\Bb^<$ from $\unit$ to $\Tt\otimes \Ss^{rev}$.
It looks like
$$
\includestandalone{pics_tex/pic-morphism2} \hspace{1cm}\includestandalone{pics_tex/pic-foldright} \hspace{1cm} \includestandalone{pics_tex/pic-folded_right_morphism} \, .
$$
We would like to compose these 1-morphisms, so we first tensor them with identity morphisms on $\Tt$ and $\Rr$ and then compute
$$\left(\Aa^> \otimes \Rr \right) \circ \left(\Tt\otimes \Bb^<\right)$$
to obtain a morphism from $\Rr$ to $\Tt$.
Pictorially, consider the projection onto the plane of a zigzag obtained by gluing the folded strips:
$$
\includestandalone{pics_tex/pic-snake_composition}\, .
$$
This composition of 1-morphisms comes from the pushforward along this projection.

Observe that there is a diffeomorphism between the snaking surface above and a multicolored square $(0,1)^2$ as embedded manifolds in the 3-cube.
Simply pull along the ends as directed by the oscillating arrows:
$$
\includestandalone{pics_tex/pic-snake_pull}\, .
$$
This diffeomorphism shows that after applying the collapse maps to get the composition of the morphisms, the compositions will be equivalent:
\begin{equation}\label{eqn fold}
\left(\Aa^> \otimes \id_{\Rr} \right) \circ_{\Tt\otimes \Ss^{rev} \otimes{\Rr}} \left(\id_{\Tt} \otimes \Bb^< \right) \simeq \Aa \circ_{\Ss} \Bb.
\end{equation}
Similarly, we could have folded the right edge of the plane for $\Aa$ down (instead of up) and to the left edge; and similarly for $\Bb$. 
Gluing these would produce a snake that is the mirror in the vertical direction of the one we have already drawn. 
But when composing, that direction is projected away, so we get the same morphisms, and
\begin{equation}\label{eqn fold2}
\left(  \id_{\Rr} \otimes \Aa^> \right) \circ_{\Rr \otimes \Ss^{rev} \otimes \Tt} \left( \Bb^< \otimes \id_{\Tt} \right) \simeq \Aa \circ_{\Ss} \Bb.
\end{equation}

\begin{remark}\label{rem source eval}
We note here that the source of $\Aa^>$ is not $\Ss^{\rev}\otimes\Rr$ on the nose, 
but rather a rescaled version: 
it is the pushforward of  $\Ss^{\rev}\otimes\Rr$ along the map $(0,1)\to (0,1), x\mapsto \sin(\frac12 \pi x)$. 
Since this map is invertible and orientation-preserving, 
pushing forward along it determines an equivalence of the locally constant factorization algebras. 
We chose to suppress the equivalence from the notation. 
A similar issue appears for $\Bb^<$ and its target~$\Tt\otimes \Ss^{rev}$.
\end{remark}

\begin{remark}\label{rem creasing}
There is another approach that does work on the nose: 
choose a different fold map that is piecewise linear and arises from ``creasing'' at the fold. 
This would amount to modifying the fold maps so that
\[
\begin{array}{cccc}
f_1: & (0,1) &\to & (0,1) \\
& x &\mapsto &
\begin{cases}
x, & 0\leq x\leq \frac12\\
1-x, & \frac12 \leq x\leq 1\\
\end{cases} 
\end{array}
\]
and
\[
\begin{array}{cccc}
g_1:& (0,1) & \to & (0,1),\\
& x & \mapsto &
\begin{cases}
1-x, & 0\leq x\leq \frac12\\
x, & \frac12 \leq x\leq 1
\end{cases}
\end{array}.
\]
We find the mental image of ``stretching the snake'' more intuitive, however,
which is the reason for our choices above.
\end{remark}

\subsection{Duals for objects}
This proposition contains the first half of the proof of the 2-dimensional main theorem.

\begin{prop}\label{prop duals exist}
Every object in $\Alg_2(\S)$ has a dual.
\end{prop}

\begin{proof}
Pick an object in $\Alg_2(\S)$, represented by a locally constant factorization algebra $\Rr$ on $(0,1)^2$. 
A dual will be provided by $\Rr^{\rev}$ from Example \ref{ex reverse object}, 
namely the pushforward of $\Rr$ along the reflection map in the first coordinate of $(0,1)^2$, namely $rev=r\times \id: (x,y) \mapsto (1-x,y)$. 
We need to exhibit the evaluation and coevalution maps and verify they satisfy the zigzag identities.

In a first step, we consider a 1-morphism, which we still denote by $\Rr$, whose underlying factorization algebra still is $\Rr$, but adding the stratfication given by a line in the middle. (Note that technically this is not the identity on the object, but it is equivalent to it.)

We apply the construction from Section \ref{sec folding}. 
Then $\mathrm{ev} = \Rr^{\rotatebox{-90}{$\curvearrowleft$}}$ is an evaluation and $\mathrm{coev} = \Rr^{\rotatebox{90}{$\curvearrowleft$}}$ is a coevaluation morphism for $\Rr$, as is exhibited by the identities \eqref{eqn fold} and~\eqref{eqn fold2}:
\begin{align*}
(\id_{\Rr}\otimes \operatorname{ev}) \circ (\operatorname{coev}\otimes \id_{\Rr}) & \simeq \id_{\Rr},\\
( \operatorname{ev} \otimes \id_{\Rr}) \circ (\id_{\Rr} \otimes \operatorname{coev})& \simeq \id_{\Rr},
\end{align*}
which is what we needed to show.
\end{proof}

\begin{remark}\label{rem dual subtleties}
Note that in light of Remark~\ref{rem source eval}, 
the source of the evaluation is only equivalent to $\Rr^{\rev}\otimes \Rr$ and the target of the coevaluation is only equivalent to $\Rr\otimes \Rr^{\rev}$. 
But ``having a dual'' is preserved under equivalences. We could just compose the exhibited (co)evaluation with the equivalence to see that our construction suffices. 
An alternative modification would be to use the fold maps $f_1$ and $g_1$ from Remark~\ref{rem creasing}.
More generally, there is a space of possible fold maps with which to construct an evaluation and coevaluation pair. 
Similarly, there is a space of orientation-reversing maps to use in place of $r$ in constructing the putative dual $\Rr^{\rev}$. 
The space of dualizability data is contractible (see \cite[Example 2.4.12]{LurieTFT}), however, so it suffices to exhibit just one triple of dualizability data.
\end{remark}

\subsection{Adjoints for 1-morphisms}

The remaining step in proving the 2-dimensional main theorem is to show the existence of left and right adjoints, which is the content of the following Proposition.
We will use the constructions from Section \ref{sec constructions} to construct (co)unit maps. 

\begin{remark}
The pictures we use have a strong resemblance to string diagrams, which should be no surprise.
The very set-up of constructible factorization algebras allows one to read such diagrams as the stratification of a constructible factorization algebra,
and we used that perspective in constructing $\Alg_2(\S)$ itself.
\end{remark}

\begin{prop}\label{prop adjoints exist}
Every morphism in the $(\infty,2)$-category $\Alg_2(\S)$ has a left and a right adjoint.
\end{prop}

\begin{proof}
We need to show that any 1-morphism $\F$ in $\Alg_2(\S)$ has both a left and a right adjoint.
It is convenient to use the visual notation from the preceding subsection, 
where we describe $\F$ as having a component $\Aa$ supported on the vertical bisector going from $\Rr$ on the left to $\Ss$ on the right:
$$
\includestandalone{pics_tex/pic-morphism}\, .
$$
The examples just given will play a crucial role in constructing the adjoints.

We start by proving the existence of the left adjoint $L_\F$. 
We posit that the 1-morphism corresponding to
$$
\includestandalone{pics_tex/pic-morphism-adjoint} \, 
$$
provides the left adjoint (note that the framing technically is not part of the data of a 1-morphism, but we can apply Lemma \ref{lemma stripe}).
Our ansatz is that the 2-morphism constructed in Example \ref{ex:counit} will provide the counit of the adjunction, so we will reverse-engineer $L_\F$ by unpacking that construction.

The source of the counit 2-morphism is encoded in the bottom part,
which is given by the composite $L_\F \circ_\Ss \F$.
This composite $L_\F \circ_\Ss \F$ arises, in turn, from the pushforward map determined by the map $\varrho^b_a\times \varrho^d_c$.
(More accurately, by how that map behaves on the bottom part.)
Hence, we can read off the proposed left adjoint by looking at the right side of the bottom half of the picture before taking the pushforward. 

Visually, we are working with
$$
\includestandalone{pics_tex/pic-bendright_framing}\,
$$
and it is essential to keep track of the framing that determines the diffeomorphism used in the Lemma.
The source of the counit 2-morphism arises by compressing the bottom part, with its framing determined by Lemma~\ref{lemma stripe}:
$$
\includestandalone{pics_tex/pic-strip_morphism_adjoint} \, .
$$
Note that our proposal for the left adjoint is the 1-morphism corresponding to
$$
\includestandalone{pics_tex/pic-morphism-adjoint} \, ,
$$
the right half of that strip.

The unit 2-morphism of the adjunction is given by a similar ansatz, 
but now the interesting 1-morphism is in the upper half.
Thus we replace the diffeomorphism $\phi$ of Example \ref{ex:counit} with 
a diffeomorphism $\psi$ that  ``bends'' the lower end of the line to the left.
$$
\includestandalone{pics_tex/pic-bendleft}
$$
Now it is the framing along the top that is interesting, and the rotation now happens in the first half. 
It rotates clock-wise until the stratum, where the framing is rotated by $\pi$, and then it rotates back counter-clockwise:
$$
\includestandalone{pics_tex/pic-strip_morphism_adjoint2}
$$
Note that the left hand side is the same as the $L_\F$ proposed already.

To obtain the unit 2-morphism, we apply Lemma \ref{lemma cross} and push forward the factorization algebra along a collapsing map $\varrho$, which sends everything between the following strips along the edges to a point:
$$
\includestandalone{pics_tex/pic-bendleft_collars} \, .
$$

It remains to verify that our 2-morphisms satisfy the zigzag identities,
which boils down to examining the following pictures.
Fix diffeomorphisms $\sigma, \varsigma$ sending the stratifications on the left to those on the right (with ``page framing''): 
$$
\includestandalone{pics_tex/pic-adjunction1}
$$
$$
\includestandalone{pics_tex/pic-adjunction2}
$$
Since pushforward along the collapse maps commutes with gluing of factorization algebras (respectively, composition of 2-morphisms in $\Alg_2(\S)$), these invertible diffeomorphisms determine a weak equivalence of the necessary compositions in the zigzag identity (left pictures) to the identity bimodule (right pictures). 

For the right adjoint, the argument is similar, with some slight variations.
The right adjoint arises from the picture
$$

\includestandalone{pics_tex/pic-morphism-rightadjoint} \, .
$$
The only difference is the direction of rotations of the framings. In this case the counit arises from bending the top of the line of a 1-morphism to the left (instead of previously to the right), and the unit is given by bending the bottom of the line to the right (instead of previously to the left). The pictures for the adjunction are similar, but with the red and blue colorings reversed.
\end{proof}

\section{The general result}
\label{sec: arbitrary dim}

The generalization to arbitrary dimensions is straightforward: 
we perform essentially the same manipulations as in the proofs for~$\Alg_2(\S)$ to obtain the following result.

\begin{theorem}\label{thm main theorem}
Let $\S$ be a symmetric monoidal $(\infty,N)$-category and $\S^\Box$ is \ensuremath\otimes-sifted-cocomplete (cf.~Definition \ref{technical condition}). The symmetric monoidal $(\infty,n+N)$-category $\Alg_n(\S)$ is fully $n$-dualizable,~i.e.
\begin{enumerate}
\item\label{dual} every object in $\Alg_n(\S)$ has a dual; and
\item\label{adjoint} if $1\leq k < n$, any $k$-morphism in $\Alg_n(\S)$ has both a left and a right adjoint.
\end{enumerate}
\end{theorem}

\begin{remark}
Recall from Lemma \ref{lem truncation detects dualizability} that full $n$-dualizability is detected in the $(\infty,n)$-truncation. As the $(\infty,n)$-truncations of all three variants of the factorization higher Morita category from \cite{JFS} agree (see Remark \ref{rem variants (op)lax}), the theorem is true for all three variants.
\end{remark}

In the proof we will need notions and notation recalled at the beginning of Section~\ref{sec: fact Alg_n},
particularly stratifications given by ``affine flags''of the form
\begin{align*}
\{(a^1,\ldots, a^k)\}\times &(0,1)^{n-k} \subset \{(a^1,\ldots, a^{k-1})\}\times (0,1)^{n-k+1} \subset \cdots \\ 
&\cdots \subset \{a^1\}\times (0,1)^{n-1} \subset (0,1)^n.
\end{align*}
To give such a flag, we must specify a point $a^i \in (0,1)$ for each $1 \leq i \leq k$ that says how to split each of the first $k$ coordinates into two pieces.

The reader has already seen how this kind of stratification played a role in the 2-dimensional case,
and so should readily recognize that analogous stratifications would play a role in higher dimensions.
If the reader simply keeps in mind this idea of working with such constructible factorization algebras,
the proof should be intelligible,
but we will indicate where the auxiliary data is used so that the reader can fill in all details.

\begin{proof} 
To show claim \eqref{dual}, pick an object in $\Alg_n(\S)$, represented by a locally constant factorization algebra $\Rr$ on $(0,1)^n$.
Note that the auxiliary data for this object is the $n$-tuple of intervals $\oul{I}=(I^i_0)_{1\leq i\leq n}$ given by $I^i_0=(0,1)$ for every $1\leq i\leq n$. 
This data determines the empty stratification of~$(0,1)^n$.

We claim that a dual is given by $\Rr^{\rev}$, 
the pushforward of $\Rr$ along the map $\bar r=r\times \id_{(0,1)^{n-1}}:(0,1)^n\to (0,1)^n$,
which simply reverses the first coordinate $x \mapsto 1-x$. 
The auxiliary data of the dual object is the same as for $\Rr$, and hence specifies the empty stratification.

Both the evaluation and coevaluation are obtained via ``fold maps,'' as in Proposition \ref{prop duals exist}. 
For the evaluation, consider the product $F= f\times \id_{(0,1)^{n-1}}$, 
which acts by the fold map $f$ of Section \ref{sec folding} and by the identity on the remaining coordinates. 
An evaluation 1-morphism is given by the factorization algebra $F_*\Rr$ on $(0,1)^n$,
along with the auxiliary data of the intervals $I^1_0 =(0,\frac12]\leq [\frac12, 1)=I^1_1$ and $I^i_0=(0,1)$ for every $2\leq i \leq n$. 
The data of these intervals determines the stratification of $(0,1)^n$ by the hyperplane $\{x_1=\frac12\}$, and $F_*\Rr$ is constructible on this stratified space.

For the coevaluation, we use the other fold map $g$ and consider the product $G = g\times \id_{(0,1)^{n-1}}$. 
A coevaluation 1-morphism consists of the factorization algebra $G_*\Rr$ on $(0,1)^n$ together with the same auxiliary data as the evaluation, namely the intervals $I^1_0 =(0,\frac12]\leq [\frac12, 1)=I^1_1$ and $I^i_0=(0,1)$ for every~$2\leq i \leq n$.\footnote{As we pointed out in Remarks \ref{rem source eval}  and \ref{rem dual subtleties} for the 2-dimensional setting, 
the source of the evaluation is isomorphic on the nose to $\Rr^{\rev}\otimes \Rr$ (and similarly for the target of the coevaluation). 
It is, however, weakly equivalent, by the same argument as in earlier remarks.  
Moreover, since ``having a dual'' is preserved under equivalences, our construction suffices. 
The alternative construction of co/evaluations, given in Remark \ref{rem creasing}, would also work here.
More generally, there is a space of fold maps and associated auxiliary data: 
for example, we could have chosen the stratification to be given by any hyperplane $\{x_1=c\}$ and the fold map such that the fold is at $c$ instead of $\frac12$. 
Each of these would create a different evaluation and coevaluation pair. 
It suffices to exhibit one triple of dualizability data, however,
since the space thereof is known to be contractible.}

To show claim \eqref{adjoint}, we give the argument for the left adjoint, as the case of the right adjoint is a simple variation.

We note first that if a $k$-morphism is equivalent to an identity morphism on some $k'$-morphism for $k'<k$. Thus, it is invertible (up to an invertible higher morphism) and the inverse provides both a left and a right adjoint.
Hence it suffices to deal with the case of noninvertible $k$-morphisms.

We now remark that a noninvertible $k$-morphism will always have auxiliary data that specifies a nontrivial stratification
\begin{multline*}
\{(a^1,\ldots, a^k)\}\times (0,1)^{n-k} \subset \{(a^1,\ldots, a^{k-1})\}\times (0,1)^{n-k+1} \subset \cdots \\ \subset \{a^1\}\times (0,1)^{n-1} \subset (0,1)^n
\end{multline*}
so we will show claim (2) for a factorization algebra $\F$ constructible with respect to such a stratification.
(Note that such a $k$-morphism may nonetheless be invertible 
--- suppose the factorization algebra is actually locally constant, so that the stratification is invisible to it --- 
but our argument will still apply.)

Before giving the general argument,
we note that the 2-dimensional situation provides useful intuition. 
Consider the projection $p_{\{k,k+1\}}: (0,1)^n \to (0,1)^2$ onto the $k$th and $(k+1)$st coordinate.
The pushforward $(p_{\{k,k+1\}})_*\F$, along with the induced stratification $\{a_k\}\times (0,1)$, determines a 1-morphism in $\Alg_2(\S)$.
The argument of Proposition \ref{prop adjoints exist} then applies and produces a left adjoint in $\Alg_2(\S)$. 
We will use similar constructions here, but extended to $n$ dimensions.

It will be convenient for our general argument to borrow notation from that proof.
Recall that we produced the counit 2-morphism using two maps on the square.
There was a diffeomorphism $\phi$ 
that takes a square $(0,1)^2$ divided into two equal halves by a vertical line 
and maps it to a square where the right half is squashed into a half-disk along the bottom of the square.
(For a detailed description of $\phi$, see the initial discussion of Example \ref{ex:counit}.) 
There was also a collapse map $\varrho$ that squashes two crossing strips
(i.e., the thickened neighborhood of a cross formed by the union of a vertical and horizontal bisector of the square) 
to a cross.\footnote{To be more explicit, we note that the collapse map has the form $\varrho_a^b \times \varrho_c^d$, 
where $\varrho_s^t: (0,1) \to (0,1)$ is the 1-dimensional collapse map specified in Definition \ref{def: collapse}.
We are free to chose $a< b$ and $c< d$ as we wish,
but then we must construct $\phi$ so that the image of the right half lies inside the region collapsed by $\varrho$. 
Hence there is a space of choices of suitable pairs $(\phi,\varrho)$.}
(Look at the third figure for Example \ref{ex:counit}.) 
The counit itself arises by pushing forward along $\phi$ and then pushing forward along $\rho$.
To produce the unit 2-morphism,
we used a similar diffeomorphism $\psi$ that pushes the left half into a half-disk along the top of the square, as well as a collapse map.
(For a picture of $\psi$, see where the construction of the unit 2-morphism starts in the proof of Proposition~\ref{prop adjoints exist}.)

We now undertake the $n$-dimensional case. 
Our goal is to exhibit an adjunction triple: 
a left adjoint $k$-morphism, a unit $(k+1)$-morphism, and a counit $(k+1)$-morphism that satisfy the appropriate zigzag relations. 
We will construct such a triple using certain choices,
but the space of adjunction data is  contractible, 
so exhibiting a particular triple suffices.

To obtain the counit of the adjunction, 
we construct a map
\[
\bar\phi= \id_{(0,1)^{k-1}} \times \phi \times \id_{(0,1)^{n-k-1}}
\] 
and a compatible collapse map 
\[
\bar\varrho = \id_{(0,1)^{k-1}} \times \varrho \times \id_{(0,1)^{n-k-1}}.
\]
Let $(\alpha,\beta) \in (0,1)^2$ denote the intersection point of the cross onto which the map $\varrho$ collapses the crossing strips.
The pushforward factorization algebra $(\bar\varrho \circ \bar\phi)_* \F$ is constructible with respect to a stratification of the form
\begin{align*}
\{(a^1,\ldots, &\alpha, \beta)\}\times (0,1)^{n-k-1} \subset \{(a^1,\ldots, \alpha)\}\times (0,1)^{n-k}  \\ 
&\subset\{(a^1,\ldots, a^{k-1})\}\times (0,1)^{n-k+1} \subset \cdots \subset \{a^1\}\times (0,1)^{n-1} \subset (0,1)^n.
\end{align*}
Note that we increased the depth of the stratification by one in the $(k+1)$st direction.

Similarly, to obtain the unit of the adjunction, we construct a map
\[
\bar\psi= \id_{(0,1)^{k-1}} \times \psi \times \id_{(0,1)^{n-k-1}}
\]
compatible with $\bar\varrho$ and then work with the pushforward  $(\bar\varrho \circ \bar\psi)_* \F$ and the same stratification.

The zigzag identities follow by the same argument as in Proposition~\ref{prop adjoints exist}, the 2-dimensional case.
One side of a zigzag identity (e.g., $L \Rightarrow L \circ R \circ L \Rightarrow L$) 
involves a composition of 1- and 2-morphisms,
and in our setting, this composition is determined by a gluing of stratified spaces.
There is an obvious gluing of the stratified spaces that appear before applying the collapse maps,
where one see a stratification diffeomorphic to the standard one for a $k$-morphism.
Pushforward along the collapse maps commutes with composition of morphisms,
so we obtain the zigzag identity.
\end{proof}

\begin{remark}
The reader might wonder why claim (2) of the Theorem holds for $k<n$ but not for $k = n$. 
(In fact, in the next section we will see that $k=n$ holds only in a very special case.)
It is quick to see why the arguments do not extend to $k = n$: 
our arguments use geometric manipulations of bending and folding to produce to unit and counit $(k+1)$-morphisms.
These manipulations increase the depth of the stratification,
which requires $k<n$ so that we have an extra direction in $(0,1)^n$ within which to work.
For $k=n$, by contrast, the $(k+1)$-morphisms are given by morphisms of bimodules and hence have a rather different flavor.
\end{remark}

\section{Pointings prevent \texorpdfstring{$(n+1)$}{(n+1)}-dualizability}\label{sec pointings}

The main result of this section is the following.\footnote{As noted earlier, Theo Johnson-Freyd suggested this claim during a collaboration with the second author. See Section 7 of \cite{JFHeis} for his perspective on this result and a sketch of a different argument for it.}

\begin{theorem}\label{thm pointing}
An $(n+1)$-dualizable object in $\Alg_n(\S)$ is equivalent to the unit locally constant factorization algebra on~$(0,1)^n$.
\end{theorem}

In other words, the only $(n+1)$-dualizable object is the unit object.
Under the dictionary with $E_n$ algebras, 
an $(n+1)$-dualizable object is equivalent to the unit $E_n$ algebra.
For example, let $n =1$ and work with $\kk$-algebras over some ordinary commutative algebra $\kk$. 
In this pointed version of the classic Morita $2$-category, the result says that the only 2-dualizable algebra is $\kk$ itself.
Contrast this situation with the observation of Lurie that for the usual Morita 2-category,
the 2-dualizable objects are the separable $\kk$-algebras whose underlying module is finitely-generated and projective over~$\kk$.

As the proof will show, the key reason for this very strong result is the pointedness of the $n$-morphisms in $\Alg_n(\S)$.
Hence, this theorem indicates that it would be more fruitful to work with various ``unpointed'' Morita categories, as discussed in Section~\ref{sec pointed comment} of the introduction.

This result is a consequence of a more technical result of independent interest.

\begin{prop}\label{thm pointed adjoint implies invertible}
If an $n$-morphism in the symmetric monoidal $(\infty,n+1)$-category $\Alg_n(\S)$ has an adjoint, then it is an equivalence.
\end{prop}

Recall from the theory of bicategories, the following useful notion about adjoint equivalences.
It plays a key role in proving the proposition and also in deducing the theorem from this proposition.

\begin{lemma}\label{lem: inverses from adjoints}
In a bicategory $\mathcal{B}$, let $L$ and $R$ be 1-morphisms that form an adjunction exhibited by unit and counit 2-morphisms $u$ and $c$. 
If $u$ and $c$ are invertible, then they exhibit $L$ and $R$ as inverse to each other.
\end{lemma}

\begin{proof}[Proof of Theorem \ref{thm pointing}]
Let $\Rr$ be an object in $\Alg_n(\S)$ that is $(n+1)$-dualizable. 
In particular, it admits a dual with evaluation and coevaluation 1-morphisms, 
which in turn have left and right adjoints, 
in a process extending to all levels of morphisms, as we have seen in Theorem~\ref{thm main theorem}. 
Toward the top, each $n$-morphism that appears in this process (by exhibiting an adjunction of the relevant $(n-1)$-morphisms) has an (left and right) adjoint. 
By Proposition \ref{thm pointed adjoint implies invertible}, these $n$-morphisms are invertible. Since these $n$-morphisms are the unit and counits of adjunctions for certain $(n-1)$-morphisms, 
Lemma \ref{lem: inverses from adjoints} shows that these $(n-1)$-morphisms are invertible as well. 
Working down morphism-level by induction, 
we see that the evaluation and coevaluation maps are invertible. 
But this process exhibits the dual of $\Rr$ as being an inverse of $\Rr$.
As the evaluation map is invertible, we see that $\Rr \otimes \Rr^{\rev}$ is Morita equivalent to the unit $\unit$.
We know something more, however: 
by Lemma \ref{lem: equiv in S} we know that $\Rr$, 
viewed as the evaluation map and hence as an $(\Rr^{\rev}\otimes\Rr, \unit)$-bimodule, 
is equivalent to $\unit$ as an element of $\S$.
This equivalence is via the pointing of $\Rr$ by the unit element $1_\Rr: \unit \to \Rr$.
This pointing is forgotten down from the category of factorization algebras. 
As this forgetful functor detects equivalences,
we see that $\Rr$ is equivalent to $\unit$ as a factorization algebra.
\end{proof}

We sharpen Proposition \ref{thm pointed adjoint implies invertible} by proving a stronger statement. From here on, we will use the language of algebras and bimodules, but at this point the reader can convert to the language of objects and morphisms in the factorization Morita category.

\begin{prop}
\label{thm: pointed adjoints are equivalences}
Let $M$ be a $(A,B)$-bimodule equipped with the pointing $m_0: \unit \to M$,
which is a morphism in $\S$. 
If $(M, m_0: \unit \to M)$ possesses a left adjoint $(N,n_0: \unit \to N)$ among pointed bimodules, then the unit and counit maps are invertible.
Hence $A$ and $B$ are equivalent via the adjunction $(N,n_0) \dashv (M, m_0)$.
Moreover, the unit and counit maps of this adjunction provide equivalences $A \simeq M \simeq N \simeq B$ in~$\S$.
\end{prop}

This proposition says that $A$ and $B$ are Morita equivalent in the sense of this pointed Morita category, 
which is a more restrictive condition than traditional Morita equivalence.
It also ensures that  if $(M, m_0)$ possesses a right adjoint, then $A$ and $B$ are Morita equivalent and all the objects appearing as data in the adjunction are equivalent in~$\S$.

To build toward the proof of this proposition,
we recall some key facts and make some simple observations.

Recall that in $(\infty,2)$-category $\Alg_1(\S)$, 
the algebra $A$ is pointed by its unit element $1_A : \unit \to A$.
The $A^{\op}\otimes A$-module $A$ that exhibits the algebra $A$ as 1-dualizable is naturally pointed by $1_A$ as well.

A key feature is that the pointings induce a lot of extra maps.
For example, observe that a pointing such as $m_0$ determines a map ${m_0^A}: A \to M$ of left $A$-modules by the composite
\[
A \simeq A \otimes \unit \xto{\id_A \otimes m_0} A \otimes M \xto{\star_{A,M}} M
\]
where $\star_{A,M}$ denotes the action of $A$ on $M$.
Likewise, there are maps ${m_0^B}: B \to M$, ${n_0^A}: A \to N$, and ${n_0^B}: B \to N$,
determined by the pointing and the module structures.
Similarly, we have a canonical map $M \to M \otimes_B N$ via the composite
\[
M \simeq M \otimes \unit \xto{\id_M \otimes n_0} M \otimes N \to M \otimes_B N,
\]
which we denote by $\id_M \otimes_B n_0$.
We will use a similar style to indicate similar maps produced by combining pointings with bimodule structures.

We now turn to verifying some surprising properties that pointings imply in the context of this putative adjunction.

As usual we use $u: B \longrightarrow N \otimes_A M$ to denote the unit of the adjunction and $c: M\otimes_{B} N \longrightarrow A$ to denote the counit.
The zigzag maps using $u$ and~$c$
\begin{align}
M \longrightarrow & M\otimes_B N \otimes_A M \longrightarrow M \label{eq M snake}\\
N \longrightarrow & N\otimes_A M \otimes_B N \longrightarrow N
\end{align}
are both equivalent to the identity maps.
Note that we have written these zigzags to emphasize that they are the usual formulas for adjunctions.
The composition (\ref{eq M snake}) can be unpacked into a sequence of maps in $\S$ as follows:
\[
M \simeq M \otimes \unit \xto{\id \otimes 1_B} M \otimes B \xto{\id \otimes u} M \otimes N \otimes_A M \to  M\otimes_B N \otimes_A M \xto{c \otimes_A \id} A \otimes_A M \simeq A ,
\]
which will be useful below.

There are strong compatibilities between these (co)unit maps and the pointings.
For instance, 
the composite $\unit \xto{m_0 \otimes_B n_0} M \otimes_B N \xto{c} A$ is equivalent to the identity element $\unit \xto{1_A} A$
since maps must preserve the pointing.
Similarly, the composite $\unit \xto{1_B} B \xto{u} N\otimes_A M$
is equivalent to the pointing $\unit  \xto{n_0 \otimes_A m_0} N \otimes_A M$.

These unit and counit maps also lead to further interesting maps,
such as the composite
\[
M \xto{\id_M \otimes_B n_0} M \otimes_B N \xto{c} A.
\]
This composition, which we denote $a_M$, is an inverse to $m_0^A$, as we now show.

\begin{lemma}
\label{lem: equiv in S}
The map ${m_0^A}: A \to M$ is an equivalence in $\S$, 
 as are the maps $m_0^B$, $n_0^A$, and~$n_0^B$.
\end{lemma}
One immediate consequence is that $A$ and $B$ are equivalent as objects in~$\S$.

For another consequence, 
note that $m_0^A$ is a map of left $A$-modules, 
by its construction.
Hence we know it is an equivalence of left $A$-modules, 
as the forgetful functor from left $A$-modules to $\S$ detects equivalences.
Analogous arguments apply to the other maps,
so we obtain:

\begin{cor}
The map $m_0^A$ is an equivalence of left $A$-modules, 
the map $m_0^B$ is an equivalence of right $B$-modules, 
the map $n_0^A$ is an equivalence of right $A$-modules, and 
the map $n_0^B$ is an equivalence of left $B$-modules.
\end{cor}

\begin{proof}[Proof of Lemma \ref{lem: equiv in S}]
We give the proof for $m_0^A$ as quite similar arguments imply the other cases.

First, consider the morphism $1_A: \unit \to A$ in $\S$ that points $A$ as an $(A,A)$-bimodule, 
and hence as a left $A$-module.
If we apply the functor $A \otimes -$, 
which is the left adjoint of the free-forget adjunction between left $A$-modules and $\S$, 
we obtain a morphism $A \otimes 1_A: A \to A$, 
which is manifestly the identity map~$\id_A$.

Now, observe that the map $a_M$ is a map of left $A$-modules, 
as its constituents are.
Hence the composite $a_M \circ m_0^A: A \to A$ is a map of left $A$-modules,
and since it preserves the pointings, 
we see that the composite $a_M \circ m_0^A \circ 1_A: \unit \to A$ in $\S$ is equivalent to the pointing $1_A$.
Hence, $a_M \circ m_0^A \simeq \id_A$ as map of left $A$-modules and so also in~$\S$.

We now need to show that the other composite $m_0^A \circ a_M \simeq \id_M$,
which is a bit more involved.
The key idea is to identify this composite with the composition appearing in the zigzag identity~(\ref{eq M snake}),
which then implies the composite is equivalent to the identity.

The first step is to examine the following commutative diagram in~$\S$:
\[
\begin{tikzcd}
& M \otimes N \otimes \unit \arrow{r} \arrow{dd}{\id \otimes m_0} & M \otimes_B N \otimes \unit\arrow{r}{c} \arrow{dd}{\id \otimes m_0} & A \otimes \unit \arrow{dr}{m_0^A} \arrow{dd}{\id \otimes m_0} & \\
M \simeq M \otimes \unit \arrow{ur}{\id \otimes n_0} \arrow{dr}{\id \otimes n_0 \otimes m_0} & & & & M \\
& M \otimes N \otimes M \arrow{r}  & M \otimes_B N \otimes M \arrow{r}{c} 
& A \otimes M \arrow{ur}{\star} & 
\end{tikzcd}
\]
The top row is the composite $m_0^A \circ a_M \simeq \id_M$ decomposed into smaller constituents.
We have added copies of $\unit$ on the top row to make the downward arrows clearer;
they all amount to adjoining the pointing for $M$.
To see that the diagram commutes, 
note that each constituent triangle and square commutes.

The second step is to show the bottom row of the preceding diagram is equivalent to the zigzag identity.
Hence we examine the following commutative diagram in~$\S$:
\[
\begin{tikzcd}
& M \otimes N \otimes M \arrow{r} \arrow{dd} & M \otimes_B N \otimes M \arrow{r}{c}  \arrow{dd} & A \otimes M  \arrow{dd}{\star} \arrow{dr}{\star} & \\
M \simeq M \otimes \unit \arrow{ur}{\id \otimes n_0 \otimes m_0} \arrow{dr}{\id \otimes u \circ 1_B} & & & & M \\
& M \otimes N \otimes_A M \arrow{r}  & M \otimes_B N \otimes_A M \arrow{r}{c} 
& M \arrow[equal]{ur} &
\end{tikzcd}
\]
Here the first two vertical arrows map from the tensor product in $\S$ to the relative tensor product over $A$.
Direct inspection shows that the two squares and the rightmost triangle commute.

We now show that the leftmost triangle commutes.
Consider the commuting triangle
\[
\begin{tikzcd}
& \unit \arrow{dl}{1_B} \arrow{dr}{n_0 \otimes_A m_0} &  \\ 
B \arrow{rr}{u}  & & N \otimes_A M
\end{tikzcd}
\]
that arises from the pointings.
Take the tensor product $M \otimes-$ with this diagram to see that 
\[
\id_M \otimes n_0 \otimes_A m_0 \simeq \id_M \otimes u \circ 1_B
\]
as maps from $M$ to $M \otimes N \otimes_A M$.
But this implies the claim since $n_0 \otimes_A m_0$ is given by postcomposing $n_0 \otimes m_0$ with the map down to the relative tensor product over~$A$.
\end{proof}

We now turn to proving the proposition,
which boils down to verifying $A$ and $B$ are equivalent as unital algebras.
With that result in hand, Lemma \ref{lem: equiv in S} and its immediate consequences ensure the rest of the proposition.

\begin{proof}[Proof of Proposition \ref{thm: pointed adjoints are equivalences}]
We show the statement for $n_0: \unit \to N$ a left adjoint, as the right adjoint case is similar.
By Lemma \ref{lem: inverses from adjoints}, it suffices to show that the unit and counit maps are invertible.

We start by showing that $c$ is an equivalence and hence invertible.
Since the forgetful functor from $(A,A)$-bimodules to left $A$-modules detects equivalences (indeed, even down to $\S$),
it is enough to show $c$ is an equivalence as a map of left $A$-modules.
To see this assertion, observe that $M \otimes_B N \simeq M$ since $N \simeq B$ as a left $B$-module, 
and that this equivalence holds as left $A$-modules. 
Hence the map $M \xto{\id \otimes_B n_0} M \otimes_B N$ is an equivalence of left $A$-modules.
We have shown that the composite $a_M = c \circ \id \otimes_B n_0$ is the inverse to $m_0^A$ as maps of left $A$-modules, and hence $a_M$ is an equivalence.
The 2-out-of-3 property then implies that $c$ is an equivalence as well.

We now show $u$ is an equivalence.
It suffices to show that $u$ is an equivalence when viewed as a map of left $B$-modules,
since the forgetful functor from $(B,B)$-bimodules to left $B$-modules detects equivalences.
Under the free-forget adjunction between left $B$-modules and $\S$,
we have an equivalence of mapping spaces
\[
\Map_{_B\!\Mod}(B,N \otimes_A M) \simeq \Map_\S(\unit, N \otimes_A M).
\]
The pointing $n_0 \otimes_A m_0: \unit \to N \otimes_A M$ picks out a distinguished component of maps in $\S$, 
and hence also in maps in left $B$-modules.
Hence, as a map of $B$-modules, $u$ is determined by the pointing of $N \otimes_A M$.
But we know there is a distinguished equivalence $N \otimes_A M \simeq N$ as left $B$-modules,
as we have already constructed an explicit equivalence of left $A$-modules from $M$ to $A$.
In consequence, we have an explicit identification 
\[
\Map_{_B\!\Mod}(B,N \otimes_A M) \simeq \Map_{_B\!\Mod}(B,N).
\]
The map $n_0^B$ is the distinguished element in the second space associated to the pointing,
so we know $u$ and $n_0^B$ are equivalent under the identification.
Because we have shown $n_0^B$ is an equivalence of left $B$-modules, 
we are finished.
\end{proof}

\appendix
\section{The symmetric monoidal structure on \texorpdfstring{$\Alg_n(\S)$}{Algn(S)}}
\label{appx: sym mon}

We use freely notations from \cite{JFS} and~\cite{CSThesis}.

Recall that a symmetric monoidal $(\infty,N)$-category $\S$ can be chosen to be represented by a functor $[m]\mapsto \S[m]$ from $\mathrm{Fin}_*$ to the category of complete $n$-fold Segal spaces satisfying the Segal conditions.
 In \cite{CSThesis}, for any $\otimes$-sifted cocomplete $(\infty,1)$-category $\tilde{\S}$, 
 a symmetric monoidal structure on $\Alg_n(\tilde{\S})$ is constructed by giving an explicit functor $[m]\mapsto \Alg_n[m](\tilde{\S})$. 
 In brief, $k$-morphisms for $0\leq k\leq n$ in $\Alg_n[m](\tilde{\S})$ are given by an auxiliary data of some subintervals of $(0,1)$, together with $m$ factorization algebras $\F_1,\ldots, \F_m$ on $(0,1)^n$, all of which are constructible for the stratification determined by the auxiliary data.

\begin{prop}
\label{prop symm monoidal structure}
Suppose that $\S^\Box$ is \ensuremath\otimes-sifted-cocomplete.
The assignment
$$[m] \mapsto \Alg_n[m](\S^\Box_{\vec\bullet})_{\vec\bullet}$$
defines a symmetric monoidal structure on $\Alg_n(\S)$.
\end{prop}

\begin{proof}
The Segal condition in $[m]\in \mathrm{Fin}_*$ is just the symmetric monoidality of~$\Alg_n$.
The rest of the proof is completely analogous to the proof of  \cite[Theorem 8.5 (1)]{JFS}, but we reiterate it here for completeness of the argument.

It remains to show that for fixed $m$, the $n+N$-fold simplicial space
$$\Delta^n\times \Delta^N \ni (\vec k, \vec l) \longmapsto  \Alg_n[m](\S^\Box_{\vec l})_{\vec k}$$
satisfies the Segal condition separately in each variable.
Segality and completeness in $\vec k$ is just the Segal condition for $\Alg_n[m](\tilde{\S})$.

It remains to prove Segality and completeness in~$\vec l$. Since $\S^\Box_{\vec\bullet}$ is a complete $N$-fold Segal object by \cite[Remark 8.4]{JFS}, it suffices to prove that $\Alg_n[m](-)_{\vec k}$ preserves fiber products. The proof of \cite[Proposition 8.17]{JFS} applies to $\Alg_n[m]$ verbatim.
\end{proof}

\begin{remark}
We point out that the argument works for all three variants of the higher Morita category from \cite{JFS}, using the ``strong'', ``lax,'' or ``oplax'' versions of~$\S^\Box$.
\end{remark}

\bibliography{bib_dualizability}
\bibliographystyle{amsalpha}

\end{document}

%% file: prestuff.tex
\usepackage{mathrsfs}
\usepackage{mathtools}
\usepackage{amsmath, amscd}
\usepackage{amssymb, amsfonts}
\usepackage{amsthm}
\usepackage{mathabx}
\usepackage{enumitem}
\usepackage{standalone}

\usepackage{tikz}
\usepackage{tikz-cd}

\usetikzlibrary{matrix,arrows,decorations.pathmorphing,decorations.pathreplacing,decorations.markings, snakes}
\tikzset{
    dot/.style={circle,draw,fill,inner sep=1pt},
    arrow/.style={->,thick,shorten <=2pt,shorten >=2pt},
    twoarrow/.style={double,double distance=1.5pt,shorten <=9pt,shorten >=10pt,decoration={markings,mark=at position -8pt with {\arrow[scale=1.75]{>}}},preaction={decorate}},
    twoarrowlonger/.style={double,double distance=1.5pt,shorten <=5pt,shorten >=6pt,decoration={markings,mark=at position -4pt with {\arrow[scale=1.75]{>}}},preaction={decorate}},
    twoarrowshorter/.style={double,double distance=1.5pt,shorten <=13pt,shorten >=14pt,decoration={markings,mark=at position -12pt with {\arrow[scale=1.75]{>}}},preaction={decorate}},
    twoarrowshorthead/.style={double,double distance=1.5pt,shorten <=9pt,shorten >=20pt,decoration={markings,mark=at position -18pt with {\arrow[scale=1.75]{>}}},preaction={decorate}},
    threearrowpart1/.style={ thick,double,double distance=3pt,shorten <=9pt,shorten >=11pt},
    threearrowpart2/.style={ thick,shorten <=9pt,shorten >=10pt},
    threearrowpart3/.style={ shorten <=9pt,shorten >=10pt,decoration={markings,mark=at position -8pt with {\arrow[scale=3]{>}}},preaction={decorate}},
fourarrowpart1/.style={thick, double,double distance=4pt,shorten <=1pt,shorten >=2.75pt},
fourarrowpart2/.style={thick,  double,double distance=1pt,shorten <=1pt,shorten >=1.25pt,decoration={markings,mark=at position -.05pt with {\arrow[scale=3,ultra thin]{>}}},preaction={decorate} },
 lax/.style={double,double distance=1.5pt,shorten <=5pt,shorten >=6pt, decoration={markings,mark=at position -4pt with {\arrow[scale=0.75, thick]{>}}, mark=at position 0.5*\pgfdecoratedpathlength with {\node {
\begin{tikzpicture} \draw[-stealth, scale=0.3] (-0.8,-0.8) -- (1,1.3); \end{tikzpicture}
}}},postaction={decorate}},
 oplax/.style={double,double distance=1.5pt,shorten <=5pt,shorten >=6pt, decoration={markings,mark=at position -4pt with {\arrow[scale=0.75, thick]{>}}, mark=at position 0.5*\pgfdecoratedpathlength with {\node {
\begin{tikzpicture} \draw[stealth-, scale=0.3] (-0.8,-0.8) -- (1,1.3); \end{tikzpicture}
}}},postaction={decorate}}
}
\usetikzlibrary{patterns}

\newlength{\hatchspread}
\newlength{\hatchthickness}
\newlength{\hatchshift}
\newcommand{\hatchcolor}{}
\tikzset{hatchspread/.code={\setlength{\hatchspread}{#1}},
         hatchthickness/.code={\setlength{\hatchthickness}{#1}},
         hatchshift/.code={\setlength{\hatchshift}{#1}},% must be >= 0
         hatchcolor/.code={\renewcommand{\hatchcolor}{#1}}}
\tikzset{hatchspread=3pt,
         hatchthickness=0.4pt,
         hatchshift=0pt,% must be >= 0
         hatchcolor=black}
\pgfdeclarepatternformonly[\hatchspread,\hatchthickness,\hatchshift,\hatchcolor]% variables
   {custom north west lines}% name
   {\pgfqpoint{\dimexpr-2\hatchthickness}{\dimexpr-2\hatchthickness}}% lower left corner
   {\pgfqpoint{\dimexpr\hatchspread+2\hatchthickness}{\dimexpr\hatchspread+2\hatchthickness}}% upper right corner
   {\pgfqpoint{\dimexpr\hatchspread}{\dimexpr\hatchspread}}% tile size
   {% shape description
    \pgfsetlinewidth{\hatchthickness}
    \pgfpathmoveto{\pgfqpoint{0pt}{\dimexpr\hatchspread+\hatchshift}}
    \pgfpathlineto{\pgfqpoint{\dimexpr\hatchspread+0.15pt+\hatchshift}{-0.15pt}}
    \ifdim \hatchshift > 0pt
      \pgfpathmoveto{\pgfqpoint{0pt}{\hatchshift}}
      \pgfpathlineto{\pgfqpoint{\dimexpr0.15pt+\hatchshift}{-0.15pt}}
    \fi
    \pgfsetstrokecolor{\hatchcolor}
    \pgfusepath{stroke}
   }
\usepackage{enumitem}
\usepackage{bbm}
\usepackage{url}
\usepackage{todonotes}

\usepackage[parfill]{parskip}

\usepackage{comment}

\usepackage[colorlinks=true,linktocpage=true,linkcolor=blue,citecolor=red]{hyperref}

\DeclareFontFamily{OT1}{pzc}{}
\DeclareFontShape{OT1}{pzc}{m}{it}{<-> s * [1.10] pzcmi7t}{}
\DeclareMathAlphabet{\mathpzc}{OT1}{pzc}{m}{it}

\DeclareMathOperator{\Bord}{Bord}

\DeclareMathOperator{\Alg}{Alg}

\DeclareMathOperator{\opp}{{op}}
\DeclareMathOperator{\id}{id}
\DeclareMathOperator{\Mod}{Mod}

\def\xto{\xrightarrow}

\newcommand{\R}{{\mathbb{R}}}

\DeclareMathOperator{\Hom}{Hom}
\DeclareMathOperator{\Map}{Map}

\newcommand{\kk}{\mathbb{K}}

\newcommand{\unit}{\mathbbm{1}}

\DeclareMathOperator{\op}{op}
\DeclareMathOperator{\rev}{rev}

\renewcommand{\S}{\mathcal S}
\newcommand{\C}{\mathcal C}

\newcommand{\F}{\mathcal F}
\newcommand{\G}{\mathcal G}

\def\RR{\mathbb{R}}

\newcommand{\oul}[1]{{\overline{\underline{#1}}}}

\newcommand\myframe{}
\def\myframe{
\begin{scope}[xshift=\x, yshift=\y]
\begin{scope}[rotate=\k]
\draw[-stealth] (0,0) -- (0.25,0);
\draw[->]  (0,0) -- (0, 0.25);
\fill (0, 0) circle (0.007cm);
\end{scope}
\end{scope}
}
\newcommand\framing{}
\def\framing(#1, #2) (#3){
\def\k{#3}
\def\x{#1cm}
\def\y{#2cm}
\myframe
}

\newcommand\shiftframe{}
\def\shiftframe{
\begin{scope}[xshift=\x, yshift=\y]
\begin{scope}[rotate=\k]
\begin{scope}[shift={(-0.125,-0.125)}]
\draw[-stealth] (0,0) -- (0.25,0);
\draw[->]  (0,0) -- (0, 0.25);
\fill (0, 0) circle (0.007cm);
\end{scope}
\end{scope}
\end{scope}
}

\newcommand\shiftframing{}
\def\shiftframing(#1, #2) (#3){
\def\k{#3}
\def\x{#1cm}
\def\y{#2cm}
\shiftframe
}

\newcommand\vecs{}
\def\vecs{
\begin{scope}[xshift=\x, yshift=0.15cm]
\begin{scope}[rotate=\k]
\begin{scope}[shift={(-0.05,-0.05)}]
\draw[-stealth] (0,0) -- (0.1,0);
\draw[->]  (0,0) -- (0, 0.1);
\fill (0,0) circle (0.0011);
\end{scope}
\end{scope}
\end{scope}
}

\newcommand{\coord}[2]{
\def\k{#1}
\def\x{#2cm}
\vecs
}

\newcommand{\Aa}{\tt{A}}
\newcommand{\Bb}{\tt{B}}
\newcommand{\Rr}{\mathscr{R}}
\newcommand{\Ss}{\mathscr{S}}
\newcommand{\Tt}{\mathscr{T}}

\definecolor{bluebg}{HTML}{8372E2}
\definecolor{bluecirc}{HTML}{A295EC}
\definecolor{blue_ver}{HTML}{6A55DC}
\definecolor{bluedark}{HTML}{543DD2} 
\colorlet{blue_hor}{bluedark}

\definecolor{redbg}{HTML}{FD6E78}
\definecolor{redcirc}{HTML}{F1313F} %use this for red_ver
\colorlet{red_ver}{redcirc}
\definecolor{reddark}{HTML}{FC2F39} %use this for red_hor
\colorlet{red_hor}{reddark}
\definecolor{greenbg}{HTML}{89F169}
\definecolor{greencirc}{HTML}{53D22B}
\colorlet{greendark}{greencirc!70!black}

%% file: pics_tex/pic-morphism_small.tex
\begin{tikzpicture}[scale=1, baseline=(current bounding box.base)]
\fill[color=bluebg] (-1,0) -- (0,0) -- (0,2) -- (-1,2) -- cycle;
\fill[color=redbg] (0,0) -- (1,0) -- (1,2) -- (0,2) -- cycle;
%\draw[thin] (-1,0) -- (1,0) -- (1,2) -- (-1,2) -- cycle;
\draw[thick] (0,0) -- (0,2);
\draw[fill, color=bluecirc] (-0.7, 0.4) node[color=black] {$\Rr$} circle (0.2);
\draw[fill, color=redcirc] (0.6, 1.4) node[color=black] {$\Ss$} circle (0.2);
\draw (-0.08, 1) node {$\Aa$} circle (0.2);
\end{tikzpicture}

%% file: pics_tex/pic-2-morphism_small.tex
\begin{tikzpicture}[scale=1, baseline=(current bounding box.base)]
\fill[color=bluebg] (-1,0) -- (0,0) -- (0,2) -- (-1,2) -- cycle;
\fill[color=redbg] (0,0) -- (1,0) -- (1,2) -- (0,2) -- cycle;

\fill (0,1) circle (0.05);
\draw[thick] (0,0) -- (0,2);
\draw[fill, color=bluecirc] (-0.7, 0.4) node[color=black] {$\Rr$} circle (0.2);
\draw[fill, color=redcirc] (0.6, 1.4) node[color=black] {$\Ss$} circle (0.2);

\draw (0.11, 0.4) node {$\Aa$} circle (0.2);
\draw (-0.11, 1.6) node {$\Bb$} circle (0.2);

\draw (-0.17, 1) node {$M$} circle (0.24);
\end{tikzpicture}

%% file: pics_tex/pic-morphism.tex
\begin{tikzpicture}[scale=2, baseline=(current bounding box.base)]
\fill[color=bluebg] (-1,0) -- (0,0) -- (0,2) -- (-1,2) -- cycle;
\fill[color=redbg] (0,0) -- (1,0) -- (1,2) -- (0,2) -- cycle;
%\draw[thin] (-1,0) -- (1,0) -- (1,2) -- (-1,2) -- cycle;
\draw[thick] (0,0) -- (0,2);
\draw[fill, color=bluecirc] (-0.7, 0.4) node[color=black] {$\Rr$} circle (0.2);
\draw[fill, color=redcirc] (0.6, 1.4) node[color=black] {$\Ss$} circle (0.2);
\draw (-0.08, 1) node {$\Aa$} circle (0.2);
\end{tikzpicture}

%% file: pics_tex/pic-morphism-op.tex
$$
\begin{tikzpicture}[scale=2]
\fill[color=redbg] (-1,0) -- (0,0) -- (0,2) -- (-1,2) -- cycle;
\fill[color=bluebg] (0,0) -- (1,0) -- (1,2) -- (0,2) -- cycle;
\draw[thick] (0,0) -- (0,2);
\draw[fill, color=redcirc] (-0.7,1.4) node[color=black] {$\Ss$} circle (0.2);
\draw[fill, color=bluecirc] (0.6,0.6) node[color=black] {$\Rr$} circle (0.2);
\draw (0.13, 1) node {$\Aa^{op}$} circle (0.3);
\end{tikzpicture}
$$

%% file: pics_tex/pic-lemma-stripe.tex
$$
\begin{tikzpicture}[scale=2]
\draw[fill, bluebg] (-1,0) -- (-0.5,0) -- (-0.5, 0.5) .. controls (-0.5, 1) and (-0.2, 1) .. (-0.2, 2) -- (-1, 2) -- cycle
(0.5, 0) -- (0.5,0.5) .. controls (0.5, 1) and (0.3, 1.2) .. (0.3, 2) -- (0,2) .. controls (0, 1.5) and (0.1, 1.3) .. (0.1, 1) -- (0.1, 1) .. controls (0.1, 0.6) and (0, 0.5) .. (0,0.2) -- (0, 0) -- cycle;

\fill[fill=redbg] (0,2) .. controls (0, 1.5) and (0.1, 1.3) .. (0.1, 1) -- (0.1, 1) .. controls (0.1, 0.6) and (0, 0.5) .. (0,0.2) -- (0, 0) -- (-0.5,0) -- (-0.5, 0.5) .. controls (-0.5, 1) and (-0.2, 1) .. (-0.2, 2)
(0.5, 0) -- (0.5,0.5) .. controls (0.5, 1) and (0.3, 1.2) .. (0.3, 2) -- (1, 2) -- (1,0);
\draw[thick] (-0.5,0) -- (-0.5, 0.5) .. controls (-0.5, 1) and (-0.2, 1) .. (-0.2, 2)
(0.5, 0) -- (0.5,0.5) .. controls (0.5, 1) and (0.3, 1.2) .. (0.3, 2)
(0, 0) -- (0,0.2) .. controls (0, 0.5) and (0.1, 0.6) .. (0.1, 1) (0.1, 1) .. controls (0.1, 1.3) and (0, 1.5) .. (0,2);

\draw[] (-0.7, 0) node[anchor=north] {$b$}-- (-0.7, 1)
(0.7, 0) node[anchor=north] {$a$} -- (0.7, 1);
\draw[color=black] (-0.7, 1) -- (-0.7, 2)
(0.7, 1) -- (0.7, 2);

\end{tikzpicture}
$$

%% file: pics_tex/pic-lemma-cross.tex
$$
\begin{tikzpicture}[scale=2]
\fill[fill, bluebg] (-1,0) -- (-0.5,0) -- (-0.5, 0.5) .. controls (-0.5, 1) and (-0.2, 1) .. (-0.2, 2) -- (-1, 2) -- cycle
(0.5, 0) -- (0.5,0.5) .. controls (0.5, 1) and (0.3, 1.2) .. (0.3, 2) -- (1, 2) -- (1,0);

\fill[fill=redbg] (0.5, 0) -- (0.5,0.5) .. controls (0.5, 1) and (0.3, 1.2) .. (0.3, 2) -- (-0.2, 2) .. controls (-0.2, 1) and (-0.5, 1) .. (-0.5, 0.5) -- (-0.5,0);

\draw[thick] (-0.5,0) -- (-0.5, 0.5) .. controls (-0.5, 1) and (-0.2, 1) .. (-0.2, 2)
(0.5, 0) -- (0.5,0.5) .. controls (0.5, 1) and (0.3, 1.2) .. (0.3, 2);
\filldraw[thick, fill= bluebg]
(-0.3,0) arc (180:0:0.3 and 0.6)
(0, 1) circle (0.2);

\draw[dashed] (-0.7, 0.3) -- (0.7, 0.3);
\draw[dashed] (-1, 0.3) node[anchor=east, color=black] {$d$} -- (-0.7, 0.3)
(0.7, 0.3) -- (1, 0.3);
\draw[dashed] (-1, 1.7) node[anchor=east, color=black] {$c$} -- (1, 1.7);

\draw[] (-0.7, 0) node[anchor=north] {$b$}-- (-0.7, 1)
(0.7, 0) node[anchor=north] {$a$} -- (0.7, 1);
\draw[color=black] (-0.7, 1) -- (-0.7, 2)
(0.7, 1) -- (0.7, 2);

\end{tikzpicture}
$$

%% file: pics_tex/pic-twistobject.tex
%can change color of dashing here: undecided whether red or black...
\colorlet{dashed}{redcirc!85!black}

$$
\begin{tikzpicture}[scale=2]
\fill[fill = redbg] (-1,-1) -- (1,-1) -- (1,1) -- (-1, 1) -- cycle;
\draw[redcirc!70!black] (-1, 0.25) -- (1, 1);
\draw[redcirc!85!black] (-1, 0) -- (1, 0);
\draw[redcirc] (-1, -0.25) -- (1, -1);

\draw[thick, dashed, color=dashed] (-1,-0.5) -- (-0.5, -0.5) arc (-90:90:0.5) -- (-1,0.5);
%\framing(-0.9, -0.9) (0);
%\framing (0.7, -0.9) (0);
%\framing (-0.9, -0.125) (0);
%\framing (-0.9, 0.7) (0);
%\framing (-0.125, -0.9) (0);
%\framing (-0.125, -0.125) (0);

\draw[|->] (1.25, 0) -- node[anchor=south] {$\varphi$} (1.75, 0);

\begin{scope}[xshift=3cm]
\fill[fill = redbg] (-1,-1) -- (1,-1) -- (1,1) -- (-1, 1) -- cycle;
\draw[redcirc] (-0.5, -1) -- (-0.5, 1);
\draw[redcirc!85!black] (0, -1) -- (0, 1);
\draw[redcirc!70!black]  (0.5, -1) -- (0.5, 1);

\draw[thick, dashed, color=dashed] (-1, 0) -- (1, 0);
%\framing(-0.9, -0.9) (0);
%\framing (-0.9, -0.125) (0);
%\framing (-0.9, 0.7) (0);
%
%\shiftframing (-0.375, -0.775) (-45);
%\shiftframing (-0.375, 0) (-45);
%
%\framing (-0.125, -0.65) (-90);
%\framing (-0.125, 0.125) (-90);
%
%\shiftframing (0.375, -0.775) (-135)
%\framing (0.9, -0.65) (180);
\end{scope}
\end{tikzpicture}
$$

%% file: pics_tex/pic-twistframingobject.tex
%can change color of dashing here: undecided whether red or black...
\colorlet{dashed}{redcirc!85!black}

$$
\begin{tikzpicture}[scale=2]
\fill[fill = redbg] (-1,-1) -- (1,-1) -- (1,1) -- (-1, 1) -- cycle;
\draw[redcirc!70!black] (-1, 0.25) -- (1, 1);
\draw[redcirc!85!black] (-1, 0) -- (1, 0);
\draw[redcirc] (-1, -0.25) -- (1, -1);

\draw[thick, dashed, color=dashed] (-1,-0.5) -- (-0.5, -0.5) arc (-90:90:0.5) -- (-1,0.5);
\framing(-0.9, -0.9) (0);
\framing (0.7, -0.9) (0);
\framing (-0.9, -0.125) (0);
\framing (-0.9, 0.7) (0);
\framing (-0.125, -0.9) (0);
\framing (-0.125, -0.125) (0);

\draw[|->] (1.25, 0) -- node[anchor=south] {$\varphi$} (1.75, 0);

\begin{scope}[xshift=3cm]
\fill[fill = redbg] (-1,-1) -- (1,-1) -- (1,1) -- (-1, 1) -- cycle;
\draw[redcirc] (-0.5, -1) -- (-0.5, 1);
\draw[redcirc!85!black] (0, -1) -- (0, 1);
\draw[redcirc!70!black]  (0.5, -1) -- (0.5, 1);

\draw[thick, dashed, color=dashed] (-1, 0) -- (1, 0);
\framing(-0.9, -0.9) (0);
\framing (-0.9, -0.125) (0);
\framing (-0.9, 0.7) (0);

\shiftframing (-0.375, -0.775) (-45);
\shiftframing (-0.375, 0) (-45);

\framing (-0.125, -0.65) (-90);
\framing (-0.125, 0.125) (-90);

\shiftframing (0.375, -0.775) (-135)
\framing (0.9, -0.65) (180);
\end{scope}
\end{tikzpicture}
$$

%% file: pics_tex/pic-twistframing-collapse_equal.tex
$$
\begin{tikzpicture}[scale=6]

\begin{scope}[xshift=0.3cm]
\fill[redbg] (-0.25,0) -- (-0.25, 0.3)  -- (0.5, 0.3) -- (0.5,0) -- cycle;
\fill[redcirc] (-0.1,0) -- (-0.1, 0.3)  -- (0.35, 0.3) -- (0.35,0) -- cycle;

\coord{0}{-0.167}
%\coord{0}{0}
\coord{315}{0.}
\coord{270}{0.125}
\coord{225}{0.25}
\coord{180}{0.42}

%shaded region
\fill[pattern color= redcirc, pattern = crosshatch dots] (-0.1,0) -- (0.35,0) --  (0.125, -0.3) -- cycle;
\draw[|->] (0.125, -0.11) -- node[anchor=west] {$\varrho_*$} (0.125,-0.19);
\end{scope}

%second bottom strip
\begin{scope}[shift={(0.3, -0.6)}]
\fill[redbg] (-0.07,0) -- (-0.07, 0.3)  -- (0.32, 0.3) -- (0.32,0) -- cycle;
%\draw[thin] (-0.075,0) -- (-0.075, 0.3)  -- (0.325, 0.3) -- (0.325,0) -- cycle;

\draw[thick, redcirc] (0.125,0) -- (0.125, 0.3);
\draw[fill, color=redcirc] (0, 0.1) node[color=black] {$\scriptstyle \Rr$} circle (0.05);
\draw[fill, color=redcirc] (0.125, 0.15) node[color=black] {$ \Rr^{\,\circlearrowright}$} circle (0.07);
\draw[fill, color=redcirc] (0.25, 0.2) node[color=black] {$\scriptstyle \Rr$} circle (0.05);
\end{scope}

\end{tikzpicture}
$$

%% file: pics_tex/pic-twistframing-collapse_equal_left.tex
$$
\begin{tikzpicture}[scale=6]

\begin{scope}[xshift=0.3cm]
\fill[redbg] (-0.25,0) -- (-0.25, 0.3)  -- (0.5, 0.3) -- (0.5,0) -- cycle;
\fill[redcirc] (-0.1,0) -- (-0.1, 0.3)  -- (0.35, 0.3) -- (0.35,0) -- cycle;

\coord{0}{-0.167}
%\coord{0}{0}
\coord{45}{0.}
\coord{90}{0.125}
\coord{135}{0.25}
\coord{180}{0.42}

%shaded region
\fill[pattern color= redcirc, pattern = crosshatch dots] (-0.1,0) -- (0.35,0) --  (0.125, -0.3) -- cycle;
\draw[|->] (0.125, -0.11) -- node[anchor=west] {$\varrho_*$} (0.125,-0.19);
\end{scope}

%second bottom strip
\begin{scope}[shift={(0.3, -0.6)}]
\fill[redbg] (-0.07,0) -- (-0.07, 0.3)  -- (0.32, 0.3) -- (0.32,0) -- cycle;
%\draw[thin] (-0.075,0) -- (-0.075, 0.3)  -- (0.325, 0.3) -- (0.325,0) -- cycle;

\draw[thick, redcirc] (0.125,0) -- (0.125, 0.3);
\draw[fill, color=redcirc] (0, 0.1) node[color=black] {$\scriptstyle \Rr$} circle (0.05);
\draw[fill, color=redcirc] (0.125, 0.15) node[color=black] {$ \Rr^{\,\circlearrowleft}$} circle (0.07);
\draw[fill, color=redcirc] (0.25, 0.2) node[color=black] {$\scriptstyle \Rr$} circle (0.05);
\end{scope}

\end{tikzpicture}
$$

%% file: pics_tex/pic-bendright.tex
$$
\begin{tikzpicture}[scale=2]
\fill[bluebg] (-1,0) -- (-0.5,0) -- (-0.5, 0.5) arc (180:0:0.5) -- (0.5,0) -- (1,0) -- (1, 2) -- (-1, 2) -- cycle;
\fill[fill=redbg] (-0.5,0) -- (-0.5, 0.5) arc (180:0:0.5) -- (0.5,0) -- cycle;
\draw[thick] (-0.5,0) -- (-0.5,0) -- (-0.5, 0.5) arc (180:0:0.5) -- (0.5,0);

\draw[fill, color=bluecirc] (-0.7, 1.2) node[color=black] {$\Rr$} circle (0.2);
\draw[fill, color=redcirc] (0.2, 0.3) node[color=black] {$\Ss$} circle (0.2);
\draw (0, 1.08) node {$\Aa$} circle (0.2);
\end{tikzpicture}
$$

%% file: pics_tex/pic-bendright_collars.tex
$$
\begin{tikzpicture}[scale=2]
\draw[fill, bluebg] (-1,0) -- (-0.5,0) -- (-0.5, 0.5) arc (180:0:0.5) -- (0.5,0) -- (1,0) -- (1, 2) -- (-1, 2) -- cycle;
\fill[fill=redbg] (-0.5,0) -- (-0.5, 0.5) arc (180:0:0.5) -- (0.5,0) -- cycle;
\draw[thick] (-0.5,0) -- (-0.5,0) -- (-0.5, 0.5) arc (180:0:0.5) -- (0.5,0);

\draw[dashed] (-0.7, 0.3) -- (0.7, 0.3);
\draw[dashed, color=blue_hor] (-1, 0.3) node[anchor=east, color=black] {$d$} -- (-0.7, 0.3)
(0.7, 0.3) -- (1, 0.3);
\draw[dashed, color=blue_hor] (-1, 1.7) node[anchor=east, color=black] {$c$} -- (1, 1.7);
\draw[] (-0.7, 0) node[anchor=north] {$b$}-- (-0.7, 1)
(0.7, 0) node[anchor=north] {$a$} -- (0.7, 1);
\draw[color=blue_ver] (-0.7, 1) -- (-0.7, 2)
(0.7, 1) -- (0.7, 2);

\end{tikzpicture}
$$

%% file: pics_tex/pic-bendright_collapse.tex
$$
\begin{tikzpicture}[scale=2]
\fill[color=bluebg] (-1,0) -- (1,0) -- (1,2) -- (-1,2) -- cycle;
\draw[thick, color=blue_ver] (0,1) -- (0,2);
\fill[color=bluecirc] (0, 1.5) node[color=black] {$\Rr$} circle (0.2);
\fill[color=bluecirc] (0.7, 1.3) node[color=black] {$\Rr$} circle (0.2);
\draw[thick, dashed, color=blue_hor] (-1, 1) -- (1, 1);
\draw[fill] (0,1) circle (0.03);
\fill[color=bluecirc] (-0.7, 1) node[color=black] {$\Rr$} circle (0.2);
\draw[thick] (0,1) -- (0,0);
\draw[color=redcirc, pattern color=red, pattern=north east lines] (0.08, 0.3) node[color=black] {$\Bb$} circle (0.2);
\draw (0, 1) node[anchor=south west] {$\Aa$} circle (0.22);
\end{tikzpicture}
$$

%% file: pics_tex/pic-bendright-strip-composition.tex
$$
\begin{tikzpicture}[scale=4]
\begin{scope}[xshift=-0.05cm]
%first top strip
\draw[fill, bluebg] (-1,0) -- (-0.5,0) -- (-0.5, 0.3) -- (-1, 0.3) -- cycle;
\draw[fill, redbg] (-0.5,0) -- (-0.5, 0.3)  -- (0, 0.3) -- (0,0) -- cycle;
\draw[thick] (-0.5,0) -- (-0.5, 0.3);

\coord{0}{-0.9}
\coord{0}{-0.7}
\coord{0}{-0.5}
\coord{0}{-0.333}
\coord{0}{-0.167}
\coord{0}{-0.05}

\end{scope}

%second top strip
\begin{scope}[xshift=0.3cm]
\draw[fill, redbg] (-0.25,0) -- (-0.25, 0.3)  -- (0.5, 0.3) -- (0.5,0) -- cycle;

\coord{0}{-0.167}
\coord{0}{0}
\coord{300}{0.15}
\coord{240}{0.25}
\coord{180}{0.42}
\end{scope}

%third top strip
\begin{scope}[xshift=0.6cm]
\draw[fill, bluebg] (0.5, 0.3) -- (0.5,0) -- (0.7,0) -- (0.7,0.3) -- cycle;
\draw[fill, redbg] (0.3,0) -- (0.3, 0.3)  -- (0.5, 0.3) -- (0.5,0) -- cycle;
\draw[thick] (0.5, 0.3) -- (0.5,0);

\coord{180}{0.36}
\coord{180}{0.5}
\coord{180}{0.64}

\end{scope}

%last top strip
\begin{scope}[xshift=0.9cm]
\draw[fill, bluebg]
(0.5, 0.3) -- (0.5,0) -- (1,0) -- (1,0.3) -- cycle;

\coord{180}{0.56}
\coord{240}{0.7}
\coord{300}{0.85}
\coord{0}{0.94}

\end{scope}

\end{tikzpicture}
$$

%% file: pics_tex/pic-fold_morphism.tex
$$
\begin{tikzpicture}[scale=2]

\draw[white] (-0.1, -1.8) -- (2.1, -1.8) ;

\fill[bluebg]  (0.707, - 0.707) arc (-45:0:1) -- (2,1) arc (0 : -45 : 1) -- (0.707, - 0.707);

\fill[bluecirc!70]  (1,0) -- (0,-1) arc (-90:0:1);

\fill[redbg] (1,0) arc (0:90:1) -- (1,2) arc (90 : 0 : 1);

\foreach \t in {0,5, ..., 90}
	\draw[redcirc] ({cos(\t)},{sin(\t)}) -- ({1+cos(\t)},{1+sin(\t)});

\foreach \t in {-45, -40,..., 0}
	\draw[bluedark] ({cos(\t)},{sin(\t)}) -- ({1+cos(\t)},{1+sin(\t)});

\foreach \t in {-50, -55,..., -90}
	\draw[bluedark, very thin] ({cos(\t)},{sin(\t)}) -- ({1+cos(\t)},{1+sin(\t)});

%\draw (1, 0) -- (0,-1) arc (-90:90:1) -- (1,2) arc (90 : -45 : 1) -- (0.707, - 0.707);

%stratification
\draw[thick] ({cos(0)},{sin(0)}) -- ({1+cos(0)},{1+sin(0)});

\draw[-latex, gray, very thin] (0,-1.7) -- (2,-1.7);
\draw[-latex, gray, very thin] (0,-1.7) -- (1, -0.7);
\draw[-latex, gray, very thin] (0,-1.7) -- (0, -1.1) node[anchor=east] {\scriptsize $\mathbb{R}$};
\draw[gray] (1, -1.4) node {\scriptsize $(0,1)^2$};

\end{tikzpicture}
$$

%% file: pics_tex/pic-folded_morphism.tex
$$
\begin{tikzpicture}[scale=2]
\fill[color=white] (-2, 0) -- (2,0) -- (2,2) -- (-2, 2);

\fill[color=bluebg] (-1,0) -- (0,0) -- (0,2) -- (-1,2) -- cycle;
\fill[pattern=custom north west lines, hatchspread=6pt,hatchthickness=2pt,hatchcolor=redbg] (-1,-0.05) rectangle  (0,2);
\filldraw[color=white] (-1,0) rectangle (0,-0.05);
%\draw[fill, color=blue!60] (-0.7, 0.4) circle (0.2);
\draw[] (-0.7, 0.4) circle (0.2);

\draw[gray, thin, dotted] (0,0) -- (1,0) -- (1,2) -- (0,2);
\draw[thick] (0,0) -- (0,2);

\draw[<-|] (-1.2, 0.5) node[anchor=east, color=black] {$\Ss^{\rev}\otimes \Rr$} -- (-0.7, 0.4);
\draw (0.6, 1.4) node[color=black] {$\unit$} circle (0.2);
\draw (0, 1) circle (0.2)
(0.1, 1) node {$\Aa$};
\draw (1.1,0) node {$.$};
\end{tikzpicture}
$$

%% file: pics_tex/pic-morphism2.tex
\begin{tikzpicture}[scale=1.5, baseline=(current bounding box.base)]
\fill[color=redbg] (-1,0) -- (0,0) -- (0,2) -- (-1,2) -- cycle;
\fill[color=greenbg] (0,0) -- (1,0) -- (1,2) -- (0,2) -- cycle;
%\draw[thin] (-1,0) -- (1,0) -- (1,2) -- (-1,2) -- cycle;
\draw[thick] (0,0) -- (0,2);
\draw[fill, color=redcirc] (-0.7, 0.4) node[color=black] {$\Ss$} circle (0.2);
\draw[fill, color=greencirc] (0.6, 1.4) node[color=black] {$\Tt$} circle (0.2);
\draw (-0.08, 1) node {$\Bb$} circle (0.2);
\end{tikzpicture}

%% file: pics_tex/pic-foldright.tex
$$
\begin{tikzpicture}[scale=1]

% coevaluation

\fill[redbg] ({cos(180)},{2+sin(180)}) -- ({1+cos(180)},{3+sin(180)}) arc(180:270:1) -- (0,1) arc(270:90:1);

\foreach \t in {185, 190,..., 270}
	\draw[redcirc] ({cos(\t)},{2+sin(\t)}) -- ({1+cos(\t)},{3+sin(\t)});

\fill[greenbg!50] (0,3) arc (90:180:1);

\foreach \t in {135, 140,..., 175}
	\draw[greendark, very thin] ({cos(\t)},{2+sin(\t)}) -- ({1+cos(\t)},{3+sin(\t)});
%stratification
\draw[thick] ({cos(180)},{2+sin(180)}) -- ({1+cos(180)},{3+sin(180)});

\fill[fill=white] (-0.707,  2.707) arc (135:90:1) -- (1,4) arc (90: 135:1) -- (-0.707,  2.707);

\fill[greenbg] (-0.707,  2.707) arc (135:90:1) -- (1,4) arc (90: 135:1) -- (-0.707,  2.707);

\foreach \t in {130, 125,..., 95}
	\draw[greendark] ({cos(\t)},{2+ sin(\t)}) -- ({1+cos(\t)},{3+sin(\t)});

\end{tikzpicture}
$$

%% file: pics_tex/pic-folded_right_morphism.tex
$$
\begin{tikzpicture}[scale=1.5, xscale=-1]
%\fill[color=white] (-2, 0) -- (2,0) -- (2,2) -- (-2, 2);

\fill[color=greenbg] (-1,0) -- (0,0) -- (0,2) -- (-1,2) -- cycle;
\fill[pattern=custom north west lines, hatchspread=6pt,hatchthickness=2pt,hatchcolor=redbg] (-1,-0.05) rectangle  (0,2);
\filldraw[color=white] (-1,0) rectangle (0,-0.05);
%\draw[fill, color=blue!60] (-0.7, 0.4) circle (0.2);
\draw[] (-0.7, 0.4) circle (0.2);

\draw[gray, thin, dotted] (0,0) -- (1,0) -- (1,2) -- (0,2);
\draw[thick] (0,0) -- (0,2);

\draw[<-|] (-1.2, 0.5) node[anchor=west, color=black] {$\Tt\otimes \Ss^{\rev}$} -- (-0.7, 0.4);
\draw (0.6, 1.4) node[color=black] {$\unit$} circle (0.2);
\draw (0, 1) circle (0.2)
(0.1, 1) node {$\Bb$};
%\draw (-1.1,0) node {$.$};
\end{tikzpicture}
$$

%% file: pics_tex/pic-snake_composition.tex
$$
\begin{tikzpicture}[scale=1]

\fill[bluebg]  (0.707, - 0.707) arc (-45:0:1) -- (2,1) arc (0 : -45 : 1) -- (0.707, - 0.707);

\fill[bluecirc!70]  (1,0) -- (0,-1) arc (-90:0:1);

\fill[redbg] (1,0) arc (0:90:1) -- (1,2) arc (90 : 0 : 1);

\foreach \t in {0,5, ..., 90}
	\draw[redcirc] ({cos(\t)},{sin(\t)}) -- ({1+cos(\t)},{1+sin(\t)});

\foreach \t in {-45, -40,..., 0}
	\draw[bluedark] ({cos(\t)},{sin(\t)}) -- ({1+cos(\t)},{1+sin(\t)});

\foreach \t in {-50, -55,..., -90}
	\draw[bluedark, very thin] ({cos(\t)},{sin(\t)}) -- ({1+cos(\t)},{1+sin(\t)});

%\draw (1, 0) -- (0,-1) arc (-90:90:1) -- (1,2) arc (90 : -45 : 1) -- (0.707, - 0.707);

%stratification
\draw[thick] ({cos(0)},{sin(0)}) -- ({1+cos(0)},{1+sin(0)});

% coevaluation

\fill[redbg] ({cos(180)},{2+sin(180)}) -- ({1+cos(180)},{3+sin(180)}) arc(180:270:1) -- (0,1) arc(270:90:1);

\foreach \t in {185, 190,..., 270}
	\draw[redcirc] ({cos(\t)},{2+sin(\t)}) -- ({1+cos(\t)},{3+sin(\t)});

\fill[greenbg!50] (0,3) arc (90:180:1);

\foreach \t in {135, 140,..., 175}
	\draw[greendark, very thin] ({cos(\t)},{2+sin(\t)}) -- ({1+cos(\t)},{3+sin(\t)});
%stratification
\draw[thick] ({cos(180)},{2+sin(180)}) -- ({1+cos(180)},{3+sin(180)});

\fill[fill=white] (-0.707,  2.707) arc (135:90:1) -- (1,4) arc (90: 135:1) -- (-0.707,  2.707);

\fill[greenbg] (-0.707,  2.707) arc (135:90:1) -- (1,4) arc (90: 135:1) -- (-0.707,  2.707);

\foreach \t in {130, 125,..., 95}
	\draw[greendark] ({cos(\t)},{2+ sin(\t)}) -- ({1+cos(\t)},{3+sin(\t)});

% identities
\fill[bluecirc!70] (0, -1) -- (-2, -1) -- (-1,0) -- (1,0);
\foreach \t in {-2, -1.9, -1.8,...,0}
	\draw[bluedark, very thin] (\t, -1) -- ({1+\t}, 0);

\begin{scope}[shift={(2,4)}]
\fill[greenbg] (-2, -1) -- (0, -1) -- (1,0) -- (-1,0);

\foreach \t in {-2, -1.9, -1.8,...,0}
	\draw[greendark] (\t, -1) -- ({1+\t}, 0);

\end{scope}

%\draw[->, snake=snake,line before snake=3pt] (3.5, 3.5) -- (4.5, 3.5);
%\draw[->, snake=snake,line before snake=3pt] (-2.5, -0.5) -- (-3.5, -0.5);

\end{tikzpicture}
$$

%% file: pics_tex/pic-snake_pull.tex
$$
\begin{tikzpicture}[scale=0.7]

\fill[bluebg]  (0.707, - 0.707) arc (-45:0:1) -- (2,1) arc (0 : -45 : 1) -- (0.707, - 0.707);

\fill[bluecirc!70]  (1,0) -- (0,-1) arc (-90:0:1);

\fill[redbg] (1,0) arc (0:90:1) -- (1,2) arc (90 : 0 : 1);

\foreach \t in {0,5, ..., 90}
	\draw[redcirc] ({cos(\t)},{sin(\t)}) -- ({1+cos(\t)},{1+sin(\t)});

\foreach \t in {-45, -40,..., 0}
	\draw[bluedark] ({cos(\t)},{sin(\t)}) -- ({1+cos(\t)},{1+sin(\t)});

\foreach \t in {-50, -55,..., -90}
	\draw[bluedark, very thin] ({cos(\t)},{sin(\t)}) -- ({1+cos(\t)},{1+sin(\t)});

%\draw (1, 0) -- (0,-1) arc (-90:90:1) -- (1,2) arc (90 : -45 : 1) -- (0.707, - 0.707);

%stratification
\draw[thick] ({cos(0)},{sin(0)}) -- ({1+cos(0)},{1+sin(0)});

% coevaluation

\fill[redbg] ({cos(180)},{2+sin(180)}) -- ({1+cos(180)},{3+sin(180)}) arc(180:270:1) -- (0,1) arc(270:90:1);

\foreach \t in {185, 190,..., 270}
	\draw[redcirc] ({cos(\t)},{2+sin(\t)}) -- ({1+cos(\t)},{3+sin(\t)});

\fill[greenbg!50] (0,3) arc (90:180:1);

\foreach \t in {135, 140,..., 175}
	\draw[greendark, very thin] ({cos(\t)},{2+sin(\t)}) -- ({1+cos(\t)},{3+sin(\t)});
%stratification
\draw[thick] ({cos(180)},{2+sin(180)}) -- ({1+cos(180)},{3+sin(180)});

\fill[fill=white] (-0.707,  2.707) arc (135:90:1) -- (1,4) arc (90: 135:1) -- (-0.707,  2.707);

\fill[greenbg] (-0.707,  2.707) arc (135:90:1) -- (1,4) arc (90: 135:1) -- (-0.707,  2.707);

\foreach \t in {130, 125,..., 95}
	\draw[greendark] ({cos(\t)},{2+ sin(\t)}) -- ({1+cos(\t)},{3+sin(\t)});

% identities
\fill[bluecirc!70] (0, -1) -- (-2, -1) -- (-1,0) -- (1,0);
\foreach \t in {-2, -1.9, -1.8,...,0}
	\draw[bluedark, very thin] (\t, -1) -- ({1+\t}, 0);

\begin{scope}[shift={(2,4)}]
\fill[greenbg] (-2, -1) -- (0, -1) -- (1,0) -- (-1,0);

\foreach \t in {-2, -1.9, -1.8,...,0}
	\draw[greendark] (\t, -1) -- ({1+\t}, 0);

\end{scope}

\draw[->, snake=snake,line before snake=1.5pt] (3.5, 3.5) -- (4.5, 3.5);
\draw[->, snake=snake,line before snake=1.5pt] (-2.5, -0.5) -- (-3.5, -0.5);

\draw (5, 1.5) node {$\simeq$};

\begin{scope}[shift={(8, 2)}]
% identities
\fill[bluecirc!70] (0, -1) -- (-2, -1) -- (-1,0) -- (1,0);
\foreach \t in {-2, -1.9, -1.8,...,0}
	\draw[bluedark, very thin] (\t, -1) -- ({1+\t}, 0);

%\draw (0, -1) -- (-2, -1) -- (-1,0) -- (1,0);

\begin{scope}[xshift=4cm]
\fill[greenbg] (-2, -1) -- (0, -1) -- (1,0) -- (-1,0);

\foreach \t in {-2, -1.9, -1.8,...,0}
	\draw[greendark] (\t, -1) -- ({1+\t}, 0);
\end{scope}

\begin{scope}[xshift=2cm]
\fill[redbg] (-2, -1) -- (0, -1) -- (1,0) -- (-1,0);

\foreach \t in {-2, -1.9, -1.8,...,0}
	\draw[redcirc] (\t, -1) -- ({1+\t}, 0);

\draw[thick] (-2, -1) -- (-1,0);
\draw[thick] (0, -1) -- (01,0);
\end{scope}
\end{scope}

\end{tikzpicture}
$$

%% file: pics_tex/pic-morphism-adjoint.tex
$$
\begin{tikzpicture}[scale=4]
\fill[redbg] (-1,0) -- (-0.5,0) -- (-0.5, 0.5) -- (-1, 0.5) -- cycle;
\fill[bluebg] (-0.5,0) -- (-0.5, 0.5)  -- (0, 0.5) -- (0,0) -- cycle;
%\draw[thin] (-1,0) -- (0,0) -- (0,0.5) -- (-1,0.5) -- cycle;
\draw[thick] (-0.5,0) -- (-0.5, 0.5);

\draw (-0.44, 0.1) node {$\,\,\Aa^{op}$} circle (0.11 and 0.09);

\begin{scope}[shift={(-1, 0.1)}]
\coord{0}{0.07}
\coord{300}{0.167}
\coord{240}{0.333}
\coord{180}{0.5}
\coord{240}{0.667}
\coord{300}{0.833}
\coord{0}{0.94}
\end{scope}
\end{tikzpicture}
$$

%% file: pics_tex/pic-bendright_framing.tex
$$
\begin{tikzpicture}[scale=2]
\fill[bluebg] (-1,0) -- (-0.5,0) -- (-0.5, 0.5) arc (180:0:0.5) -- (0.5,0) -- (1,0) -- (1, 2) -- (-1, 2) -- cycle;
\fill[fill=redbg, thick] (-0.5,0) -- (-0.5, 0.5) arc (180:0:0.5) -- (0.5,0) -- cycle;
\draw[thick] (-0.5,0) -- (-0.5, 0.5) arc (180:0:0.5) -- (0.5,0);

\draw[fill, color=bluecirc] (-0.7, 1.2) node[color=black] {$\Rr$} circle (0.2);
\draw[fill, color=redcirc] (0.2, 0.3) node[color=black] {$\Ss$} circle (0.2);
\draw (0.04, 1.08) node {$\Aa$} circle (0.2);

\framing(-0.5, 0.2) (0);
\framing (-0.35, 0.85) (-45);
\framing (0.433, 0.75) (-150);
\end{tikzpicture}
$$

%% file: pics_tex/pic-strip_morphism_adjoint.tex
$$
\begin{tikzpicture}[scale=4]
\draw[fill, bluebg] (-1,0) -- (-0.5,0) -- (-0.5, 0.3) -- (-1, 0.3) -- cycle
(0.5, 0.3) -- (0.5,0) -- (1,0) -- (1,0.3) -- cycle;
\draw[fill, redbg] (-0.5,0) -- (-0.5, 0.3)  -- (0.5, 0.3) -- (0.5,0) -- cycle;
%\draw[thin] (-1,0) -- (1,0) -- (1,0.3) -- (-1,0.3) -- cycle;
\draw[thick] (-0.5,0) -- (-0.5, 0.3)
(0.5, 0.3) -- (0.5,0);

\coord{0}{-0.9}
\coord{0}{-0.7}
\coord{0}{-0.5}
\coord{0}{-0.333}
\coord{0}{-0.167}
\coord{0}{0}
\coord{300}{0.167}
\coord{240}{0.333}
\coord{180}{0.5}
\coord{240}{0.667}
\coord{300}{0.833}
\coord{0}{0.95}
\end{tikzpicture}
$$

%% file: pics_tex/pic-bendleft.tex
$$
\begin{tikzpicture}[scale=2]
\fill[redbg] (-1,0) -- (-0.5,0) -- (-0.5, -0.3) arc (-180:0:0.5) -- (0.5,0) -- (1,0) -- (1,-1.5) -- (-1,-1.5) -- cycle;
\fill[bluebg] (-0.5,0) -- (-0.5, -0.3) arc (-180:0:0.5) -- (0.5,0) -- cycle;
\draw[thick] (-0.5,0) -- (-0.5, -0.3) arc (-180:0:0.5) -- (0.5,0);

\fill[color=redcirc] (-0.7, -1) node[color=black] {$\Ss$} circle (0.2);
\fill[color=bluecirc] (0.2, -0.3) node[color=black] {$\Rr$} circle (0.2);
\draw (0, -0.88) node {$\Aa$} circle (0.2);

\framing(0.5, -0.25) (0);
\framing (0.35, -0.65) (-45);
\framing (-0.433, -0.55) (-150);
\end{tikzpicture}
$$

%% file: pics_tex/pic-strip_morphism_adjoint2.tex
$$
\begin{tikzpicture}[scale=4]
\draw[fill, redbg] (-1,0) -- (-0.5,0) -- (-0.5, 0.3) -- (-1, 0.3) -- cycle
(0.5, 0.3) -- (0.5,0) -- (1,0) -- (1,0.3) -- cycle;
\draw[fill, bluebg] (-0.5,0) -- (-0.5, 0.3)  -- (0.5, 0.3) -- (0.5,0) -- cycle;
%\draw[thin] (-1,0) -- (1,0) -- (1,0.3) -- (-1,0.3) -- cycle;
\draw[thick] (-0.5,0) -- (-0.5, 0.3)
(0.5, 0.3) -- (0.5,0);

\begin{scope}[xshift=-1cm]
\coord{0}{0.06}
\coord{300}{0.167}
\coord{240}{0.333}
\coord{180}{0.5}
\coord{240}{0.667}
\coord{300}{0.833}
\end{scope}
\coord{0}{0}
\coord{0}{0.167}
\coord{0}{0.333}
\coord{0}{0.5}
\coord{0}{0.667}
\coord{0}{0.833}
\coord{0}{0.95}
\end{tikzpicture}
$$

%% file: pics_tex/pic-bendleft_collars.tex
$$
\begin{tikzpicture}[scale=2]
\fill[redbg] (-1,0) -- (-0.5,0) -- (-0.5, -0.5) arc (-180:0:0.5) -- (0.5,0) -- (1,0) -- (1,-2) -- (-1,-2) -- cycle;
\fill[bluebg] (-0.5,0) -- (-0.5, -0.5) arc (-180:0:0.5) -- (0.5,0) -- cycle;

\draw[thick] (-0.5,0) -- (-0.5, -0.5) arc (-180:0:0.5) -- (0.5,0);

\draw[red_hor, dashed] (-1, -0.3) -- (-0.7, -0.3)
(0.7,-0.3) -- (1, -0.3);
\draw[dashed] (-0.7, -0.3) -- (0.7, -0.3);
\draw[red_hor, dashed] (-1, -1.7) -- (1, -1.7);

\draw[] (-0.7, 0) -- (-0.7, -1);
\draw[red_ver] (-0.7, -1) -- (-0.7, -2);
\draw[] (0.7, 0) -- (0.7, -1);
\draw[red_ver] (0.7, -1) -- (0.7, -2);

\end{tikzpicture}
$$

%% file: pics_tex/pic-adjunction1.tex
$$\begin{tikzpicture}[scale=2]
 extra identity squares
\fill[fill=redbg] (-0.5,0) -- (0, 0) -- (0,-1.5) -- (-0.5, -1.5) -- cycle
(1.5, 0) -- (2,0) -- (2,1.5) -- (1.5,1.5);
\fill[color=black, fill=bluebg] (-0.5,0) -- (-1, 0) -- (-1,-1.5) -- (-0.5, -1.5) -- cycle
(1,0) -- (1.5, 0) -- (1.5, 1.5) -- (1,1.5) -- cycle;
\draw[thick] (1.5, 0) -- (1.5, 1.5)
(-0.5, 0) -- (-0.5,-1.5);

\fill[bluebg] (-1,0) -- (-0.5,0) -- (-0.5, 0.3) arc (180:0:0.5) -- (0.5,0) -- (1,0) -- (1,1.5) -- (-1,1.5) -- cycle;
\fill[redbg] (-0.5,0) -- (-0.5, 0.3) arc (180:0:0.5) -- (0.5,0) -- cycle;
\draw[thin] (-1,0) -- (1,0) -- (1,1.5); % -- (-1,1.5) -- cycle;
\draw[thick] (-0.5,0) -- (-0.5, 0.3) arc (180:0:0.5) -- (0.5,0);
\fill[color=bluecirc] (-0.7, 1) node[color=black] {$\Rr$} circle (0.2);
\fill[color=redcirc] (0.2, 0.3) node[color=black] {$\Ss$} circle (0.2);
\draw (0.04, 0.88) node {$\Aa$} circle (0.2);

frames upper left and lower left
\framing (-0.5, -0.8) (0);
\framing (-0.35, 0.65) (-45);

\begin{scope}[shift={(1,0)}]
\fill[redbg] (-1,0) -- (-0.5,0) -- (-0.5, -0.3) arc (-180:0:0.5) -- (0.5,0) -- (1,0) -- (1,-1.5) -- (-1,-1.5) -- cycle;
\fill[bluebg] (-0.5,0) -- (-0.5, -0.3) arc (-180:0:0.5) -- (0.5,0) -- cycle;
\draw[thin] (-1,-1.5) -- (-1,0) -- (1,0); % -- (1,-1.5) --  cycle;
\draw[thick] (-0.5,0) -- (-0.5, -0.3) arc (-180:0:0.5) -- (0.5,0);

\fill[color=redcirc] (-0.7, -1) node[color=black] {$\Ss$} circle (0.2);
\fill[color=bluecirc] (0.2, -0.3) node[color=black] {$\Rr$} circle (0.2);
\draw (0, -0.88) node {$\Aa$} circle (0.2);

frames lower right and upper right
\framing (0.5, 0.35) (0);
\framing (0.35, -0.65) (-45);
\end{scope}

\draw[|->] (2.25,0) -- node[anchor=south] {$\sigma$} (2.75,0);

\begin{scope}[shift={(4,-1)}]
\fill[color=bluebg] (-1,0) -- (0,0) -- (0,2) -- (-1,2) -- cycle;
\fill[color=redbg] (0,0) -- (1,0) -- (1,2) -- (0,2) -- cycle;
%\draw[thin] (-1,0) -- (1,0) -- (1,2) -- (-1,2) -- cycle;
\draw[thick] (0,0) -- (0,2);
\fill[color=bluecirc] (-0.7, 0.4) node[color=black] {$\Rr$} circle (0.2);
\fill[color=redcirc] (0.6, 1.4) node[color=black] {$\Ss$} circle (0.2);
\draw (-0.08, 1) node {$\Aa$} circle (0.2);
%\draw[fill, color=bluebg] (-1,0) -- (0,0) -- (0,1) -- (-1,1) -- cycle;
%\draw[fill, color=redbg] (0,0) -- (1,0) -- (1,1) -- (0,1) -- cycle;
%\draw[thin] (-1,0) -- (1,0) -- (1,1) -- (-1,1) -- cycle;
%\draw[thick] (0,0) -- (0,1);
%\draw[fill, color=bluecirc] (-0.7, 0.4) node[color=black] {$A$} circle (0.2);
%\draw[fill, color=redcirc] (0.6, 0.6) node[color=black] {$B$} circle (0.2);
%\draw (0, 0.5) node {$M$} circle (0.2);
\end{scope}

\end{tikzpicture}
$$

%% file: pics_tex/pic-adjunction2.tex
$$\begin{tikzpicture}[scale=2]
% extra identity squares
\fill[fill=redbg] (0.5,0) -- (0, 0) -- (0,-1.5) -- (0.5, -1.5) -- cycle
(-1.5, 0) -- (-2,0) -- (-2,1.5) -- (-1.5,1.5);
\fill[color=black, fill=bluebg] (0.5,0) -- (1, 0) -- (1,-1.5) -- (0.5, -1.5) -- cycle
(-1,0) -- (-1.5, 0) -- (-1.5, 1.5) -- (-1,1.5) -- cycle;
\draw[thick] (0.5, 0) -- (0.5,-1.5)
(-1.5,0) -- (-1.5,1.5);

% upper right fill
\fill[bluebg] (-1,0) -- (-0.5,0) -- (-0.5, 0.3) arc (180:0:0.5) -- (0.5,0) -- (1,0) -- (1,1.5) -- (-1,1.5) -- cycle;
\fill[redbg] (-0.5,0) -- (-0.5, 0.3) arc (180:0:0.5) -- (0.5,0) -- cycle;
\fill[color=bluecirc] (-0.7, 1) node[color=black] {$\Rr$} circle (0.2);
\fill[color=redcirc] (-0.1, 0.3) node[color=black] {$\Ss$} circle (0.2);
\draw (0, 0.88) node {$\Aa$} circle (0.2);

%frames upper right and lower right
\framing (0.5, -0.8) (180);
\framing (0.433, 0.55) (-150);

%lines between composed guys
\draw[thin] (-1,1.5) -- (-1,0) -- (1,0);
\draw[thick] (-0.5,0) -- (-0.5, 0.3) arc (180:0:0.5) -- (0.5,0);

\begin{scope}[shift={(-1,0)}]
\fill[redbg] (-1,0) -- (-0.5,0) -- (-0.5, -0.3) arc (-180:0:0.5) -- (0.5,0) -- (1,0) -- (1,-1.5) -- (-1,-1.5) -- cycle;
\fill[bluebg] (-0.5,0) -- (-0.5, -0.3) arc (-180:0:0.5) -- (0.5,0) -- cycle;
\draw[thin] (-1,0) -- (1,0) -- (1,-1.5);
\draw[thick] (-0.5,0) -- (-0.5, -0.3) arc (-180:0:0.5) -- (0.5,0);

\fill[color=redcirc] (-0.7, -1) node[color=black] {$\Ss$} circle (0.2);
\fill[color=bluecirc] (0.2, -0.3) node[color=black] {$\Rr$} circle (0.2);
\draw (0, -0.88) node {$\Aa$} circle (0.2);

%frames lower right and upper right
\framing (-0.5, 0.75) (180);
\framing (-0.433, -0.55) (-150);
\end{scope}

% arrow
\draw[|->] (1.25,0) -- node[anchor=south] {$\varsigma$} (1.75,0);

% right half
\begin{scope}[shift={(3,-1)}]
\draw[fill, color=redbg] (-1,0) -- (0,0) -- (0,2) -- (-1,2) -- cycle;
\draw[fill, color=bluebg] (0,0) -- (1,0) -- (1,2) -- (0,2) -- cycle;
\draw[thick] (0,0) -- (0,2);
\draw[fill, color=redcirc] (-0.7,1.4) node[color=black] {$\Ss$} circle (0.2);
\draw[fill, color=bluecirc] (0.6,0.6) node[color=black] {$\Rr$} circle (0.2);
\draw (0.13, 1) node {$\Aa^{op}$} circle (0.25);

\end{scope}
\end{tikzpicture}
$$

%% file: pics_tex/pic-morphism-rightadjoint.tex
$$
\begin{tikzpicture}[scale=4]
\fill[redbg] (-1,0) -- (-0.5,0) -- (-0.5, 0.5) -- (-1, 0.5) -- cycle;
\fill[bluebg] (-0.5,0) -- (-0.5, 0.5)  -- (0, 0.5) -- (0,0) -- cycle;
%\draw[thin] (-1,0) -- (0,0) -- (0,0.5) -- (-1,0.5) -- cycle;
\draw[thick] (-0.5,0) -- (-0.5, 0.5);

\begin{scope}[shift={(-1, 0.1)}]
\coord{0}{0.07}
\coord{-300}{0.18}
\coord{-240}{0.333}
\coord{-180}{0.5}
\coord{-240}{0.6}
\coord{-300}{0.8}
\coord{0}{0.94}
\end{scope}
\end{tikzpicture}
$$